\documentclass[]{amsart}
\usepackage{amssymb,amscd}
 
\usepackage{amsmath}
\usepackage{amsxtra}
\usepackage{amsfonts} 
\usepackage{amsthm}
\usepackage{enumerate}
\usepackage{psfrag}
\usepackage{microtype}
\usepackage{todonotes}
\usepackage[pdftitle={Lyapunov spectrum of
ball quotients with applications
to commensurability questions},
	    pdfauthor={Andre Kappes and Martin Moeller},
	    linktocpage=true
	    ]{hyperref}

\newlength{\halfbls}

\usepackage[all]{xy}
\newtheorem{theorem}{Theorem}[section]
\newtheorem{prop}[theorem]{Proposition} 
\newtheorem{cor}[theorem]{Corollary}
\newtheorem{rem}[theorem]{Remark}
\newtheorem{exam}[theorem]{Example}

\newtheorem{lemma}[theorem]{Lemma}

\newtheorem{defn}[theorem]{Definition}

\theoremstyle{remark}
\newtheorem*{remark}{Remark}

\newcommand{\ord}{\mathrm{ord}}

\newcommand{\BB}{\mathbb{B}}

\newcommand{\PP}{\mathbb{P}}

\newcommand{\HH}{\mathbb{H}}

\newcommand{\LL}{\mathbb{L}}

\newcommand{\UU}{\mathbb{U}}
\newcommand{\WW}{\mathbb{W}}

\newcommand{\QQ}{\mathbb{Q}}

\newcommand{\VV}{\mathbb{V}}

\newcommand{\RR}{\mathbb{R}}
\newcommand{\CC}{\mathbb{C}}
\newcommand{\ZZ}{\mathbb{Z}}
\newcommand{\NN}{\mathbb{N}}

\newcommand{\Gal}{{\rm Gal}}

\newcommand{\SL}{\mathrm{SL}}

\newcommand{\CH}{\mathrm{CH}}
\newcommand{\GL}{\mathrm{GL}}

\newcommand{\PGL}{\mathrm{PGL}}
\newcommand{\SU}{\mathrm{SU}}
\newcommand{\PU}{\mathrm{PU}}

\newcommand{\id}{\mathrm{id}}

\newcommand{\cB}{\mathcal{B}}
\newcommand{\cE}{\mathcal{E}}
\newcommand{\cF}{\mathcal{F}}
\newcommand{\cL}{\mathcal{L}}

\newcommand{\cM}{\mathcal{M}}
\newcommand{\cO}{\mathcal{O}}

\newcommand{\cR}{\mathcal{R}}
\newcommand{\cT}{\mathcal{T}}
\newcommand{\cV}{\mathcal{V}}
\newcommand{\cW}{\mathcal{W}}

\newcommand{\cX}{\mathcal{X}}

\newcommand{\XFam}{\mathcal{X}}

\newcommand{\SheafHom}{\mathcal{H}om}

\newcommand{\moduli}[1][g]{{\mathcal M}_{#1}}
\newcommand{\omoduli}[1][g]{{\Omega \mathcal M}_{#1}}

\newcommand{\cI}{\mathcal{I}}

\newcommand{\sing}{\mathrm{sing}}
\newcommand{\Sing}{\mathrm{Sing}} 
\newcommand{\supp}{\mathrm{supp}}

\newcommand{\hyp}{{\rm hyp}}

\newcommand{\tr}{{\rm tr}}
\newcommand{\Gr}{{\rm Gr}}

\newcommand{\semistable}{\mathit{sst}}
\newcommand{\Per}{\mathrm{Per}}
\newcommand{\rk}{\mathrm{rk}}

\newcommand{\ol}[1]{{\overline{#1}}}
\newcommand{\ul}[1]{{\underline{#1}}}
\renewcommand{\tilde}[1]{\widetilde{#1}}
\newcommand{\fracpart}[1]{\left\langle #1 \right\rangle}

\renewcommand{\Re}{\mathrm{Re\,}}        
\renewcommand{\Im}{\mathrm{Im\,}}       

\renewcommand{\top}{{\rm top}}

\newcommand{\dual}{\vee} 
\newcommand{\isom}{\cong}
\newcommand{\sms}{\setminus}
\newcommand{\orb}{\mathrm{orb}} 
\newcommand{\nc}{\mathrm{nc}} 

\newcommand{\parab}{{\rm par}}
\newcommand{\dR}{\mathrm{dR}}
\newcommand{\backmod}{\backslash}

\DeclareMathOperator{\Aut}{Aut}

\DeclareMathOperator{\vol}{vol}
\DeclareMathOperator{\dvol}{dvol}

\DeclareMathOperator{\diag}{diag}
\DeclareMathOperator{\Sym}{Sym}

\DeclareMathOperator{\Hom}{Hom}
\DeclareMathOperator{\Pic}{Pic}
\DeclareMathOperator{\koker}{\mathrm{coker}}
\DeclareMathOperator{\congruent}{\equiv}

\newcommand{\Teichmuller}{{Teich\-m\"uller} }

\DeclareMathOperator{\FDom}{\mathcal{F}}

\DeclareMathOperator{\Ric}{\mathrm{Ric}}

\newcommand{\Laplace}{\Delta}
\newcommand{\dd}{\,\mathrm{d}}

\DeclareMathOperator{\tensor}{\otimes}   
\DeclareMathOperator{\Chern}{\mathrm{c}} 
\DeclareMathOperator{\orbEuler}{\mathrm{e}^{\orb}} 
\DeclareMathOperator{\topEuler}{\mathrm{e}^{\top}} 

\newcommand{\lind}{\langle}
\newcommand{\rind}{\rangle_{1,n}} 
\newcommand{\ldef}{(}
\newcommand{\rdef}{)} 

\newcommand{\pder}[2][]{\frac{\partial #1}{\partial #2}}  
\newcommand{\set}[2]{\bigl\{#1\mid #2\bigr\}} 		
\newcommand{\ie}{i.\,e.\ }

\makeatletter
\renewcommand{\subsubsection}{\@startsection{subsubsection}{2}%
        {\z@}{-3.25ex plus -1ex minus-.2ex}{-1em}{\bf}}
\makeatother

\title[Lyapunov spectrum of ball quotients]{Lyapunov spectrum of
ball quotients with applications
to commensurability questions}

\author{Andr\'e Kappes and Martin M\"{o}ller}
\date{\today}

\thanks{The authors are partially supported
by the ERC-StG 257137. }

\address{Institut f\"{u}r Mathematik, Goethe-Universit\"{a}t Frankfurt, 
Robert-Mayer-Str. 6-8, 60325 Frankfurt am Main, Germany}

\email{kappes@math.uni-frankfurt.de}
\email{moeller@math.uni-frankfurt.de}

\begin{document} 
\begin{abstract} We determine the Lyapunov spectrum of ball quotients 
arising from cyclic coverings. The computations are performed by rewriting
the sum of Lyapunov exponents as ratios of intersection numbers 
and by the analysis of the period map near boundary divisors.
\par
As a corollary, we complete the classification of commensurability
classes of all presently known non-arithmetic ball quotients.
\end{abstract}

\maketitle
\setcounter{tocdepth}{1}
\tableofcontents

\noindent
%

\section{Introduction}
\par
This paper is intended to contribute to the problem of classifying commensurability classes of non-arithmetic ball quotients
with a technique that has been useful to understand  the \Teichmuller geodesic flow, the 
calculation of the Lyapunov spectrum.
\par
By Margulis' arithmeticity theorem (\cite{margulisbook}) non-arithmetic lattices
only exist in Lie groups of rank one. While there are irreducible
non-arithmetic lattices in the isometry group of real hyperbolic space 
of dimension $n$ for any $n$, the construction of non-arithmetic lattices
in the isometry group of complex hyperbolic $n$-space $\PU(1,n)$
 is a hard open problem if $n > 1$.
\par
Up to commensurability, all\footnote{After the completion of this work, 
finitely many new non-arithmetic ball quotients not commensurable to the Deligne-Mostow 
examples have been found by Deraux, Parker and
Paupert \cite{derauxppnew}, see also \cite{derauxppcensus}}
the presently known non-arithmetic ball quotients
arise from cyclic coverings of the projective line, an investigation started by
\cite{DeligneMostow86} and completed
by \cite{mostowhalfint} and \cite{thurstonshapes}. 
Other constructions, earlier and shortly after the work of Deligne, Mostow
and Thurston, turned out to be commensurable to these ball quotients.
In fact, the book \cite{delcommen} gathers a lot 
of techniques to detect commensurabilities between lattices. 
\par
To detect non-commensurability, 
the only technique appearing in the literature seems to be the trace field
and (non-)compactness. As a consequence, \cite[Remark~5.1]{paupertIII} asks whether among
the
$15$ cyclic covering examples, there are $7$, $8$ or $9$ commensurability
classes. As one result of our methods we show that there are in fact $9$
classes. The first new commensurability invariant that we propose is the
set of Lyapunov exponents of a variation of Hodge structures associated 
with ball quotients arising from cyclic coverings. We first give some 
background on these notions and then explain that they fit into a larger 
class of commensurability invariants.
\par
\paragraph{\textbf{Cyclic coverings}} 
The ball quotients of Deligne and Mostow stem from
families of algebraic
 curves that are cyclic coverings of $\PP^1$ branched at $N$ points.
For a fixed degree $d$ and a ramification datum $(a_1,\dots,a_N)$, one
considers the algebraic curve
\[y^d = \prod_{i=1}^N (x-x_i)^{a_i},\qquad (x_1,\dots,x_N)\in \moduli[0,N]\]
Moving the branch points in $\moduli[0,N]$ yields a family of curves.
The first cohomology groups of the fibers of this family form a local system
(or flat vector bundle) on $\moduli[0,N]$ and the ${(1,0)}$-subspaces of the
Hodge decomposition yield a holomorphically varying subbundle. 
This is the prototypical example of a variation of Hodge structures of weight
1 (VHS) to have in mind for the theorem below (see Section~\ref{sec:weigh1VHS}
for a definition of VHS). 
\par
Using the action of the Galois group $\ZZ/(d)$ of the
covering, one can decompose the cohomology into eigenspaces; these form
themselves variations of Hodge structures defined over some cyclotomic field,
and the natural polarization on the fibers sometimes is a hermitian form of
signature $(1,n)$, where $n= N-3$. In this case, the period map, which records
the position of the $(1,0)$-subspace in the cohomology, is a holomorphic map
from the universal cover of $\moduli[0,N]$ to a ball $\BB^n$, which is the
parameter space of Hodge structures of this type. The fundamental group of
$\moduli[0,N]$ in turn acts on a fiber of the local system by parallel transport
of flat sections; since it preserves the hermitian form, this yields a
representation into $\PU(1,n)$, and the period map is equivariant for the two
actions of the fundamental group. The main achievement of \cite{DeligneMostow86} is
to define a suitable compactification of $\moduli[0,N]$ and to determine for
which parameters $(d;a_i)$ and which eigenspace one can extend the period map to
yield an isomorphism with the complex ball. In this case, the representation of
the fundamental group gives a lattice in $\PU(1,n)$, and the VHS is called
\textit{uniformizing}.
\par
\paragraph{\textbf{Lyapunov exponents}}
The {\em Lyapunov exponents} of a VHS of rank $k$ are $2k$ real
numbers $\lambda_i$ that group symmetrically around $0$ and measure, roughly
speaking, the logarithmic growth rate of cohomology classes under parallel
transport along the geodesic flow of the ball quotient. By Oseledets' theorem
such Lyapunov exponents can be associated to any ergodic cocycle under a weak
integrability hypothesis which holds for variations of Hodge structures, see
Lemma~\ref{le:Oseledec-integrability}.
\par
Lyapunov exponents for VHS have first been investigated for the space
$\omoduli$
of flat surfaces, i.e.\ pairs $(X,\omega)$ of a compact Riemann surface
together with a non-zero holomorphic one-form.
This space admits a natural action of $\SL_2(\RR)$, and the action of the
diagonal subgroup is the \Teichmuller geodesic flow, see \cite{zorich06} for a
survey. In this situation Lyapunov exponents have been determined e.g.\ for
families of cyclic coverings branched over $4$ points in \cite{bouwmoel}
and \cite{cyclicekz}. The simplicity of the spectrum (but no precise values) are
known for a generic flat surface $(X,\omega)$ of genus at least two. A
fundamental observation of Kontsevich and Zorich (see \cite{kontsevich} and also
\cite{ekz}) relates the sum of Lyapunov exponents to the degree of the Hodge
bundle (or a summand, if the Hodge structure splits) and this degree can
actually be calculated.
\par 
For a comparison with our Theorem~\ref{thm:main}, we
state Kontsevich's formula for the sum of Lyapunov exponents in the case of a
family of curves $f:\XFam \to C$ over a compact hyperbolic curve $C
=\HH/\Gamma$. The VHS in question is the relative cohomology $R^1f_*\CC$
together with the subbundle $f_*\omega_{\XFam/C}$ of relative $1$-forms pushed
forward to
$C$. In this case,
\begin{align}
\sum_{i=1}^g \lambda_i =
\frac{2\Chern_1(f_*\omega_{\XFam/C})}{c_1(\omega_{C})}
\end{align}
If $C$ is not compact as in the case of \Teichmuller curves, one has to take
the Deligne extension of $f_*\omega_{\XFam/C}$ and the canonical bundle
$\omega_{\ol{C}}$ of the completion $\ol{C}$ instead.
\par
Kontsevich's observation can be extended to ball quotients, since they are 
K\"ahler-Einstein manifolds as in the one-dimensional case.
More precisely, we have the following result, which we state for abstract
variations of Hodge structures.
\par
\begin{theorem} \label{thm:main}
Suppose that $\VV$ is a real polarized 
variation of Hodge structures of weight $1$ over
a ball quotient $B = \BB^n/\Gamma$ of
constant curvature $-4$, where $\Gamma$ is a torsionfree lattice in
$\PU(1,n)$. Let $\ol{B}$ be a smooth compactification of $B$ with normal
crossing boundary divisor $\Delta$, and assume that the local monodromy
of $\VV$ about $\Delta$ is unipotent.
Then the Lyapunov spectrum of 
$\VV$ has the following properties.
\begin{itemize}
\item[i)] \textbf{Normalization.} If $\VV_\CC = \VV\tensor_\RR{\CC}$ has an irreducible summand
which is uniformizing, then the top Lyapunov exponent is one.
\item[ii)] \textbf{Duplication.} If an $\RR$-irreducible direct summand
$\WW$ of $\VV$ is reducible over $\CC$, then each Lyapunov exponent of
$\WW$ has even multiplicity.
\item[iii)] \textbf{Zero exponents for non-real factors.}
If an irreducible summand $\WW_\CC$ of $\VV_\CC$ has signature 
$(p,q)$, then at least $2|p-q|$ of the Lyapunov exponents 
corresponding to $\WW_\CC$ are zero.
\item[iv)] \textbf{Partial sums are intersection numbers.}
Let $\WW$ be a direct summand of rank $2k$ in the decomposition of
$\VV$. Then the positive
Lyapunov exponents $\lambda_1,\dots,\lambda_k$ of $\WW$ satisfy
\begin{equation} \label{eq:sumlyapformula}
\lambda_1+\dots +\lambda_k =
\frac{(n+1)\Chern_1(\cE^{1,0}).\Chern_1(\omega_{\ol
B})^{n-1}}{\Chern_1(\omega_{\ol B})^n},
\end{equation}
where $\omega_{\ol B} = \bigwedge^n \Omega^1_{\ol
B}(\log\Delta)$, and $\cE^{1,0}$ is the Deligne extension of $\cW^{1,0} \subset \WW\tensor\cO_B$.
\end{itemize}
\end{theorem}
\par
\par
\paragraph{\textbf{Commensurability invariants}}
Due to their construction via families of curves, all the ball quotients 
arising from cyclic coverings come with a $\QQ$-VHS that contains a uniformizing
sub-VHS. The Galois conjugates of this sub-VHS constitute the {\em primitive
part} that we associate to such a ball quotient (see Definition~\ref{def:pp} for
the precise statement). 
The local system of the primitive part only depends on
the lattice as it corresponds to the sum of the representations
$\Gamma\to\Gamma^\sigma$ given by Galois conjugation. On the other hand, the
Hodge decomposition is a priori not an intrinsic datum of the lattice. However,
it is unique if it exists. This is shown in Theorem~\ref{thm:modemb_unique}
using the notion of \textit{modular embeddings}. Informally, a modular embedding
is the collection of the equivariant period maps associated with the different
representations $\Gamma\to \Gamma^\sigma$. Using this terminology, we show
in Section~\ref{sec:commensurability}:
\par
\begin{theorem}[{see Theorem~\ref{thm:lyap_spec_comminvariant}}] 
\label{thm:introcomm}
The Lyapunov spectrum of the primitive part is a commensurability invariant
among lattices in $\PU(1,n)$ that admit a modular embedding.
\end{theorem}
\par
The right hand side of \eqref{eq:sumlyapformula} can also be generalized to give
a new class of commensurability invariants.
In general, for lattices as in Theorem~\ref{thm:introcomm} the ratios of the 
form 
\begin{equation} \label{eq:cab}
\frac{\Chern_1(\cE^{1,0})^a.\Chern_1(\omega_{\ol
B})^{b}}{\Chern_1(\omega_{\ol B})^n}
\end{equation}
for all $(a,b)$ with $a+b=n$ are natural invariants of the 
commensurability class of $\Gamma$. Here the $\cE^{1,0}$ correspond to irreducible
summands in the primitive part of the VHS associated with $\Gamma$.
We refer to Corollary~\ref{cor:comminvariants} for the precise statement.
\par
\medskip
The second part of this paper shows that for all presently known
commensurability classes of ball quotients all {\em individual}
Lyapunov exponents in the primitive part can be calculated using
Theorem~\ref{thm:main} and local computations.
\par
First of all, arithmetic lattices are
of little interest in this context. As a consequence of Theorem~\ref{thm:main}
we show in Proposition~\ref{prop:arithmetic} that for an arithmetic lattice
the primitive Lyapunov spectrum is {\em  maximally degenerate}, 
i.e.\ it consists of $\{+1, 0, -1\}$ only and the number of $+1$
is determined by the signature of the Hodge inner product.
\par
The non-arithmetic examples are much more interesting. We calculate
the individual Lyapunov exponents of the primitive part for all known
non-arithmetic examples in Theorem~\ref{thm:allprimvalues}. We state
a particularly interesting case here.
\par
\begin{theorem} \label{thm:lyapSU21}
Let $f:\XFam \to B_0 = (\PP^1)^5 \setminus (\text{diagonals})$ 
be the family of cyclic coverings of $\PP^1$ given by
\begin{equation} \label{eq:intfam1} 
y^{12} = (x-x_1)^3(x-x_2)^3(x-x_3)^5(x-x_4)^6(x-x_5)^7, 
\end{equation}
respectively by
\begin{equation} \label{eq:intfam2} 
y^{12} = (x-x_1)^4(x-x_2)^4(x-x_3)^4(x-x_4)^5(x-x_5)^7. 
\end{equation}
Let $\LL$ be the uniformizing direct summand of $R^1f_*\CC$
that when extended to a suitable compactification $B$ of $B_0$ exhibits
$B$ as a non-compact orbifold ball quotient.
Then the Lyapunov exponents of the primitive part $\PP$ are given by
\[1,\tfrac{5}{17},0,-\tfrac{5}{17},-1,\qquad
\text{respectively}\qquad
1,\tfrac{7}{22},0,-\tfrac{7}{22},-1,\]
where $\PP$ in this case is the $\CC$-subvariation associated 
with the direct sum of the Galois
conjugates of the $\QQ(\zeta_{12})$-form
of $\LL$.
\end{theorem}
\par
As a corollary to this theorem we obtain that the two associated lattices are
not commensurable to each other. The full commensurability statement mentioned
at the beginning is given in Corollary~\ref{cor:commensurability-result}. 
\par
In order to prove Theorem~\ref{thm:allprimvalues}, we reduce the calculation of
$\Chern_1(\cE^{1,0})$ to the
computation of local invariants near the 'boundary divisors', where two
of the branch points collapse. It turns out that it suffices to compute the
cokernel of Kodaira-Spencer maps near these boundary divisors and
this will be done by a local analysis of hypergeometric integrals.
As final piece of information we need the intersection rings and
Chern classes of the tangent bundle of the moduli spaces of weighted
stable curves (or their quotients by finite groups in the case
\eqref{eq:SigmaINT}).
\par
Naturally, one would also like to calculate the Lyapunov exponents of the
non-primitive part of the VHS for a cyclic covering. However, one runs into the
following problem.
\par
\begin{rem}
{\rm The problem of calculating the
Lyapunov exponents of the whole
variation of Hodge structures $R^1f_*\CC$ of a family of cyclic coverings
branched over $N\geq 5$ points is
not well-defined in any of the non-arithmetic cases, since there always exists a
direct summand that does not extend to the ball, i.e. whose period map is not
defined on $\BB^n$ (see Example~\ref{ex:non-extendability}). For a discussion
of the case $N=4$, we refer to \cite{bouwmoel}, \cite{Wright11}, and
\cite{cyclicekz}.} 
\end{rem}
\par
\paragraph{\textbf{Orbifold Euler numbers}}
After a first version of this paper was circulated, McMullen 
pointed out to us that one can define another invariant for non-arithmetic 
ball quotients coming from cyclic coverings, the relative orbifold Euler numbers. 
For each Galois conjugate $\Gamma^\sigma$ of $\Gamma$, there is a hyperbolic cone 
manifold $B^\sigma$
(with $B^{\id} = \BB^n/\Gamma$ being the usual orbifold ball quotient). 
The relative Euler numbers are then defined as the collection
$\orbEuler(B^\sigma)/\orbEuler(B^{\id})$. The precise definition, entirely combinatorial, 
is given in Section~\ref{sec:orbeuler_logbq} and an explicit formula is stated 
in Proposition~\ref{prop:formula for orbEuler for mu}. 
McMullen's  observation was that the two sets of invariants, 
positive Lyapunov exponents and relative orbifold Euler numbers, agree for one-dimensional
ball quotients, but differ in dimension two 
\footnote{The values that McMullen calculated (meanwhile also available in
the preprint 'The Gauss-Bonnet theorem for cone manifolds and volumes of moduli spaces' on his web page)
match with our Corollary~\ref{cor:compute-otherinvariant}. They are included in the table
in Theorem~\ref{thm:allprimvalues}.}. Our next result, a simplified restatement
of Corollary~\ref{cor:compute-otherinvariant} gives the explanation.
\par
\begin{theorem} \label{thm:introRelEuler} Let $\LL^\sigma$ be a Galois conjugate
of the uniformizing VHS $\LL$ of a family of cyclic coverings. Suppose $n=2$ and 
that the polarization of $\LL^\sigma$ has signature $(1,2)$.  Then
\begin{equation} \label{eq:introRelEuler}
\frac{\orbEuler(B^\sigma)}{\orbEuler(B^{\id})} = 
\frac{\bigl(3\cdot \Chern_1(\cE^{1,0})\bigr)^2}{\Chern_1(\omega_{\ol B})^2}.
\end{equation}
For the families of curves given by \eqref{eq:intfam1} and \eqref{eq:intfam2} the set 
of relative orbifold Euler numbers is 
\[1,\tfrac{1}{17},\qquad
\text{respectively}\qquad
1,\tfrac{1}{22}\]
\end{theorem}
\par 
Said differently, the sum of Lyapunov exponents corresponds to the case 
$(1,n-1)$ of the invariants in \eqref{eq:cab} while the orbifold
Euler characteristics correspond to the case $(2,0)$.
\par
To prove Theorem~\ref{thm:introRelEuler} we show 
in Theorem~\ref{thm:log-ballquotients} 
that $B^\sigma$ is a log ball quotient, i.e.\ for an appropriate choice of
weights on boundary divisors this log manifolds attains the upper bound in
Langer's logarithmic version of the Bogomolov-Miyaoka-Yau inequality
(\cite{langerorbifold}). 
\par
\medskip
\paragraph{\textbf{Structure of the paper}} The plan of the paper is as follows.
In Section~\ref{sec:backmoduli} we provide background
information about ball quotients, Lyapunov exponents and variations of Hodge structures.
Section~\ref{sec:cylicgeneral} contains the proof of Theorem~\ref{thm:main}. 
 Section~\ref{sec:commensurability} contains the definition
of modular embeddings for ball quotients and the proof of commensurability invariance of
the primitive Lyapunov spectrum. The arithmetic case
is quickly discussed in Section~\ref{sec:arithmeticity}.
The basic facts about cyclic covers and those
cases that give rise to (non-arithmetic) ball quotients are recalled in Section~\ref{sec:cyliccover}.
This section also contains the description of intersection rings of these ball quotients.
\par
\medskip
\paragraph{\textbf{Acknowledgments}} The authors thank J\"urgen Wolfart for
useful comments on modular embeddings, Simion Filip for useful comments
on the proof of Lemma~2.7, Curt McMullen for asking about the relation 
between Lyapunov exponents and orbifold Euler numbers and the 
referees for their valuable suggestions to 
improve the exposition of the paper.

\section{Background} \label{sec:backmoduli}

\subsection{The complex ball and its K\"ahler structure}
We first collect some well-known facts about the K\"ahler structure
of ball quotients. We carefully carry along the dependence on the
curvature in order to have a consistent normalization in Theorem~\ref{thm:main}~i).
Let $\CC^{1,n}$ be $\CC^{n+1}$ equipped with the following 
hermitian pairing. For $\ul{W} = (W_0, {W})$ and  $\ul{Z} = (Z_0, {Z})
\in \CC^{n+1}$
we let 
\[ \lind \ul{W},\ul{Z} \rind = W_{0} \ol{Z_{0}} - \ldef {W},{Z} \rdef,\]
where $\ldef {W}, {Z}\rdef = \sum_{i=1}^{n} {W}_i\ol{Z_i}$.
Then $\HH_{\CC}^n$ is the space of lines where the hermitian pairing is positive
definite. We may identify $\HH_{\CC}^n$ with
\[ \BB^n = \{{z} = (z_1,\ldots, z_n) : \sum |z_i|^2 < 1 \} \]
via the map $f:  \BB^n \to \PP(\CC^{1,n}), {z} \mapsto (1:{z})$.
\par
$\BB^n$ is naturally endowed with a K\"ahler structure, whose K\"ahler form in
these coordinates is given by
\begin{equation*}
\begin{aligned}
 & \omega_\hyp = 2\kappa i \partial\ol{\partial} \log(\tfrac{1}{2} \lind
f(z),f(z) \rind ) \ \\
& = \frac{2\kappa i}{(1-\ldef z,z\rdef)^2} \left\{
\left(\sum_{j=1}^n \bar{z_j}dz_j \right) \wedge \left(\sum_{k=1}^n
{z_k}d\bar{z}_k \right)
+ (1-\ldef z,z\rdef) \sum_{j=1}^n dz_j \wedge d\bar{z}_j
\right\} \\
\end{aligned}
\end{equation*}
with $\kappa > 0$ \cite{Goldman99}. Its holomorphic sectional curvature is $\tfrac{-2}{\kappa}$.
Setting $\kappa = \tfrac{1}{2}$, we obtain the metric of Theorem~\ref{thm:main}.

The corresponding Riemannian metric is given by
\[ g_\hyp
 = \frac{2\kappa }{(1-\ldef z,z\rdef)^2} \left\{
\left(\sum_{j=1}^n \bar{z_j}dz_j \right)  \left(\sum_{k=1}^n {z_k}d\bar{z}_k
\right)
+ (1-\ldef z,z\rdef) \sum_{j=1}^n dz_j d\bar{z}_j \right\}. 
\]
The volume form is
\begin{align}\label{eqn:volume form on Bn}
\dvol_{\BB^n} = \tfrac{1}{n!}\omega_\hyp^n = \frac{(2\kappa i)^n}{(1-\ldef z,z
\rdef)^{n+1}}dz_1\wedge d\bar{z}_1 \wedge \dots \wedge dz_n\wedge d\bar{z}_n.
\end{align}

\subsubsection{The Laplacian on the complex ball}

Let $\Laplace_\hyp$ be the Laplacian for the Riemannian metric on $\BB^n$. It is
given by 
\begin{align}\label{eqn:Laplacian form on Bn}
\Laplace_\hyp = 
\tfrac1{\kappa}(1-\ldef z,z\rdef) \left(\sum_{i=1}^n
\frac{\partial^2}{\partial z_i\partial\bar{z}_i} - \sum_{i,j = 1}^n
z_i\bar{z}_j\frac{\partial^2}{\partial z_i\partial\bar{z}_j}\right).
\end{align}

\par
\begin{lemma} \label{le:omegahochnminus1}
If $F$ is a smooth function on the ball, we have
\begin{align} \label{eqn: Laplace_and_ddbar}
\Laplace_{\hyp}(F)\cdot \dvol_{\BB^n} = \frac{2 i}{(n-1)!} \cdot \partial\bar{\partial} F \wedge \omega_\hyp^{n-1}
\end{align}
\end{lemma}
\begin{proof}
Since $\BB^n$ is K\"ahler, there is a coordinate system $w_1,\dots,w_n$ 
centered at $x\in \BB^n$ such that in the fiber over $x$
\[\omega_x = 2\kappa i \sum_{j=1}^n dw_{j,x}\wedge d\bar{w}_{j,x},\quad \text{and}\quad
 \Laplace_{\hyp}(F)_x = \tfrac1{\kappa} \sum_{i=1}^n\frac{\partial^2F}{\partial w_i\partial\bar{w}_i}\Big|_x.\]
 Therefore, it suffices to prove~\eqref{eqn: Laplace_and_ddbar} 
for this particular coordinate system, 
which is a straightforward calculation.
\end{proof}

\par

\subsubsection{Geodesic polar coordinates} 
We use the parametrization of $\BB^n$ by geodesic polar
coordinates centered at $x\in \BB^n$. Let $\theta =
(\theta_1,\dots,\theta_{2n-1})$ be a
local parametrization of $T^1_x\BB^n\subset T_x\BB^n$. Then the map
\[\exp_x: T_x\BB^n \to \BB^n, \quad (t,\theta) \mapsto
\exp_x(t\theta_1,\dots,t\theta_{2n-1})\]
parametrizes $\BB^n$ by geodesic polar coordinates.

 The volume form on $\BB^n$ written in geodesic polar coordinates is
 \[\dvol_{\BB^n} = \frac{(4\kappa)^n}{2\sqrt{2\kappa}}
\sinh^{2n-2}\biggl(\frac{t}{\sqrt{2\kappa}}\biggr)\sinh\biggl(\frac{2t}{\sqrt{
2\kappa } } \biggr)\dd t\dd
\sigma,\]
where $\dd\sigma$ denotes the volume form on $T^1_x\BB^n$.
This follows from \eqref{eqn:volume form on Bn} and the $\PU(1,n)$-invariance,
by which we can reduce the computation to $x= 0$. Further, we use the equality
\[\left(\tfrac{i}{2}\right)^n\dd z_1\wedge \dd\bar z_1 \wedge \dots \wedge
\dd z_n\wedge \dd \bar{z}_n = \dd x_1\wedge \dd y_1 \wedge \dots \wedge \dd x_n
\wedge
\dd y_n\]
and
\[\dd x_1\wedge \dd y_1 \wedge \dots \wedge \dd x_n \wedge \dd y_n = r^{2n-1}\dd
r\dd\sigma.\]
Moreover, the Euclidean radius $r$ is related 
to the hyperbolic radius $t$ by $r = \tanh(t/(\sqrt{2\kappa}))$.
Hence the formula follows using
$\cosh(x)\sinh(x) = \tfrac{1}{2}\sinh(2x)$.

Moreover, we need the volume of $\BB^n_t = \BB^n_t(x) = \set{z\in
\BB^n}{d_{\hyp}(x,z) < t}$. It is
given by
\[\vol(\BB^n_t) = \int_{0}^t \int_{T^1_x\BB^n} \dvol_{\BB^n} =
\frac{(4\kappa)^n\sigma_{2n-1}}{2n}\sinh^{2n}\Bigl(\tfrac{t}{\sqrt{2\kappa}}
\Bigr).\]
Here $\sigma_{2n-1}$ is the Euclidean volume of $T^1_x\BB^n$.

Finally, by \cite{HelgGeomAna}
the Laplacian on $\BB^n$ written in geodesic polar coordinates centered at $0$
is
\begin{align}\label{eqn:Laplacian on Bn}
\Laplace_{\hyp} = \frac{\partial^2}{\partial t^2} + \frac{1}{A(t)}\frac{\dd
A}{\dd t}\frac{\partial}{\partial t} + \Laplace_{S_t(0)}
\end{align}
where $A(t)$ is the area of the sphere $S_t(0)\subset \BB^n$ of radius $t$
centered at $0$.
Here $\Laplace_{S_t(0)}$ is the Laplacian on $S_t(0)$ for the Riemannian metric
induced from the one on $\BB^n$.

We have
\[A(t) = \frac{\dd}{\dd t}\vol(\BB^n_t) =
\frac{(4\kappa)^n\sigma_{2n-1}}{2\sqrt{2\kappa}}\sinh^{2n-2}(\tfrac{t}{\sqrt{
2\kappa}})\sinh(\tfrac{2t}{\sqrt{2\kappa}}).\]
Therefore the factor in front of $\partial/\partial t$ in \eqref{eqn:Laplacian
on Bn} is
\begin{align*}
A(t)^{-1}A'(t) &=
A(t)^{-1}\cdot(4\kappa)^{n-1}\sigma_{2n-1}2\sinh^{2n-2}(\tfrac{t}{\sqrt{2\kappa}
})\bigl(2n\cosh^2(\tfrac{t}{\sqrt{2\kappa}}) -1\bigr)\\
	       &=
\frac{\sqrt{2\kappa}}{\kappa}\bigl(n\coth(2\tfrac{t}{\sqrt{2\kappa}}) +
\frac{n-1}{\sinh(2\tfrac{t}{\sqrt{2\kappa}})}\bigr)
\end{align*}


The following version of Green's formula was used in \cite{forni02} for
the hyperbolic plane. It carries over to higher dimension. 
\par

\begin{lemma}\label{lem:Forni}
 Let $x\in\BB^n$, let $\Lambda: \BB^n\to \RR$ be a smooth function, and
let $L$ be a
smooth solution to the Poisson equation $\Laplace L = \Lambda$. Then
\[\frac{1}{\sigma_{2n-1}}\frac{\partial}{\partial t} \int_{T^1_x\BB^n}
L(\exp_x(t,\theta)) d\sigma =
\frac{\sqrt{2\kappa}}{2n}\tanh\biggl(\frac{t}{\sqrt{2\kappa}}\biggr)\cdot
\frac{1}{\vol(\BB^n_t)} \int_{\BB^n_t(x)} \Lambda \dvol_{\BB^n}.\]
\end{lemma}

\begin{proof}[Proof of Lemma~\ref{lem:Forni}]
By homogeneity, we can assume $x = 0$.
Proceeding in exactly the same way as in \cite{forni02}, we set
\[L_r(t) = \frac{1}{\sigma_{2n-1}} \int_{T^1_0\BB^n} L(\exp_0(t,\theta))
\dd\sigma\quad \text{and} \quad \Lambda_r(t) = \frac{1}{\sigma_{2n-1}}
\int_{T^1_0\BB^n} \Lambda(\exp_0(t,\theta)) \dd\sigma\]
Then since $\Laplace_{S_t(0)}L_r = 0$ as it involves only partial derivatives in
$\theta$, we have
\begin{align}
\Laplace_{\hyp}L_r(t) = \frac{\partial^2}{\partial t^2}L_r(t) +
f(t)\frac{\partial}{\partial t}L_r(t) = \Lambda_r(t). 
\end{align}
Thus $u(t) = \frac{\partial}{\partial t}L_r(t)$ satisfies a first-order ODE,
whose solution is
\begin{align}
u(t) =
\frac{1}{\sinh^{2n-2}(\tfrac{t}{\sqrt{2\kappa}})\sinh(\tfrac{2t}{\sqrt{2\kappa}}
)} \int_0^t
\Lambda_r(\tau)\sinh^{2n-2}(\tfrac{\tau}{\sqrt{2\kappa}})\sinh(\tfrac{2\tau}{
\sqrt{2\kappa}}) \dd \tau 
\end{align}
Therefore,
\begin{align*}
\frac{1}{\sigma_{2n-1}}u(t) &=
\frac{(4\kappa)^n}{2n}\sinh^{2n}(\tfrac{t}{\sqrt{2\kappa}})\cdot\frac{1}{
\vol(\BB^n_t)} u(t)\\
			    &=
\frac{\sqrt{2\kappa}}{n}\cdot\frac{\sinh^{2}(\tfrac{t}{\sqrt{2\kappa}})}{
\sinh(\tfrac{2t}{\sqrt{2\kappa}})}\cdot\frac{1}{\vol(\BB^n_t)} \cdot \\
&\phantom{xxx}\cdot \int_0^t
\int_{T^1_0\BB^n} \Lambda(\exp_0(\tau,\theta))
\frac{(4\kappa)^n}{2\sqrt{2\kappa}}\sinh^{2n-2}(\tfrac{\tau}{\sqrt{2\kappa}}
)\sinh(\tfrac{2\tau}{\sqrt{2\kappa}}) \dd \tau \dd\sigma\\
			    &=  \frac{\sqrt{2\kappa}}{2n}
\tanh(\tfrac{t}{\sqrt{2\kappa}})\frac{1}{\vol(\BB^n_t)}\int_{\BB^n_t}\Lambda
\dvol_{\BB^n}
\end{align*}
\end{proof}

\subsection{Weight one variations of Hodge structures (VHS)} \label{sec:weigh1VHS}
The first cohomology of a complex algebraic curve $X$ 
naturally carries a Hodge decomposition 
$H^1(X,\CC) = H^{1,0}\oplus H^{0,1}$ into the direct sum of classes of 
holomorphic and antiholomorphic forms. The generalization of the Hodge
decomposition
to families of curves or more generally to families of compact K\"ahler manifolds
is abstracted by the notion of a variation of Hodge structres. 
Below, we collect the basic definitions needed. For simplicity, we state everything for
weight one only. A general reference for this section is \cite{CMSP}.
\par
Let $K\subset \RR$ be a field, and let $W$ be a $K$-vector space. By a \emph{$K$-Hodge structure 
(of weight one)} on $W$, we understand a decomposition $W\tensor_K\CC = W^{1,0}\oplus W^{0,1}$, such 
that $\ol{W^{1,0}}= W^{0,1}$. We extend this notion to $K=\CC$ by dropping the condition 
that $\ol{W^{1,0}}= W^{0,1}$.
\par
Let $B$ be a base manifold which is the complement in the projective
variety $Y$ of a normal crossing divisor $\Delta$. 
A \emph{(weight one) $K$-variation of Hodge structures} over 
$B$ 
consists of a $K$-local system $\WW$ on $B$ and a $C^\infty$-decomposition
\[\WW\tensor_K\cO_B = \cW^{1,0}\oplus \cW^{0,1}\]
into a holomorphic (resp. antiholomorphic)
subbundle $\cW^{1,0}$ (resp. $\cW^{0,1}$) that induces fiber by fiber a 
Hodge structure of weight $1$. Note that for higher weight, there is an additional 
transversality condition, which is vacuous for weight one.
\par
A $\CC$-VHS $\WW$ is called {\em  polarized}, if there is a
locally constant hermitian form $\psi$ on $\WW$, for which the decomposition
$\cW^{1,0}\oplus \cW^{0,1}$ is orthogonal, and which is positive definite
on $\cW^{1,0}$ and negative definite on $\cW^{0,1}$. Consequently, its 
signature is $(\rk \cW^{1,0}, \rk \cW^{0,1})$. This indefinite hermitian form
will be referred to as Hodge inner product. If $K\subset \RR$,
we require instead the existence of a non-degenerate, locally constant,
antisymmetric form $Q( \cdot,\cdot )$ on
$\WW$, which is zero
on $\cW^{1,0} \tensor \cW^{1,0}$ and on $\cW^{0,1} \tensor
\cW^{0,1}$ and such that $i^{p-q} Q( v,\ol{v}) >0$ for
every non-zero $v \in \cW^{p,q}$. The indefinite hermitian form on
$\WW_\CC = \WW\tensor_K\CC$ is then given by\footnote{Be aware that our
definition of $\psi$ parallels the one used for the area
of a flat surface (see e.g. \cite{ekz}). In Hodge theory (e.g.
\cite{schmid73}), the factor $\tfrac12$ is usually omitted.}
\[\psi(v,w) := \tfrac{i}{2} Q( v,\ol{w} ) \]
\par
If $K \subset \RR$, then we may write every element $w \in  \WW$
as $w = \Re(\omega)$ for some $\omega\in \cW^{1,0}$ and define the 
{\em Hodge norm} as
$\|w\| = \psi(\omega,\omega)^{1/2}$. We extend this norm to
$\WW_\CC = \WW_\RR + i\WW_\RR$ as orthogonal direct sum and
also to $\CC$-sub-VHS. This Hodge norm will be used throughout
when talking about Lyapunov exponents. In particular, for $v\in \WW_{\CC}$ we
have
\begin{align*}
\|v\|^2 &= \|\Re(v)\|^2 + \|i\Im(v)\|^2 =
\psi((v+\ol{v})^{1,0},(v+\ol{v})^{1,0}) +
\psi((v-\ol{v})^{1,0},(v-\ol{v})^{1,0})\\
 &= 2(\psi(v^{1,0},{v^{1,0}}) + \psi(\ol{v^{0,1}},\ol{v^{0,1}})) 
\end{align*}
\par
Suppose that we are given a polarized $\RR$-VHS $\WW$ that splits
over $\CC$ as $\WW_{\CC} = \VV \oplus \ol{\VV}$ into $\CC$-VHS. Since
$\psi(\ol{v},\ol{v}) = -\psi(v,v)$, the above computation shows that we may
also flip the sign of $\psi$ on the negative definite part $\cV^{0,1}$ in order
to obtain a positive hermitian form that computes the Hodge norm of vectors in
$\VV$:
\[\|v\|^2 = 2(\psi(v^{1,0},v^{1,0}) - \psi(v^{0,1},v^{0,1})).\]
\par
If $f: \cX \to B$ is a family of curves, then $\WW = R^1f_*K$ is a
$K$-VHS of weight one, polarized by the natural intersection pairing on
cohomology.
\par
Let $W$ be a $\CC$-vector space and let $\psi$ be a polarization. 
The period domain $\Per(W)$ is the classifying space of 
all $\CC$-Hodge structures polarized by $\psi$ that can be put on $W$. 
Analogously, we define the period domain for $K$-Hodge structures 
on the $K$-vector space $W$, polarized by the $K$-bilinear form $Q$.
\par
Note that for weight one, $\Per(W)$ is always a hermitian symmetric domain. 
In fact, weight one $\CC$-Hodge structures of signature $(1,n)$ are parametrized by
the complex ball $\BB^n$, and weight one $\RR$-Hodge structures 
are parametrized by the Siegel upper halfspace $\HH_{\dim W}$.
\par
Let $b\in B$ be a base point. A polarized VHS $\WW$ defines a \emph{period map} 
$p: \hat{B} \to \Per(\WW_b)$ 
from the universal cover of $B$ to the period domain $\Per(\WW_b)$ that 
records the position of the $(1,0)$-subspace in the fibers of $\WW$.
The period map is only well-defined on the universal cover, since the identification
of the fibers of $\WW$ by parallel transport depends on the chosen path, and in fact produces 
an action of the fundamental group $\pi_1(B,b)$ on $\WW_b$, preserving
the polarization. The group homomorphism $\rho: \pi_1(B,b) \to \Aut(\WW_b, \psi_b)$ is 
called \emph{monodromy representation}. There is an induced action of $\pi_1(B,b)$
on $\Per(\WW_b)$ and the period map is equivariant 
with respect to the actions of $\pi_1(B,b)$ on the universal
cover by deck transformations and on $\Per(\WW_b)$. 
In particular, if $K$ denotes the Kernel of $\rho$, then $p$
drops to a map $\hat{B}/K \to \Per(\WW_b)$.
In the sequel, we will often omit the base point from the notation, 
if its choice does not matter.
\par
\begin{defn}
A VHS is called {\em uniformizing}, if the period map is an isomorphism. 
\end{defn}
A uniformizing VHS is unique up to isomorphism. Indeed, if $\WW_i$, $i=1,2$
are uniformizing VHS on $B$, we obtain a group isomorphism $\Gamma_1\to\Gamma_2$
of the monodromy groups and a biholomorphic map $\Per(\WW_1)\to \Per(\WW_2)$
equivariant with respect to this group isomorphism, which yields
$\WW_1\isom\WW_2$ as VHS.
\par
Suppose now that $\VV$ is a $\QQ$-VHS and that $\VV\tensor_\QQ\CC$ contains a uniformizing
sub-VHS $\WW$ of signature $(1,n)$. The assumption on the signature implies that
$\WW$ is irreducible. Recall also that by Deligne's semisimplicity 
theorem \cite[Prop.~1.13]{delfinitude}, $\VV = \bigoplus_i \WW_i \tensor E_i$, where $\WW_i$
are irreducible, polarized $\CC$-VHS and $E_i$ are $\CC$-vector spaces.
\par
\begin{defn} \label{def:pp}
We define the \emph{primitive part} of a $\QQ$-VHS $\VV$ to be the 
direct sum of those
$\WW_i\tensor E_i$, where $\WW_i$ is isomorphic as a local system to $\WW^\sigma$
with $\sigma\in \Aut(\CC/\QQ)$.
\end{defn}
\par
It is clear from the definition that the primitive part does not depend on
the choice of a uniformizing sub-VHS, since any two such are isomorphic.
Thus dropping the dependence on $\WW$ will cause no confusion. The primitive
part is fixed by any element of $\Aut(\CC/\QQ)$, and thus 
a $\QQ$-sub-VHS of $\VV$.
\par
For the application below we explicit the (unique up to isomorphism)
uniformizing VHS $\UU$ over $\BB^n$ together with the polarization in
terms of the form $\langle \cdot,\cdot \rangle_{1,n}$ used above. Over $\BB^n$
we represent a constant $\CC$-local system of rank $n+1$ by row vectors
and define $\UU^{1,0}_z$ as the line generated by $\ul{z} = (1,z)$ in
$\{z\}\times \CC^{n+1}$. Let $\pi_z$ denote the orthogonal projection on the
$1,0$-part. Then for $\ul{v}\in \{z\}\times \CC^{n+1}$
\begin{align}\label{eqn:unif-norm}
\begin{split}
\tfrac{1}{2}\|\ul{v}\|_z^2 &= \langle
\pi_{z}(\ul{v}),\pi_{z}(\ul{v})\rangle_{1,n} -
\langle (1-\pi_{z})(\ul{v}), (1-\pi_{z})(\ul{v}) \rangle_{1,n}\\
	   &= -\langle \ul{v},\ul{v} \rangle_{1,n} + 2\frac{|\langle \ul{v},
\ul{z}\rangle_{1,n}|^2}{\langle \ul{z},\ul{z} \rangle_{1,n}}
\end{split}
\end{align}


\subsubsection{\bf The Higgs field.} \label{sec:Higgscurvature}
By a theorem of Borel (see e.g.\cite[Lemma 4.5]{schmid73}), the local monodromy 
about a boundary divisor in $\Delta$ is always quasi-unipotent.
In the sequel we suppose throughout that the monodromy of the local 
system $\WW$ about $\Delta$ is already unipotent, or pass to a finite 
cover where this holds. 
Then the associated holomorphic vector bundle $\WW\tensor_{\CC} \cO_{B}$ has a 
unique extension due to Deligne \cite[Prop. 5.2]{delequadiff} 
to a holomorphic bundle $\cE$ on $Y$. The
extension $\cE^{1,0}$ of $\cW^{1,0}$ inside $\cE$ is a holomorphic subbundle
and we set
$\cE^{0,1} := \cE / \cE^{1,0}$. Further let $\Omega^1_Y(\log \Delta)$ be the
bundle of differential forms on $Y$ with logarithmic poles along $\Delta$; then
$\omega_Y = \bigwedge^{\dim B}\Omega^1_Y(\log \Delta)$ is called the
log-canonical bundle. The graded piece of the Gauss-Manin connection
is an $\cO_Y$-linear map
\begin{align*}
\tau: \cE^{1,0} \to \cE^{0,1}\tensor
\Omega_Y^1(\log \Delta),
\end{align*}
called the Higgs field. Its dual is the $\cO_Y$-linear map 
\begin{align*}
\tau^\dual: \cE^{1,0} \tensor \cT_Y(-\log \Delta) \to \cE^{0,1}\tensor
\Omega_Y^1(\log \Delta) \tensor \cT_Y(-\log \Delta) \to \cE^{0,1}
\end{align*}
obtained by composing with the natural contraction. Another viewpoint on the
same map is the map 
\begin{align*}
\tilde\tau : \cT_Y(-\log \Delta) \to
\SheafHom(\cE^{1,0},\cE^{0,1}) 
\end{align*}
obtained from $\tau^\dual$ by tensoring with $(\cE^{1,0})^*$. If $p:\hat B \to
\Per(\WW)$ denotes the period map of $\WW$, then its differential descends to a
sheaf map on $B$, which is equal to $\tilde\tau$ restricted to $B$.

\subsection{Lyapunov spectrum for polarized VHS}
We quickly recall Oseledec's theorem together with some properties
\par
\begin{theorem}[see e.g. {\cite{Ruelle79}}] \label{thm:oseledec}
Let $g_t: (M,\mu) \to (M,\mu)$ be an ergodic
flow on a space $M$ with finite measure $\mu$.
Suppose that the action of $t \in \RR$ lifts
equivariantly to a linear flow $G_t$ on some measurable 
real bundle $V$ on $M$. Suppose there exists
a (not equivariant) norm
$\|\cdot\|$ on $V$ such that the functions
\begin{align}\label{eq:Oseledec-integrability}
x\mapsto \sup_{t\in [0,1]} \log^+ \|G_t\|_x\quad \text{and}\quad x\mapsto
\sup_{t\in [0,1]} \log^+ \|G_{1-t}\|_{g_t(x)}
\end{align}
are in $L^1(M,\mu)$.
Then there exist real constants $\lambda_1 \geq \cdots \geq \lambda_k$
and a 
decomposition
\[V = \bigoplus_{i=1}^k V_{\lambda_i}\]
by measurable vector subbundles such that, for almost all $m \in M$ 
and all $v \in (V_{\lambda_i})_m \sms \{0\}$, one has
\begin{equation} \label{eq:Oseledeclimit}
\lambda_i = \lim_{t \to \pm\infty} \frac{1}{t} \log \|G_t(v)\|.
\end{equation}
\end{theorem}
\par
We call the values $\lambda_i$, repeated with multiplicity $\dim V_{\lambda_i}$
the set of Lyapunov exponents or {\em Lyapunov spectrum} of $(M,\mu, V, g_t)$.
\par
Suppose that $V$ carries a symplectic structure preserved by the flow. 
Then the Lyapunov exponents are symmetric with respect to zero, i.e. 
$\lambda_{k+1-i} = - \lambda_i$.
The \emph{positive Lyapunov spectrum}
is by definition the first half of the symmetric Lyapunov spectrum.
\par
The $V_{\lambda_i}$ do not change if $\|\cdot\|$
is replaced by another norm of `comparable' size. More precisely,
\begin{lemma} \label{lemma:oseledec-comparable_norms}
 In the above theorem, suppose $M$ is a locally compact topological space and
$\mu$ a regular Borel measure. Let $\|\cdot\|_1$ and $\|\cdot\|_2$ be norms on $V$
(varying measurably) such
that for both \eqref{eq:Oseledec-integrability} is satisfied. Then the 
filtrations and the two sets of Lyapunov exponents coincide.
\end{lemma}
\begin{proof}
 The function $m\mapsto \phi(m) = \sup_{\|v\|_{2,m} = 1} \|v\|_{1,m}$
is measurable. Thus, by Lusin's theorem, there is a compact subset $K$ of
positive measure and a $C = C(K)>0$ such $\phi(m)\leq C$ for all $m\in K$. By
Poincar\'e recurrence, we find for almost all $m$ an increasing sequence of
times $t_1,t_2,\ldots$ with $\lim t_j\to \infty$ such that $g_{t_j}(m) \in K$
for all $j$. 
Now for
almost all $m$ and $v\in V_m\setminus\{0\}$, 
\begin{align*}
\lim_{t \to \infty} \frac{1}{t} \log \|G_t(v)\|_1 &= \lim_{j\to\infty}
\frac{1}{t_j} \log \|G_{t_j}(v)\|_1
 \leq \lim_{j\to\infty} \frac{1}{t_j} \log (C \cdot \|G_{t_j}(v)\|_2)\\
	    &= \lim_{j\to\infty} \frac{1}{t_j} \log \|G_{t_j}(v)\|_2 = \lim_{t
\to \infty} \frac{1}{t} \log \|G_t(v)\|_2
\end{align*}
By symmetry, we also obtain the other inequality, hence
\[\lim_{t \to \infty} \frac{1}{t} \log \|G_t(v)\|_1 = \lim_{t \to \infty}
\frac{1}{t} \log \|G_t(v)\|_2,\]
and this implies the claim.
\end{proof}        
\par
In all the applications below, the manifold $M$ will be the
unit tangent bundle $T^1B$ of a ball quotient $B$, $g_t$ will be the geodesic flow,
which is well-known to be ergodic,
and $\mu$ will be the push-forward of Haar measure on $\PU(1,n)$. $V$ will be
the pullback along $T^1B \to B$ of the vector bundle associated with a polarized
$\RR$-VHS $\VV$, and will be endowed with the Hodge norm. The flat
connection then provides a lift $G_t$ of the geodesic flow $g_t$ on the base
$T^1B$. We first check the integrability condition.
\par
\begin{lemma}  \label{le:Oseledec-integrability}
In this setting, the integrability condition \eqref{eq:Oseledec-integrability}
holds. More precisely, we have
\begin{align}\label{eq:bound-log of hodge norm}
\bigl|\tfrac{\dd}{\dd t} \log \|G_t(v)\|\bigr| \leq \frac{1}{\sqrt{2\kappa}}.
\end{align}
\end{lemma}
\par
\begin{proof}
We first argue that it suffices to show \eqref{eq:bound-log of hodge norm}.
It clearly suffices to bound $\log\|G_t(v)\|$ for every $v$ of norm 1 by a
constant depending continuously on $t\in\RR$. Assuming \eqref{eq:bound-log of
hodge norm}, we have
\[\bigl|\log\|G_t(v)\|\bigr| = \bigl|\int_0^t \tfrac{\dd}{\dd s} \log\|G_s(v)\|
\dd s\bigr| \leq \int_0^t \bigl|\tfrac{\dd}{\dd s} \log \|G_s(v)\|\bigr|\dd s
\leq
t \tfrac{1}{\sqrt{2\kappa}}.\]
%
%
Let $\vartheta\in
T^1B$, and let $\gamma:[0,t] \to B$ be the geodesic with $\gamma'(0) =
\vartheta$. For $v\in V_{\vartheta}$, the lift $G_t(v)$ of the flow is given by
a flat section $w$ of $V$ along $\gamma$ with $w_{\vartheta} = v$. We
will use the following formula, which is proved in \cite[Lemma
2.3]{FMZinvariant} for the case of the Teichm\"uller geodesic flow; the
proof carries over to our situation.
\[\tfrac{\dd}{\dd t}
\psi\bigl(w^{1,0}_{\gamma(t)},w^{1,0}_{\gamma(t)}\bigr)\Bigl|_{t=t_0} = 2\cdot
\Re\Bigl( \psi\bigl( \tilde\tau_{\gamma'(t)}(w^{1,0}_{\gamma(t)}),
w^{0,1}_{\gamma(t)}\bigr)\Bigr)\Bigl|_{t=t_0}.\]
Using $\|w_{\gamma(t)}\|^2 = 4\psi(w_{\gamma(t)}^{1,0},w_{\gamma(t)}^{1,0})$, we
have
\begin{align*}
 \tfrac{\dd}{\dd t} \log \|G_t(v)\|\Bigl|_{t=t_0} &=
      2 \|G_{t_0}(v)\|^{-2} \cdot\ \frac{\dd}{\dd t}
\psi\bigl(w^{1,0}_{\gamma(t)},w^{1,0}_{\gamma(t)}\bigr)\Bigl|_{t=t_0}\\
	&= 4 \|G_{t_0}(v)\|^{-2} \cdot\ \Re\Bigl( \psi\bigl(
\tilde\tau_{\gamma'(t_0)}(w^{1,0}_{\gamma(t_0)}),
w^{0,1}_{\gamma(t_0)}\bigr)\Bigr),
\end{align*}
hence by applying the Cauchy-Schwarz inequality
\[|\psi(x,y)| \leq \tfrac1{\sqrt{2}} \|x\| \cdot \tfrac1{\sqrt{2}} \|y\|\]
on the $(0,1)$ subspace and using the submultiplicativity of operator norm 
$\| \cdot\|_{\rm op}$ on $\Hom(V^{1,0},V^{0,1})$
\begin{align*} 
\bigl|\tfrac{\dd}{\dd t} \log \|G_t(v)\| \Bigl|_{t=t_0}\bigl| 
        &\leq
2 \|G_{t_0}(v)\|^{-2} \cdot \|\tilde\tau_{\gamma'(t_0)}(w^{1,0}_{\gamma(t_0)})\|
\cdot \|w^{0,1}_{\gamma(t_0)}\|\\
	&\leq 2 \|G_{t_0}(v)\|^{-2} \cdot
\|\tilde\tau_{\gamma'(t_0)}\|_{\rm op} \cdot \|w^{1,0}_{\gamma(t_0)}\|\cdot
\|w^{0,1}_{\gamma(t_0)}\|\\
        &\leq \|\tilde\tau_{\gamma'(t_0)}\|_{\rm op}.
\end{align*}
The map $\tilde\tau$ can be identified with the derivative of the period map
$p:\BB^n\to \HH_{\rk\VV}$ associated with $\VV$. We may use the
realization of $\HH_{\rk\VV}$ as symmetric matrices, so 
a tangent vector like $\tilde\tau_{\gamma'(t_0)}$
is also given by a symmetric matrix. Acting by an appropriate
block diagonal element in the symplectic group, we may suppose
that $\tilde\tau_{\gamma'(t_0)}$ is given by a diagonal matrix.
To bound its operator norm it suffices to bound the operator
norm of the individual diagonal entries. Let $r_i: \HH_{\rk\VV} \to \HH$
be the projection on the $i$-th diagonal element and 
$p_i = r_i \circ p: \BB^n\to \HH$. In this notation, 
we want to bound the operator norm of $\dd r_i(\tilde\tau_{\gamma'(t_0)}) = 
(\dd p_i)({\gamma'(t_0)})$.
\par
For this purpose we want to apply a generalization of the 
Schwarz-Pick Lemma (e.g.\ \cite[Theorem~2]{Roy}) to the
composition $p_i$. The ball $\BB^n$ is a 
hermitian manifold of constant negative holomorphic
sectional curvature $k_0$, in fact with our normalization we 
have $k_0= \tfrac{-2}{\kappa}$. Suppose we take the hermitian metric
of constant holomorphic sectional curvature $K_0$ on $\HH$ and let 
$\| \cdot \|_{\HH}$ 
be the associated norm.
In this situation, Royden's theorem states that 
$\|\dd p_i\|_\HH \leq \sqrt{k_0/K_0}$. Now it suffices to check that
for $K_0=-4$ the operator norm on the tangent bundle and the
norm $\| \cdot \|_\HH$  coincide.
\par
To do so, we use the unit disk model $\HH\isom \BB^1$ of hyperbolic space.
The norm $\|\cdot\|_\HH$ is given by the hermitian form
\[h = \frac{2\kappa}{(1-|z|^2)^2} dz\tensor d\bar{z},\qquad \text{where}\ K_0 =
\tfrac{-2}{\kappa}\]
and the  operator norm on $T\BB^1$ is induced
by the isomorphism $T\BB^1 \to \Hom(E^{1,0},E^{0,1})$ coming from the Higgs
field $\tilde\tau = \tilde\tau_{\BB^1}$ of the uniformizing VHS on $\BB^1$.
Here, $E = \BB^1\times\CC^2$ is the trivial bundle with trivial connection
\[\nabla(f_1e_1 + f_2e_2) = df_1\tensor e_1 + df_2 \tensor e_2\]
and indefinite hermitian form $\psi(v,w) = v_1\bar{w_1} - v_2\bar{w_2}$
with respect to the standard basis $\{e_1,e_2\}$. $E^{1,0}$ is the holomorphic
subbundle with fiber $E^{1,0}_{z} = \CC (e_1 + ze_2)$ and $E^{0,1}$ is its
orthogonal complement with fiber $E^{0,1}_z = \CC(\bar{z}e_1 + e_2)$.
The norm on $E^{1,0}$ and $E^{0,1}$ is the one induced by restricting $\psi$
and flipping the sign on $E^{0,1}$, i.e.
\[\|\lambda(e_1+ze_2)\|^2 = |\lambda|^2(1-|z|^2) = \|\lambda(\bar{z}e_1 +
e_2)\|^2.\]

For the identification $T\BB^1 \to \Hom(E^{1,0},E^{0,1})$, we first consider
the map
\[\partial_z \mapsto \bigl(\lambda(e_1+ze_2)
\mapsto \nabla_{\partial_z}(\lambda(e_1+ze_2)) = \lambda_ze_1 +
(\lambda + z\lambda_z)e_2\bigr),\]
where $\lambda_z = \pder[\lambda]{z}$. The orthogonal projection to $E^{0,1}$ is
given by
\begin{align*}
f_1e_1 + f_2e_2 \mapsto &\psi(f_1e_1 + f_2e_2, \bar{z}e_1 +
e_2) \cdot \psi(\bar{z}e_1 + e_2, \bar{z}e_1 + e_2)^{-1} \cdot (\bar{z}e_1 +
e_2)\\
&= \tfrac{1}{|z|^2 -1} (f_1z - f_2) (\bar{z}e_1 + e_2) ,
\end{align*}
and the Higgs field is the composition of these two maps. Thus,
\begin{align*}
 \tilde\tau(\partial_z)(\lambda(e_1+ze_2)) &= \tfrac{1}{|z|^2 - 1} (\lambda_z z
-
(\lambda + z\lambda_z)) (\bar{z}e_1 + e_2) \\
    &= \tfrac{\lambda}{1-|z|^2} (\bar{z}e_1 + e_2)
\end{align*}
Hence, the operator norm of $\tilde\tau(\partial_z)$ is given by
\begin{align*}
\|\tilde\tau(\partial_z)\| &= \frac{\|\sigma(\partial_z)(\lambda(e_1 +
ze_2))\|}{\|(\lambda(e_1 + ze_2)\|} =
\frac{|\lambda|\sqrt{(1-|z|^2)}}{(1-|z|^2)}
 \cdot |\lambda|^{-1}(\sqrt{1-|z|^2})^{-1}\\
    &= \frac{1}{1-|z|^2}
\end{align*}
Therefore, $\tilde\tau$ is an isometry iff $\kappa = \tfrac12$, which proves
the claim.
\end{proof}
\par
The hypothesis of Oseledec's theorem holds verbatim for $\VV_\CC$
and for the summands of a decomposition of the VHS into 
$ \VV_\CC = \oplus \WW_i$, replacing Siegel upper-half space by the respective 
period domain for the polarized VHS.
In particular, it applies to the primitive part $\PP_\CC \subset \VV_\CC$.
We call the corresponding part of the Lyapunov spectrum the
{\em primitive Lyapunov spectrum}.
\medskip

\section{Lyapunov exponents of ball quotients: Generalities}
\label{sec:cylicgeneral}

From now on $B$ will be a {\em ball quotient}, i.e.\ the quotient of $\BB^n$ by some cofinite
discrete subgroup $\Gamma \subset \PU(1,n)$.
In this section we prove the general results announced as Theorem~\ref{thm:main}.
\par
\subsection{Normalization}  \label{sec:normalization}
Recall the construction and the metric of the uniformizing local system $\UU$
from Section~\ref{sec:weigh1VHS}. The following proposition comprises the
statement of Theorem \ref{thm:main} i).
\par
\begin{prop}
 The Lyapunov exponents of a uniformizing polarized VHS $\UU$ are
\[\lambda_1,\lambda_1,\underbrace{0,\dots,0}_{2n-2},-\lambda_1,-\lambda_1\]
 where $\lambda_1 = \tfrac{1}{\sqrt{2\kappa}}$, and the Lyapunov exponents of
any polarized VHS $\VV$ on $B$ are all bounded by $\lambda_1$.
\end{prop}
\begin{proof}
We can lift the whole situation to the universal cover and look at the geodesic
flow on $T^1\BB^n$ acting on $\BB^n\times \CC^{n+1}$ endowed with the metric
from Section~\ref{sec:weigh1VHS}. Moreover, it suffices to compute the Lyapunov
exponents for the point $(0,\tfrac{1}{\sqrt{2\kappa}}e_1)\in T^1\BB^n$ since
$\PU(1,n)$ acts transitively on $T^1\BB^n$ and the measure $\mu_{T^1\BB^n}$ is
invariant under this action. The unit speed geodesic flow starting at
$(0,\tfrac{1}{\sqrt{2\kappa}}e_1)$ is given by
\[g_t\cdot (0,\tfrac{1}{\sqrt{2\kappa}}e_1) =
(\tanh\Bigl(\tfrac{t}{\sqrt{2\kappa}}\Bigr)e_1,
(1-\tanh^2\Bigl(\tfrac{t}{\sqrt{2\kappa}}\Bigr))\tfrac{1}{\sqrt{2\kappa}}e_1)\]
%
%
Since the local system is trivial, the geodesic flow acts trivially on the
fibers of $T^1\BB^n\times \CC^{n+1}$. However, the norm in this bundle varies.
Now
let $\ul{x} = (x_1,\dots,x_{n+1})$ in the fiber over
$(0,\tfrac{1}{\sqrt{2\kappa}}e_1)$. Set $\ul{z}_t =
(1,z_t)$, where $z_t$ is the projection of $g_t\cdot
(0,\tfrac{1}{\sqrt{2\kappa}}e_1)$ to $\BB^n$. By~\eqref{eqn:unif-norm},
\begin{align*}
\|G_t\cdot \ul{x}\|^2 = \|\ul{x}\|_{z_t}^2 &= -2\langle
\ul{x},\ul{x}
\rangle_{1,n} 
+ 4|\langle \ul{x}, \ul{z}_t \rangle_{1,n}|^2 \cdot (\langle
\ul{z}_t,\ul{z}_t 
\rangle_{1,n})^{-1}\\
				        &= -2\langle \ul{x},\ul{x} \rangle_{1,n}
+ 4|x_1\cosh\Bigl(\tfrac{t}{\sqrt{2\kappa}}\Bigr) -
x_{2}\sinh\Bigl(\tfrac{t}{\sqrt{2\kappa}}\Bigr)|^2
\end{align*}
since
\[\langle \ul{z}_t,\ul{z}_t\rangle_{1,n} = 1 -
\tanh^2\Bigl(\tfrac{t}{\sqrt{2\kappa}}\Bigr) =
\cosh^{-2}\Bigl(\tfrac{t}{\sqrt{2\kappa}}\Bigr)\]
and
\[\langle \ul{x}, \ul{z}_t \rangle_{1,n} =
x_1- x_2\tanh\Bigl(\tfrac{t}{\sqrt{2\kappa}}\Bigr).\]
As $\cosh(\alpha) \pm \sinh(\alpha) = e^{\pm\alpha}$,
we see that
\[V^{\lambda_1} = [(1,-1,0,\dots,0),(i,-i,0,\dots,0)]_{\RR}\]
\[V^0 = \{\ul{x}\in \CC^{n,1}| x_1 = x_{2} = 0\}\]
\[V^{-\lambda_1} = [(1,1,0,\dots,0),(i,i,0,\dots,0)]_{\RR}\]
is an orthogonal decomposition of the fiber of $\BB^n\times \CC^{n+1}$ over
$(0,\tfrac{1}{\sqrt{2\kappa}}e_1)$ which realizes the Lyapunov exponents.
\par
The second claim follows from Lemma \ref{le:Oseledec-integrability}, and its
proof. First, we can assume $\VV$ to be an $\RR$-VHS, since we otherwise
consider the real form of $\VV\oplus\ol{\VV}$, which has the same Lyapunov
exponents as $\VV$ (but possibly with different multiplicity). With 
curvature $k_0 = \tfrac{-2}{\kappa}$ on the ball, we have the bound
\[\tfrac{1}{t}\log\|G_t(v)\| =
\tfrac{1}{t} \int_0^t \tfrac{\dd}{\dd s} \log\|G_s(v)\| \dd s \leq
\tfrac{1}{\sqrt{2\kappa}}.\]
\end{proof}

%
\subsection{Duplication}
When working with $\CC$-variations, there are two duplication phenomena of the
Lyapunov spectrum. The first one is obvious from the fact that in a normed
$\CC$-vector space $V$, a vector $v\neq 0$ and $iv$ are $\RR$-linearly
independent and have the same norm. The second one occurs when an $\RR$-VHS
becomes reducible over $\CC$.
\par
\begin{prop} \label{prop: duplication}
Let $\WW_\RR$ be an irreducible direct summand as in Theorem \ref{thm:main} that
splits after tensoring with $\CC$. Then each Lyapunov exponent of
$\WW_{\RR}$ occurs with even multiplicity.
\end{prop}
\begin{proof}
We have $\WW_{\RR}\tensor \CC = \VV \oplus \ol{\VV}$ with a $\CC$-variation
$\VV$. Therefore, the Lyapunov spectrum of $\WW\tensor\CC$ is the union of the
two spectra of $\VV$ and $\ol{\VV}$, and is the spectrum of $\WW_\RR$ with each
exponent occurring with twice the multiplicity. Since $\|\ol{v}\| = \|v\|$ for
each section $v$ of $\VV$, and $G_t\cdot \ol{v} = \ol{G_t\cdot v}$, it follows
that $\VV$ and $\ol{\VV}$ have the same Lyapunov spectrum, and each individual
exponent occurs with even multiplicity because of the first duplication
phenomenon.
\end{proof}

\subsection{Zero exponents} 
This result is probably well-known, it was observed for the
Teichm\"uller geodesic flow in \cite{FMZzero} and we sketch their argument.
\begin{proof}[Proof of Theorem~\ref{thm:main} iii)] First, with respect to the
indefinite hermitian form $\psi$, the Oseledec subspaces
 $V_{\lambda_i}$ are isotropic unless $\lambda_i=0$ and pairwise orthogonal unless
$\lambda_i = - \lambda_j$. Indeed,
\[|Q( v_i, v_j )| \leq c(K) \|v_i\| \|v_j\|\]
on any compact set $K$ of positive measure with a uniform constant depending
only on $K$, and
\[\|G_t\cdot v_i\|\|G_t\cdot v_j\| \sim \exp((\lambda_i
+\lambda_j)t),\] 
which tends to zero for $t \to \infty$ or $t \to -\infty$. Since $g_t$ returns
to $K$ for a sequence of times $t_k \to \infty$ or $t_k\to -\infty$, 
\[Q(v_i,v_j) = Q(G_{t_k}\cdot v_i,G_{t_k}\cdot v_j) \to 0.\]
\par
The key observation is now that an isotropic subspace in a vector space with
an indefinite hermitian form of signature $(p,q)$ has at most dimension
$\min(p,q)$. 
 Since 
both $\bigoplus_{\lambda_i >0} V_{\lambda_i}$ and $\bigoplus_{\lambda_i <0}
V_{\lambda_i}$ are isotropic,
hence of dimension at most $\min(p,q)$, the complement $V_0$ has to be of
dimension at least $|p-q|$.
\end{proof}

\begin{remark}
The key observation of the proof also implies that the  that the Lyapunov
spectrum is symmetric with respect to $0$ \cite[Lemma A.3]{FMZzero}.
Non-degeneracy of the Hodge inner product implies that the Lyapunov spectrum
containing $\lambda_i$ also contains $-\lambda_i$. For $\lambda_i\neq 0$,
$V_{\lambda_i}\oplus V_{-\lambda_i}$ is an orthogonal factor of the Oseledec
decomposition, both of whose summands are isotropic, and hence have the same
dimension. Therefore $\lambda_i = - \lambda_{p+q+1-i}$.
\end{remark}

\subsection{Partial sums are intersection numbers}
We restate Part iv) of Theorem~\ref{thm:main} for holomorphic
sectional curvature $\tfrac{-2}{\kappa}$. Consider a polarized $\RR$-VHS $\WW_\RR$ on
a ball quotient $B = \BB^n/\Gamma$ as in Theorem~\ref{thm:main}~iv).
\begin{theorem}\label{thm:main-with-kappa}
 The positive Lyapunov exponents $\lambda_1,\dots,\lambda_k$ of $\WW_\RR$
satisfy
\begin{equation*}
\lambda_1+\dots +\lambda_k =
\frac{1}{\sqrt{2\kappa}}\cdot
\frac{(n+1)\Chern_1(\cE^{1,0}).\Chern_1(\omega_{\ol
B})^{n-1}}{\Chern_1(\omega_{\ol B})^n}.
\end{equation*}
\end{theorem}
Before we engage in the proof, a few remarks are in order. Almost all 
ingredients are just parallel to 
\cite{kontsevichzorich}, see also \cite{forni02}, explained in more detail in \cite{bouwmoel} for
the curve case and \cite{ekz} in general. There are a few modifications however. 
\par
First, we remark that the formula \eqref{eq:HodgeviaOmega} is valid not only for (flat) surfaces, 
but for general weight one VHS. Second, the Laplacian on the $n$-ball
replaces the Laplacian on the one-dimensional ball without major
difficulties. Finally, at a crucial step we use that ball quotients
are K\"ahler-Einstein manifolds to trade the first Chern class of the
cotangent bundle for the class of the K\"ahler metric.
\par 
To start with the details, let
$\Omega = \tfrac{1}{k!}\wedge^kQ \in (\bigwedge^{2k}\WW_{\CC})^*$. Consider a
local section $L$ of $\bigwedge^k\WW$ over an open set $U$ that is given by a
decomposable vector $L = v_1 \wedge \cdots \wedge v_k$.
Such a vector is called Lagrangian, if the symplectic form $Q( v_i, v_j
) =0$
for all pairs $(i,j)$. For a Lagrangian decomposable vector we claim that
the Hodge norm can be calculated (see \cite{ekz} and in more detail in
\cite{GrHub}) by the formula
\begin{equation} \label{eq:HodgeviaOmega}
 \| L\|^2 = \frac{|\Omega(\omega_1\wedge\cdots\wedge\omega_k \wedge
v_1 \wedge\cdots\wedge v_k)| \cdot |\Omega(v_1\wedge\cdots\wedge v_k \wedge
\bar\omega_1\wedge\cdots\wedge\bar\omega_k)|}{|\Omega(\omega_1\wedge\cdots
\wedge \omega_k \wedge \bar\omega_1\wedge\cdots\wedge\bar\omega_k)|}, 
\end{equation}
where $\omega_1,\ldots,\omega_k$ is a local basis of $\cW^{1,0}$. Note 
that the right hand side does not depend on the choice of this basis. 
\par 
\begin{lemma} \label{le: Laplace kills flat contributions} Given a decomposable
Lagrangian section $L$ of $\bigwedge^k\WW$
as above over some open subset $U$ of $B$, we have on $U$ the equality of
functions
\[ \Phi := \Delta_\hyp(\log\|L\|) = -\tfrac{1}{2}\Delta_\hyp \log \det
(\psi(\omega_i, \omega_j))_{i,j}.\]
\end{lemma}
\par
\begin{proof}
The proof is the same as in \cite{ekz} and the other sources above.
The only thing we need to use is that for higher-dimensional balls, too,
the hyperbolic Laplacian is a sum of $\frac{\partial^2}{\partial_i
\ol{\partial}_j}$-derivatives  and of $\frac{\partial^2}{\ol{\partial}_i
\partial_j}$-derivatives by \eqref{eqn:Laplacian form on Bn}.
Alternatively, one can use the proportionality \eqref{eqn: Laplace_and_ddbar} of
$\Laplace_{\hyp}(\log\|L\|)\dvol_{\BB^n}$ and 
$\partial\bar\partial\log\|L\| \wedge \omega_\hyp^{n-1}$ 
and the fact that $\partial\bar\partial f = 0$ for any holomorphic or 
antiholomorphic function $f$.
\end{proof}
\par
Since both sides of the equation do not depend of the choice of the
basis we may consider $\Phi_k$ as a function on the whole of $B$.

\begin{proof}[Proof of Theorem~\ref{thm:main-with-kappa}] We first set up
some notations. Let $T^1B$ be the unit tangent bundle, and let $\Gr_k(B)$ be the
Lagrangian Grassmannian bundle over $B$, whose fiber over $x\in B$ consists
of the $k$-dimensional Lagrangian $\RR$-vector subspaces of
$\WW$. Set \[\Gr_k(T^1B) = T^1B \times_B \Gr_k(B).\]
The Lagrangian Grassmannian is a homogeneous space, thus it carries a natural
measure $\gamma$. The bundle measure $\mu_{\Gr_k(T^1B)}$ is then the
product measure of $\gamma$ and the measure $\mu_{T^1B}$ on
the base. We assume the measure $\gamma$, and the measure on the fiber of $T^1B$
to be normalized to have area 1. Denote again $G_t:\Gr_k(T^1B)\to \Gr_k(T^1B)$
the lift of the geodesic flow $g_t:T^1B\to T^1B$, and let $u:\BB^n\to B$ be the
universal covering map. If $\vartheta\in T^1B$, then we denote the coordinate on
the base by $x(\vartheta)\in B$, and a lift to the universal cover by $\tilde
x(\vartheta)$.
\par


The core of the proof is the following chain of equalities. After averaging
over the whole space, we introduce another average over the unit tangent
space at each point. Then we interchange the integral and the limit;
this is possible, since the logarithmic derivative of the Hodge norm is bounded
above by Lemma \ref{le:Oseledec-integrability}. Next, in order to apply Lemma
\ref{lem:Forni}, we pass
to the universal cover. Then we use Lemma~\ref{le: Laplace kills flat
contributions} to get rid of the dependence on $L$. Then we interchange once
more limit and integral. Next, we go back to $B$ and decompose the fiber bundle
$T^{\|\cdot\| \leq t} B$ along the fibers of
\[\exp:T^{\|\cdot\| \leq t} B \to B. \]
The preimage of $z\in B$ under
this map 
has volume equal to the one of $\BB^n_t(\tilde z)$, where $\tilde z$ is a lift
of $z$ to $\BB^n$, since
\begin{align*}
\exp^{-1}(z) &= \{\vartheta \in T^{\|\cdot\| \leq t} B\mid \exp(\vartheta) = z\}
\\
	     &= \{\tilde\vartheta \in T^{\|\cdot\| \leq t} \BB^n\mid
\exp(\tilde\vartheta) = \tilde z\}\\
	     &= \bigcup_{0\leq \tau \leq t} \bigcup_{d_{\hyp}(\tilde z, y) =
\tau} \{y\}
\end{align*}
Finally, the factor
\[c_1(t) =
\frac{\sqrt{2\kappa}}{2n}\tanh\biggl(\frac{t}{\sqrt{2\kappa}}\biggr)\cdot
\frac{1}{\vol(\BB^n_t)}\]
will drop out once we apply Lemma \ref{lem:Forni}. In the last line we use that
\[\frac{1}{T}\int_0^T \vol(\BB^n_t)\cdot c_1(t) \dd t =
\frac{\kappa}{n}\cdot\frac{\log(\cosh(\tfrac{T}{\sqrt{2\kappa}}))}{T} \to
\frac{\kappa}{n\cdot \sqrt{2\kappa}}\]
as $T\to \infty$.


{\allowdisplaybreaks 
\begin{align*}
\vol(B)\sum_{i=1}^k \lambda_i  
     &= \int_{\Gr_k(T^1B)} \lim_{T \to \infty} \frac{1}{T}\int_{0}^T 
	\frac{\dd}{\dd t} \log \|G_t(\vartheta, L)\| 
	\dd t \dd\mu_{\Gr_k(T^1B)}(\vartheta,L)\\ 
%
%
     &= \int_{\Gr_k(T^1B)} \frac{1}{\sigma_{2n-1}} \int_{T^1_{x(\vartheta)}B}
\lim_{T \to \infty}  \frac{1}{T} \int_{0}^T \frac{\dd}{\dd t} \log
\|G_t(\theta, L )\|  &&&\\ 
     & \qquad\qquad\qquad\qquad\qquad\qquad\qquad\qquad\qquad 	\dd t
\dd\sigma(\theta) \dd\mu_{\Gr_k(T^1B)}(\vartheta,L)\\
%
    & = \int_{\Gr_k(T^1B)} \lim_{T \to \infty}  \frac{1}{T} \int_{0}^T 
	\frac{1}{\sigma_{2n-1}} \frac{\dd}{\dd t} \int_{T^1_{x(\vartheta)}B} 
\log \|G_t(\theta, L )\|  \\
& \qquad\qquad\qquad\qquad\qquad\qquad\qquad\qquad\qquad 
\dd\sigma(\theta) \dd t \dd\mu_{\Gr_k(T^1B)}((x,\vartheta),L) \\
%
%
    & = \int_{\Gr_k(T^1B)} \lim_{T \to \infty}  \frac{1}{T} \int_{0}^T 
	\frac{1}{\sigma_{2n-1}} \frac{\dd}{\dd t}
\int_{T^1_{\tilde x(\vartheta)}\BB^n} 
\log \|G_t(u(\tilde\theta), L )\|  \\
& \qquad\qquad\qquad\qquad\qquad\qquad\qquad\qquad\qquad 
\dd\sigma(\tilde\theta) \dd t \dd\mu_{\Gr_k(T^1B)}(\vartheta,L) \\
%
%
    & = \int_{\Gr_k(T^1B)} \lim_{T \to \infty}  \frac{1}{T}\int_{0}^T c_1(t)
\int_0^t \int_{T^1_{\tilde x(\vartheta)}\BB^n} 
	\Delta_{\hyp} \log \|G_{\tau}(u(\tilde\theta),L)\|\\
    & \qquad\qquad\qquad\qquad\qquad\qquad\qquad\qquad\qquad 
\dd\sigma(\tilde\theta) \dd\tau  \dd t \dd\mu_{\Gr_k(T^1B)}(\vartheta,L) \\ 
%
%
    & = \int_{B} \lim_{T \to \infty}  \frac{1}{T}\int_{0}^T c_1(t)
\int_0^t \int_{T^1_{\tilde x}\BB^n} 
	-\tfrac{1}{2} \Delta_{\hyp} \log
\det(\psi(\omega_i,\omega_j))_{i,j}(u(\exp_{\tilde x}(\tau\tilde
\theta)))\\
    & \qquad\qquad\qquad\qquad\qquad\qquad\qquad\qquad\qquad 
\dd\sigma(\tilde\theta) \dd\tau \dd t
\dd\mu_{B}(x) \\ 
%
%
    & = \lim_{T \to \infty}  \frac{1}{T}\int_{0}^T c_1(t)
\int_{B} \int_0^t \int_{T^1_{\tilde x}\BB^n} 
	-\tfrac{1}{2} \Delta_{\hyp} \log
\det(\psi(\omega_i,\omega_j))_{i,j}(u(\exp_{\tilde x}(\tau\tilde
\theta)))\\
    & \qquad\qquad\qquad\qquad\qquad\qquad\qquad\qquad\qquad 
\dd\sigma(\tilde\theta) \dd\tau 
\dd\mu_{B}(x) \dd t\\ 
%
%
    & = \lim_{T \to \infty}  \frac{1}{T}\int_{0}^T c_1(t)
\int_{T^{\|\cdot\| \leq t} B}\Phi(\exp_{x(v)}(v)) \dd\mu_{T^{\|\cdot\| \leq t}
B}(v) \dd t\\ 
%
%
    & = \lim_{T \to \infty}  \frac{1}{T}\int_{0}^T c_1(t)
\int_{B}\int_{\exp^{-1}(z)} \Phi(z) \dd\mu_{\exp^{-1}(z)}\dd\mu_B(z) \dd t\\ 
%
%
    & = \int_{B} \Phi(z) \dd\mu_B(z) \cdot \lim_{T \to \infty} 
\frac{1}{T}\int_{0}^T c_1(t)
\int_{\exp^{-1}(z)} \dd\mu_{\exp^{-1}(z)} \dd t\\ 
%
    & = \int_{B} \Phi(z) \dd\mu_B(z) \cdot \lim_{T \to \infty} 
\frac{1}{T}\int_{0}^T c_1(t)
\vol(\BB^n_t) \dd t\\ 
%
%
%
%
%
    &= \int_B \frac{\kappa}{n\sqrt{2\kappa}}\Phi(x) \dd \mu_B(x).
%
%
\end{align*}
}
\par
The measure $\dd\mu_B$ is given by integrating against the volume form
$\dvol_B$, which is the image of $\dvol_{\BB^n}$ via the universal cover $u:
\BB^n \to B$. Choose a fundamental domain $\FDom\subset \BB^n$ for the action of
$\Gamma$. Let $F:B\to \RR$ be the function
\[F = \log \det(\psi(\omega_i,\omega_j))_{i,j}.\]
Then Lemma \ref{le:omegahochnminus1} yields
\begin{align*}
 \frac{\kappa}{n\sqrt{2\kappa}} \int_B \Phi(x) \dd\mu_B(x) &=
\frac{\kappa}{n\sqrt{2\kappa}} \int_{\FDom} -\tfrac{1}{2}\Laplace_{\hyp}(F\circ
u) \dvol_{\BB^n}\\
					     &= -\frac{i\kappa}{n!\sqrt{2\kappa}}
 \int_{\FDom} \partial\bar\partial (F\circ u) \wedge \omega_{\hyp}^{n-1}.\\
\end{align*}

A ball quotient $B$ is a K\"ahler-Einstein manifold, therefore the first Chern
class and the K\"ahler class are proportional. Moreover, as was remarked by
B. Hunt \cite[Lemma 1.8]{Hunt00}, the K\"ahler-Einstein metric also computes the
logarithmic Chern class in case $B$ is not compact. The proportionality
constant is determined from the holomorphic section curvature.
By \cite[p. 223]{Huybrechts05},
\[\Chern_1(\omega_{\ol B}) = [-1/(2\pi) \cdot
\Ric(B,g_{\hyp})]\]
and by \cite[p. 168]{KobayashiNomizu2}
\[\Ric(B,g_{\hyp}) = \tfrac{1}{2}(n+1)k_0 \omega_{\hyp},\]
where $k_0 = \tfrac{-2}{\kappa}$ is the holomorphic sectional curvature.
Thus,
\[\Chern_1(\omega_{\ol B}) = \frac{(n+1)}{2\pi\kappa} [\omega_{\hyp}].\]
As the metric on $\cW^{1,0}$ is good in
the sense of Mumford (by Schmid's $\SL_2(\RR)$-orbit theorem, e.g.\ \cite[Theorem~5.21]{catkapsch}), 
it computes the first Chern class
of Mumford's extension of $\cW^{1,0}$ to $\ol B$, which is the same
as the Deligne extension, see e.g.\ \cite[Lemma~3.4]{MVZ}. Consequently, 
\[\Chern_1(\cE^{1,0}) = \tfrac{i}{2\pi} [\Theta(\wedge^k\cW^{1,0})],\]
where the curvature $\Theta(\wedge^k\cW^{1,0})$ is given by
$-\partial\bar\partial F$. Altogether we obtain
\begin{align*}
\vol(B)\sum_{i=1}^k \lambda_i &= \frac{1}{\sqrt{2\kappa}} \cdot
\frac{i\kappa}{n!}  \cdot \frac{2\pi}{i} \cdot
\biggl(\frac{2\pi\kappa}{(n+1)}\biggr)^{n-1}\cdot
\Chern_1(\cE^{1,0}). \Chern_1(\omega_{\ol B})^{n-1}[B]\\
			      &=
\frac{1}{\sqrt{2\kappa}}\cdot \frac{(2\pi\kappa)^{n}}{n! (n+1)^{n-1}}
\Chern_1(\cE^{1,0}). \Chern_1(\omega_{\ol B})^{n-1}[B]
\end{align*}
On the other hand,
\[[\dvol_{B}] = \frac{1}{n!}[\omega_{\hyp}]^n = \frac{(2\pi\kappa)^n}{n!(n+1)^n}
\Chern_1(\omega_{\ol B})^n.\]
This finishes the proof.
%
\end{proof}
\medskip

\section{Invariants of modular embeddings and commensurability invariants} \label{sec:commensurability}

The {\em trace field} of a lattice $\Gamma \subset \PU(1,n)$ is defined
to be $E = \QQ(\tr({\rm Ad}(\gamma)), \gamma \in \Gamma)$, where $\rm Ad$ is 
the adjoint representation. The trace field is
an invariant of the commensurability class of $\Gamma$ (e.g.\ 
\cite[Proposition~12.2.1]{DeligneMostow86}).
\par
The aim of this section is to show that, among lattices $\Gamma \subset
\PU(1,n)$ (or rather $\SU(1,n)$ for that matter) that admit a modular embedding, 
the Lyapunov spectrum 
and the relative orbifold Euler characteristics are also commensurability invariants. 
In fact, given a modular embedding, we can define several invariants using 
intersection theory and we have shown in Theorem~\ref{thm:main} resp.\ we will show 
in Corollary~\ref{cor:compute-otherinvariant} later that the two 
invariants mentioned above, the Lyapunov spectrum 
and the relative orbifold Euler characteristics, indeed arise in that way.
At the end of this section we give another proof of the commensurability 
invariance of the Lyapunov spectrum without referring to modular embeddings. 
%
\par
\begin{remark} {\rm
From the point of view of Lyapunov exponents, arithmetic
lattices in $\PU(1,n)$ are less interesting, as we will show 
in the next section.
\par
The construction of {\em non-arithmetic} 
lattices in $\PU(1,n)$ is a long-standing
challenge. Besides the cyclic coverings studied below there are several
techniques to construct (non-arithmetic) lattices in $\PU(1,n)$, notably the original
construction using complex reflection groups by Mostow (\cite{mostowremarkable}). By the
work of Sauter (\cite{sauter}) and Deligne-Mostow (\cite{delcommen})
all the presently known non-arithmetic lattices in $\PU(1,n)$ 
turn out to be commensurable to lattices arising from cyclic coverings.  
See \cite{parkersurvey} for a recent survey on commensurability results
and \cite{derauxppcensus} for other constructions that by numerical
evidence are very likely to give finitely many other non-arithmetic
lattices.
}\end{remark}
\par
\subsection{Modular embeddings}
In the following, we suppose that $\Gamma\subset \SU(1,n) \cap \GL_{n+1}(F)$
for some number field $F$, which we take to be Galois over $\QQ$.
In particular, the Zariski closure of $\Gamma$, which by 
Borel density is all of $\SU(1,n)$,  is defined over $F$, as is the
adjoint representation. This entails $E\subseteq F$.
\par
We say that $\Gamma$ admits a modular embedding, 
if for any $\sigma\in \Gal(F/\QQ)$
\begin{enumerate}[a)]
 \item the Galois conjugate group ${\Gamma}^\sigma$ fixes a hermitian form of signature
$(p(\sigma),q(\sigma))$, where we assume that $p(\sigma) \geq q(\sigma)$, and
 \item there is a holomorphic map
$\varphi^\sigma: \BB^n \to \BB_{p(\sigma),q(\sigma)}$ to the symmetric space
of $\PU(p(\sigma),q(\sigma))$ that is equivariant with respect to
the action of $\Gamma$ on the domain and $\Gamma^\sigma$ on the
range of $\varphi^\sigma$.
\end{enumerate}
We will show below that a lattice as above admits at most one
modular embedding. Namely, the signature $(p(\sigma),q(\sigma))$ 
is uniquely determined by ($\Gamma$ and)
$\sigma$ by Lemma~\ref{le:sign} below, and
the map $\varphi^\sigma$ is unique by Theorem~\ref{thm:modemb_unique}.
\par
In passing, we note that 
$p(\sigma)=0$ or $q(\sigma)=0$ implies that
$\BB_{p(\sigma),q(\sigma)}$ is a point, 
thus $\varphi^\sigma = \mathit{const}$ 
trivially does the job.
\par
Our notion of modular embedding is essentially
the same as in \cite{CoWoMod}. Modular embeddings are also discussed in
\cite[Sec. 10]{mcmullenhodge}; however one should be aware that the target space
of the period map used there is not equal to the period domain for signatures
$(p,q)$ with $p+q = n+1$, $|p-q|<n-1$. 
\par
\begin{lemma} \label{le:sign} Let $\Gamma$ and $\sigma$ be as above. Then
$\Gamma^\sigma$ preserves a hermitian form for at most one 
signature $(p(\sigma),q(\sigma))$ with $p(\sigma) \geq q(\sigma)$. 
\end{lemma}
\par
\begin{proof}
Suppose the contrary holds. Then the Zariski closure of $\Gamma^\sigma$
is contained in the intersection of two special unitary groups $\SU(p_1,q_1)$
and 
$\SU(p_2,q_2)$ with $p_i + q_i = n+1$, $p_i \leq q_i$ but $p_1 \neq q_1$. 
These two groups have both the (real) dimension $(n+1)^2 -1$ and since they are not
equal, their intersection has strictly smaller dimension. The dimension of 
the Zariski closure is invariant under Galois
conjugation and the Zariski closure of $\Gamma$ over $\RR$ 
is all of $\SU(1, n)$, hence of dimension
$(n + 1)^2 - 1$. This is a contradiction.
\end{proof}
\par
Now we prove the main result about modular embeddings. 
First, we show that we can pass to a finite index 
subgroup to realize the trace field as field of traces 
with respect to the standard embedding.
\par
\begin{lemma} \label{le:righttracefield}
There exists a finite index subgroup $\Gamma_1 \subset \Gamma \subset \SU(1,n)$,
such
that the trace field $E = \QQ(\tr({\rm Ad}(\gamma)), \gamma \in \Gamma)$
for the adjoint representation of $\Gamma$ coincides with the
trace field $\QQ(\tr(\gamma), \gamma \in \Gamma_1)$ for the standard
representation of $\Gamma_1$. 
\end{lemma}
\par
\begin{proof} By the argument given in \cite[Appendix~2]{mcreynolds} it suffices
to take $\Gamma_1$ as the kernel of the map from $\Gamma$ to its maximal abelian
quotient of exponent $n+1$. 
\end{proof}
\par
\begin{theorem} \label{thm:modemb_unique}
Suppose $\Gamma \subset \SU(1,n)$ is a lattice with Galois conjugate
$\Gamma^\sigma \subset \SU(p(\sigma),q(\sigma))$ and we are given two maps $\varphi_i: \BB^n \to
\BB_{p(\sigma),q(\sigma)}$, $i=1,2$ equivariant with respect to the action of
$\Gamma$ and $\Gamma^\sigma$ on domain and range. Then $\varphi_1 = \varphi_2$.
\par
If $\tau$ fixes $E$ and $\varphi_1$ is equivariant with respect
to $\Gamma$ and $\Gamma^\sigma$ and $\varphi_2$ is equivariant with respect
to $\Gamma$ and $\Gamma^{\sigma\tau}$, then the signatures of the hermitian forms
preserved by $\Gamma^\sigma$ and $\Gamma^{\sigma\tau}$ coincide and, moreover, 
there exists an automorphism $M$ of $\BB_{p(\sigma),q(\sigma)}$ such that $M \circ \varphi_1 = \varphi_2$.
\end{theorem}
\par
Note that this
statement is somewhat stronger than usual rigidity results for VHS (see 
e.g.~\cite[Corollary 12]{peterssteenmonodromy}) that require a priori that
the two period maps coincide at one point.
\par
\begin{proof} The case  $p(\sigma)=0$ or $q(\sigma)=0$ is trivial, since
the range of $\varphi_i$ is just a point in this case.
\par
Assume first that $\varphi_i$ is constant.
Then $\Gamma^\sigma$ is entirely contained in the
stabilizer of a point in $\SU(p(\sigma),q(\sigma))$. Thus, the 
Zariski-closure of $\Gamma^\sigma$ has strictly smaller dimension than 
$\SU(p(\sigma),q(\sigma))$. But on the other hand, 
this dimension is the same as the dimension of the Zariski-closure of
$\Gamma$, which is $\SU(1,n)$, and we arrive at a contradiction.
Thus, both $\varphi_i$ are not constant.
\par
We consider the modular embeddings $\varphi_i$ as maps to the 
Harish-Chandra realization of  $\BB_{p(\sigma),q(\sigma)}$
as a bounded symmetric domain. Being composed of bounded holomorphic functions, 
each $\varphi_i$ has a boundary map $\varphi_i^*: \partial \BB^n \to 
\ol{\BB}_{p(\sigma),q(\sigma)}$ defined almost everywhere on $\partial \BB^n$ 
and given by $\varphi_i^*(x) = \lim_{n\to \infty} \varphi_i(\xi_n)$
for any sequence of points $\xi_n \to x$ that stay in an angular sector $D_\alpha(x)$.
\par
Since $\Gamma$ is a lattice, it is in particular a discrete subgroup 
of divergence type, and hence 
the set points of approximation in $\partial \BB^n$ (those that
can be approximated by a sequence in $\Gamma \cdot z$ that stays
in a angular sector $D_\alpha$) is of full Lebesgue measure. Consequently, 
we obtain $\varphi_i^*(x) = \lim \varphi(\gamma_n (x_0))$ 
for some $x_0 \in \BB^n$ and a sequence of $\gamma_n \in \Gamma$.
We may replace $x_0$ by any other point $y_0 \in \BB^n$ and still
$x = \lim \gamma_n (y_0)$, since the hyperbolic distances 
$d(\gamma_n (y_0), \gamma_n(x_0)) = d(x_0, y_0)$ are unchanged.
\par
We claim that the boundary maps agree, i.e.\ $\varphi_1^* = \varphi_2^*$ on $\partial \BB^n$.
For this purpose we first show that almost everywhere $\varphi_i^*$ maps $\partial \BB^n$ to
$\partial \BB_{p(\sigma),q(\sigma)}$ (compare to \cite[Lemma~2.2]{Shiga}). 
Since $\varphi_i$ is non-constant, there exist
$x_0, y_0 \in \BB^n$ with $\varphi_i(x_0) \neq \varphi_i(y_0)$. This implies
\[0<d(\varphi_i(x_0),\varphi_i(y_0)) = d(\gamma_n^\sigma(\varphi_i(x_0)),\gamma_n^\sigma(\varphi_i(y_0)))
= d(\varphi_i(\gamma_n(x_0)),\varphi_i(\gamma_n(y_0))).\]
If $x = \lim \gamma_n(x_0) \in \partial \BB^n$ is a point of approximation, taking the limit $n \to \infty$
the inequality gives a contradiction if $\varphi^*_i(x)$ lies in the interior of 
$\BB_{p(\sigma),q(\sigma)}$. 
This argument also shows that for any such approximation $x = \lim \gamma_n(x_0)$
the sequence $\gamma_n^\sigma(\varphi_i(x_0))$ goes to the boundary of 
$\BB_{p(\sigma),q(\sigma)}$. The limit point of such a sequence is known to exist since
$\varphi_i^*$ is well-defined and, 
by the same argument as above in $\BB^n$, it only depends on the group elements $\gamma_n^\sigma$, 
not on the starting point. Consequently, 
\[\varphi^*_1(x) = \lim \varphi_1(\gamma_n (x_0)) = 
 \lim \gamma_n^\sigma(\varphi_1(x_0)) = \lim \gamma_n^\sigma(\varphi_2(x_0)) = \varphi^*_2(x), \]
as we claimed.
\par
Now, the Cauchy integral formula is still valid for these boundary maps
(\cite[Theorem~11.32 and Theorem 17.18]{RudinComplex}), since we may also use 
radial limits to define $\varphi_i^*$. These theorems imply that 
the maps $\varphi_1$ and $\varphi_2$ agree.
\par
For the proof of the second statement we may pass to a subgroup of
finite index and thus we suppose that $\Gamma$ itself satisfies the conclusion
of Lemma~\ref{le:righttracefield}. By this assumption and definition $\tau$ fixes
all traces of elements in $\Gamma$ and consequently the two representations
$\Gamma^{\sigma\tau}$ and $\Gamma^{\sigma}$ are isomorphic (e.g.\ \cite[Corollary~XVII.3.8]{LangAlgebra}),
in particular their signatures agree. 
If $M$ is the matrix that provides the isomorphism, i.e.\ $M \gamma^{\sigma} M^{-1} = \gamma^{\sigma\tau}$, 
then both $\varphi_2$ and $M \circ \varphi_1$ are equivariant with respect
to $\Gamma$ and $\Gamma^{\sigma\tau}$. By the first part these two maps agree.
\end{proof}
\par
In the situation considered everywhere in the sequel, the existence
of a modular embedding is no obstruction thanks to the following 
proposition,
but in general the existence of modular
embeddings is a hard problem, even for one-dimensional ball quotients,
i.e.\ for Fuchsian groups.
\par
\begin{prop}\label{prop:modemb_iff_Q-VHS} 
$\Gamma$ arises as the monodromy group of a uniformizing 
sub-local system $\WW$ in a polarized $\QQ$-VHS $\VV$ of weight 1
if and only if $\Gamma$ admits a modular embedding.
\end{prop}
\par
\begin{proof}
Suppose that $\Gamma$ is the monodromy group of a uniformizing sub-local system of $\VV$.
By Deligne's semisimplicity theorem, $\VV_\CC$ decomposes into a direct sum of irreducible
$\CC$-VHS, one of which is $\WW$. Since $\VV$ is defined over $\QQ$, all Galois conjugate 
local systems $\WW^\sigma$ appear in the decomposition of $\VV_\CC$. After tensoring with $\CC$, 
these $\WW^\sigma$ are polarized $\CC$-VHS of weight 1, thus their monodromy representation 
fixes an indefinite hermitian form of some signature $(p(\sigma),q(\sigma))$, and their period map
$\varphi^\sigma$ is an equivariant map $\BB^n \to \BB_{p(\sigma),q(\sigma)}$ as desired.
\par
Conversely, let $\Gamma$ have a modular embedding. Each group homomorphism $\Gamma\to \Gamma^\sigma$
determines a $\CC$-local system $\WW^\sigma$. These carry a polarized VHS of weight 1 induced by the 
map $\varphi^\sigma$. The direct sum $\bigoplus_{\sigma\in \Gal(F/\QQ)} \WW^\sigma$ is a 
polarized VHS that is defined over $\QQ$.
\end{proof}
\par

\subsection{Commensurability invariants} \label{sec:invs}
Fix a lattice $\Gamma\subset \SU(1,n)\cap \GL_{n+1}(F)$ with modular embedding.
This determines a Hodge structure on the corresponding local system $\WW^\sigma$
on the ball quotient $B = \BB^n/\Gamma$. We focus on the case of signature $(1,n)$.
 The subbundle of holomorphic forms $\cE^{1,0}$ is thus of rank one. All the numbers
$\bigl({\Chern_1(\cE^{1,0})^a.\Chern_1(\omega_{\ol
B})^{b}}\bigr)/{\Chern_1(\omega_{\ol B})^n}$
for $(a,b)\in \NN_0^2$ with $a+b=n$
are invariants of the modular embedding. By homogeneity, they are also
unchanged under passage from $\Gamma$ to a subgroup of finite index. 
By the preceding Theorem~\ref{thm:modemb_unique} there is no choice involved in 
fixing the modular embedding, and these numbers only depend on the coset of $\sigma$ in
$\Gal(F/\QQ)/\Gal(F/E)$. Altogether, this yields the following.
\par
\begin{cor} \label{cor:comminvariants}
For any $(a,b)\in \NN_0^2$ with $a+b=n$ and
$\sigma \in \Gal(F/\QQ)/\Gal(F/E)$ the ratios 
\[\frac{\Chern_1(\cE^{1,0})^a.\Chern_1(\omega_{\ol
B})^{b}}{\Chern_1(\omega_{\ol B})^n}, \quad \quad \cE^{1,0} =
\cE^{1,0}(\WW^{\sigma})\]
are commensurability invariants of $\Gamma$.
\end{cor}
\par
We end this section with a proof of this corollary for the special case $a=1$, i.e.\
for Lyapunov exponents, not relying on  Theorem~\ref{thm:modemb_unique}, but
rather on the definition of Lyapunov exponents.
\par
\begin{prop} \label{prop:lyapfinitindex}
In the situation of Theorem~\ref{thm:main}, passing from $\Gamma$ to a 
subgroup $\Gamma'$ of finite index and to the pullback VHS does
not change the Lyapunov spectrum.
\end{prop}
\par
\begin{proof}
Let $\cF \subset \BB$ be a fundamental domain for $\Gamma$ and $\cF' \supset \cF$ be
a fundamental domain for $\Gamma'$. Given a starting point $x \in M = T_1\BB^n$, 
for all times $t$ such that $g_t(x)$ is $\Gamma'$-equivalent to a point in $\cF$, 
the quantity $\frac{1}{t}\log\|g_t(v)\|$ appearing on the right of 
\eqref{eq:Oseledeclimit} is the same for both situations. Since the geodesic flow is
ergodic, this set of times is cofinal and thus the limit in \eqref{eq:Oseledeclimit}
is the same in both situations.
\end{proof}
\par
\begin{defn}
Let $\Gamma$ be a lattice admitting a modular embedding, and let $\WW$
be the uniformizing VHS on $\BB^n/\Gamma$. The \textit{Lyapunov spectrum} of
$\Gamma$ is the set of different Lyapunov exponents in the Lyapunov spectrum of
the $\QQ$-VHS
\[\PP(\Gamma) = \bigoplus_{\sigma\in \Gal(F/\QQ)}\WW^\sigma.\]
\end{defn}
\par
The definition might a priori depend on the choice of the modular embedding. We
show that this is not the case.
\par
\begin{theorem} \label{thm:lyap_spec_comminvariant}
The Lyapunov spectrum of a lattice $\Gamma$ in $\PU(1,n)$
that admits a modular embedding is well-defined. It is a 
commensurability invariant among all lattices in $\PU(1,n)$ 
that admit a modular embedding.
\end{theorem}
\par
\begin{proof} The Lyapunov spectrum is determined by the choice of
the geodesic flow on the base $B$, the Haar measure $\mu$ on $B$, 
the local systems with monodromy group $\Gamma^\sigma$ and a norm
satisfying the growth conditions in Theorem~\ref{thm:oseledec}. Any choice of
a modular embedding defines a Hodge structure and thus a corresponding Hodge
norm. But the Lyapunov spectrum does not depend on this choice by
Lemma~\ref{lemma:oseledec-comparable_norms}.
\end{proof}
\par
Note that for a lattice $\Gamma$ admitting a modular embedding, the $\QQ$-VHS 
$\PP(\Gamma)$ is equal to its primitive part. If on the other hand,
$\VV$ is a $\QQ$-VHS on $\BB^n/\Gamma$ containing a primitive part $\PP$, then
the uniqueness of the uniformizing sub-VHS and the above theorem yield that the
set of Lyapunov exponents of $\PP$ (forgetting multiplicities) is equal to the
Lyapunov spectrum of $\Gamma$. 
\par
\section{Arithmeticity} \label{sec:arithmeticity}

Recall that a lattice $\Gamma \subset \PU(1,n)$ is called {\em arithmetic}, if
$\Gamma$ is commensurable to $\phi(G_\ZZ)$, the image of the integral points in some 
linear algebraic group $G_\RR$, admitting a continuous surjective homomorphism $\phi: G_\RR \to
\PU(1,n)$
with compact kernel. 
\par
Let $\PP$ be a $\QQ$-irreducible $\QQ$-VHS. Note that a
uniformizing $\CC$-sub-VHS $\WW$ of signature $(1,n)$ 
of $\PP_\CC$ is defined over
$\RR$ if and only if $n=1$. We say that the Lyapunov spectrum of $\PP$ 
is {\em maximally degenerate}, if it contains twice the Lyapunov exponents 
$1$ and $-1$ for $n>1$ (resp.\ once for $n=1$) and all the other 
Lyapunov exponents are zero. 
\par
\begin{prop} \label{prop:arithmetic}
If $\Gamma\subset\PU(1,n)$ is an arithmetic lattice, then the 
Lyapunov spectrum is 
maximally degenerate. 
\end{prop}
\par
\begin{proof} Suppose that $\Gamma$ is arithmetic. Then 
$\WW$ (and hence $\overline{\WW}$) is uniformizing, and all
other conjugates of $\WW$ are unitary (\cite[12.2.6]{DeligneMostow86}). 
In particular, $\Gamma$ admits a modular embedding with $\varphi^\sigma$ 
being either the identity or constant. For all conjugates of $\WW$ 
the Lyapunov spectrum is determined by 
Theorem~\ref{thm:main}~i) and iii)
and this is precisely the content of the notion maximally degenerate.
\end{proof}
\par
We remark that the converse of the preceding proposition
is an interesting problem. Suppose that some $\QQ$-irreducible VHS 
$\PP$ over some ball quotient $B = \BB^n/\Gamma$ is maximally degenerate. 
If we knew that $\PP_\CC$ contained a sub-VHS $\WW$ of signature $(1,n)$, 
then maximally degenerate implies that this $\WW$ reaches the Arakelov bound
by \eqref{eq:sumlyapformula}. Then $\WW \oplus \overline{\WW}$ uses up
all the non-zero Lyapunov exponents. In the complement, the norm growth 
along the geodesic flow is subexponential. By
the argument in \cite[Proposition~5.4]{moellerST} this complement is in 
fact a unitary VHS.
Consequently, we can conclude as in \cite{VZShimhigher} or \cite{MVZ}
to show that $\Gamma$ is arithmetic. The missing step is thus the 
question whether there exists at all such a $\CC$-sub-VHS of $\PP_\CC$. 

\section{Cyclic covers} \label{sec:cyliccover}

The section contains well-known general material that
is used as preparation for the next section. We start
with generalities on cyclic coverings and then recall
which cyclic coverings give rise to a structure of
a ball quotient on the parameter space. 
%

\subsection{The parameter space for cyclic coverings}


Let $N\geq 3$ and let $B_0  = \moduli[0,N]$ denote the moduli
space (or configuration space) of $N$ ordered points on $\PP^1 = \PP^1(\CC)$.
Explicitly, if we let
\[M_0 = \set{x=(x_1,\dots,x_N)\in (\PP^1)^N}{\forall i,j\in
\{1,\dots,N\}: x_i\neq x_j}\]
then $\PGL_2(\CC)$ acts diagonally on $M_0$ and
\[B_0 = \PGL_2(\CC) \backmod M_0.\]

Let $a,b,c\in \PP^1$ be distinct and consider the space of normalized tuples
\[\set{x \in M_0}{x_1 = a, x_{N-1} = b, x_{N} =
c}.\]
We may identify $B_0$ with this set. A point in $B_0$ will be denoted by
$\underline{x}$; we will often choose a convenient representative by fixing
three coordinates.

\subsubsection{The type of a cyclic covering} 
A tuple $(d; a_1,\dots, a_N)$ of natural numbers is called a \textit{type of a
cyclic covering} if
\begin{align*}
d \geq 2, \quad 0 < a_i < d,\quad \gcd(a_1,\dots,a_N,d)=1,\quad
\sum_{i=1}^N a_i \in d\ZZ.
\end{align*}
A cyclic covering of type $(d; a_1,\dots,a_N)$ branched at $\underline{x}\in
B_0$ is the complete, nonsingular 
curve $X$ with affine equation
\[y^d = \prod_{i=1}^{N} (x - x_i)^{a_i}.\]
Its genus is given by the Riemann-Hurwitz formula
\begin{align*}
 g(X) = \frac{N-2}{2}d +1 -\frac{1}{2}\sum_{i=1}^N \gcd(a_i,d). 
\end{align*}

By varying $\underline{x}\in B_0$ we obtain a family of cyclic covers of
the same type 
\begin{align}\label{eqn:family-cycliccovers}
f:\XFam \to B_0.
\end{align}
Since this family is locally topologically trivial, the sheaf $R^1f_*(\CC)$ is a
local system on $B_0$.

\subsubsection{Stable and semi-stable points, cusps} 
We enlarge $B_0$ by adding divisors in the boundary that parametrize stable points (compare \cite[Section~4]{DeligneMostow86})
and this enlarged space will have the structure of a ball quotient. Sometimes the ball quotient
is not yet compact and we add some boundary points, called cusps, that parametrize classes of semi-stable points
in $M$.
\par
For a fixed $\mu =
(\mu_i)_{1\leq i \leq N} \in (\QQ\cap(0,1))^N$ such that $\sum_i \mu_i = 2$, we call $x\in
(\PP^1)^N$ a $\mu$-stable (resp.
$\mu$-semi-stable) point, if for all $z\in \PP^1$ 
\[\sum_{i: x_i = z} \mu_i    < 1\qquad (\text{resp.}\ 
  \sum_{i: x_i = z} \mu_i \leq 1).\]
Define $M = M^{\mu} \supset M_0$ to be the set of $\mu$-{stable points} and
let
\[B = B^{\mu} = \PGL_2(\CC)\backmod M.\]
Then $B$ is a complex manifold containing $B_0$ as an open subspace. We use throughout the 
convention adopted here to refer to the tuple of weights defining (semi-)stable points
by {\em upper indices} (like $B^{\mu}$ or $B^k$ later), while {\em lower indices} refer
to parameters of the local system (or period maps) on these spaces.
\par
Next let $\ol{M} = \ol{M}^{\mu} \supset M^{\mu}$ be the set of
$\mu$-{semi-stable points}. If $x \in (\PP^1)^N$ is strictly semi-stable, then
there is a unique partition 
of $\{1,\ldots,N\} = S_1 \dot{\cup} S_2$, such that $\sum_{k\in S_j} \mu_k
=1$
for $k=1,2$, $x_i\neq x_j$ for $i\in S_1, j\in S_2$, and $i\mapsto x_i$ is constant on
$S_1$ or on $S_2$. To define $\ol{B}$, we first take the quotient of $\ol{M}$ by
$\PGL_2(\CC)$ and then identify the classes of strictly semi-stable points
having the same partition. The elements of $\ol{B}\setminus B$ are called \emph{cusps}.
$\ol{B}$ with its quotient topology is a compact 
Hausdorff space and an algebraic variety \cite[4.3, 4.4.2, 4.5]{DeligneMostow86}.
\par
Let $\{i,j\} \subset \{1,\dots,N\}$ be a two-element subset such that
$\mu_i+\mu_j<1$. Define the \textit{elliptic divisor} $L_{ij} = L_{ij}^{\mu}$ as
the image of
\[\set{x =(x_1,\dots,x_N) \in M}{x_i = x_j}\]
under the projection modulo $\PGL_2(\CC)$. The interior $L_{ij}^\circ$ of
$L_{ij}$
is defined as the subset of $L_{ij}$ where precisely two coordinates are equal.

\subsection{Cyclic coverings uniformized by the ball}
For a cyclic covering of type $(d;a_1,\dots,a_N)$, define $\mu_i = a_i/d$. The
following condition on the $\mu_i$ was the first
criterion that provided a general construction of ball quotients.
\par
\begin{align} \tag{INT}
\label{eq:INT}
 \sum_{i=1}^N \mu_i = 2,\qquad \text{and} \quad \text{for all}\ i\neq j\
\text{with}\
\mu_i +
\mu_j < 1,\ (1-\mu_i-\mu_j)^{-1}\ \in \ZZ
\end{align}
\par
\begin{theorem}[\cite{DeligneMostow86}] If $\mu$ satisfies \eqref{eq:INT} then
$B$ is a ball quotient.
\end{theorem}
\par
One can relax the condition \eqref{eq:INT} in order to be more flexible and
to obtain more ball quotients. We say that the tuple of $\mu$ satisfies
\eqref{eq:SigmaINT}, 
if there is $S \subset \{1,\ldots,N\}$ such that the following conditions hold.
\par
\begin{equation} \tag{$\Sigma$INT}
\begin{aligned}\label{eq:SigmaINT}
 \sum_{i=1}^N \mu_i = 2,\qquad \text{and} \qquad \mu_i = \mu_j \quad \text{for
all} \quad i,j \in S \quad 
\text{and moreover} \\ \text{for}
\ i\neq j\ \text{with}\
\mu_i + \mu_j < 1: 
(1-\mu_i-\mu_j)^{-1}\ \in \left\{\begin{array}{ll} \frac12 \ZZ & i,j \in S \\
\ZZ & \text{otherwise} \end{array} \right.
\end{aligned}
\end{equation}
Let $\Sigma = \Sym(S)$ act on $B$ by permuting coordinates.
\par
\begin{theorem}[\cite{mostowhalfint}] If $\mu$ satisfies \eqref{eq:SigmaINT}
then
$B/\Sigma$ is a ball quotient.
\end{theorem}

\subsection{The non-arithmetic examples} \label{sec:nonarithexam}

In the following table we reproduce (and number) the non-arithmetic examples
of two-dimensional ball quotients of \cite{DeligneMostow86} and
\cite{mostowhalfint}, 
according to the corrected table in \cite{mostowdisc}. Note
that example no.~5 in loc. cit. has a misprint concerning the $a_i$.
\par
$$\begin{array}{|c|c|c|c|c|c|c|c|c|c|c|c|}
\hline
  & d  & a_1 & a_2 & a_3 & a_4 & a_5 & & &L_{ij}^\parab & g & \text{comm. to}\\
\hline
1 & 12 & 3 & 3 & 3 & 7 & 8 & INT & B_9 & -- & 12 &  4 \\
\hline
2 & 12 & 3 & 3 & 5 & 6 & 7 & INT & B_9 & L_{35} & 12 & \\
\hline
3 & 12 & 4 & 4 & 4 & 5 & 7 & INT & B_{10} & L_{45} & 12 & \\
\hline
4 & 12 & 4 & 4 & 5 & 5 & 6 & INT & B_{10} &  -- & 11 & 1 \\
\hline
5 & 15 & 4 & 6 & 6 & 6 & 8 & INT & B_{10} &  --  & 18 & 13 \\
\hline
6 & 18 & 2 & 7 & 7 & 7 & 13 & \Sigma INT & B_7/\Sigma_3 & -- & 25 & 7\\
\hline
7 & 18 & 7 & 7 & 7 & 7 & 8 & \Sigma INT & B_{10}/\Sigma_4 & -- & 25 & 6 \\
\hline
8 & 20 & 5 & 5 & 5 & 11 & 14 & INT & B_9 & --  & 22 & 9\\
\hline
9 & 20 & 6 & 6 & 9 & 9 & 10 & \Sigma INT &  B_{10} /\Sigma_2 & -- & 23 & 8\\
\hline
10 & 20 & 6 & 6 & 6 & 9 & 13 & \Sigma INT &  B_{9}/\Sigma_3 & -- & 27 & \\
\hline
11 & 24 & 4 & 4 & 4 & 17 & 19 & \Sigma INT & B_9/\Sigma_3 & -- & 30 & 12\\
\hline
12 & 24 & 7 & 9 & 9 & 9 & 14 & INT & B_{10} & --  & 31 & 11 \\
\hline
13 & 30 & 5 & 5 & 5 & 22 & 23 & \Sigma INT & B_9/\Sigma_3 & -- & 37 &  5\\
\hline
14 & 42 & 7 & 7 & 7 & 29 & 34 & \Sigma INT & B_9/\Sigma_3 & --  & 53 & 15 \\
\hline
15 & 42 & 13 & 15 & 15 & 15 & 26 & \Sigma INT & B_{10}/\Sigma_3 & -- & 58 &
14 \\
\hline
\end{array}
$$
\par
\medskip
The trace fields of these lattices are $\QQ[\cos(2\pi/d)]$.
They are distinct if and only if the $d$ are distinct with the exception
that $d=15$ and $d=30$ produce the same trace field (and indeed 
commensurable lattices). The commensurability results of Sauter
and Deligne-Mostow (see \cite[Section~3]{parkersurvey}) among
these lattices are indicated in the last column.
\par
In dimension three, there is a unique commensurability class of 
a non-arithmetic ball quotient
known, given by the cyclic covering of type $(12;7,5,3,3,3,3)$.

\subsection{Realizations of the compactification and their intersection rings} 
\label{sec:realization}
For technical reasons we will need the boundary of the ball quotients to be a normal crossing
divisor. By \cite[Lemma 4.5.1]{DeligneMostow86}, $\ol{B}$ is smooth if for all cusps, the 
associated partition $\{S_1,S_2\}$ satisfies $|S_1| = 2$ or $|S_2|=2$. Fortunately, this applies
to all cases that we need to consider. 
We will thus blow up the cusps and in the blowup $B^{\mathrm{nc}} \to
\ol{B}$ the preimage of the cusp is now a divisor, called boundary divisor. 
In the two-dimensional case, each of these divisors can be identified with
 the image modulo $\PGL_2(\CC)$
of the set of all semi-stable points in $M$ with $x_i= x_j$ for some $i,j$. We will
write
$L_{ij}^\parab$ to distinguish these divisor from elliptic divisors which 
are also some $L_{ij}$, but the monodromy is elliptic, as opposed to parabolic
in the case of boundary divisors.
\par

\subsubsection{Dimension two, case (INT)} 
\label{sec:2dimINT}
We start with two-dimensional ball quotients satisfying \eqref{eq:INT}. In all
the cases that we will be interested in 
$B^{\rm nc}$ will contain all the $10$ divisors $L_{ij}$
(some of them possibly being $L_{ij}^\parab$) or $9$ out of the $10$.
\par
We call the first case $B_{10}$. In this case $|L_{ij} \cap L_{kl}| = 1$ if
$\{i,j\} \cap \{k,l\} = \emptyset$
and $L_{ij} \cap L_{kl} = \emptyset$ other wise. The variety $B_{10}$ is isomorphic to
$\PP^2$ blown up at $4$ points in general
position \cite[\S 10.5]{YoshidaFuchs87}. Therefore, $\CH_1(B_{10})$ is generated by
\[h, e_1,\dots, e_4\]
where $h$ is the pullback of the hyperplane class in $\CH_1(\PP^2)$, and
$e_1,\dots,e_4$ are the classes of the four exceptional divisors. The ten lines
$L_{ij}$ are given by the strict transforms of the six lines in $\PP^2$
connecting the four points, and by the four exceptional divisors. The classes
of the lines can be expressed in the generators (up to
renumbering) as
\begin{align*}
 [L_{ij}] &= h - e_k - e_m\quad \text{where}\quad \{i,j,k,m\} = \{1,\dots,4\}\\
 [L_{i5}] &= e_i.
\end{align*}
The intersection matrix with respect to the ordering $h,e_1,e_2,e_3,e_4$ is
\[\diag(1,-1,-1,-1,-1),\]
and the canonical class is $K_{B_{10}} = -3h + e_1 + e_2 + e_3 + e_4$.
\par
We call the second case $B_9$ and we may choose indices so that $B_{10} \to B_9$
contracts $L_{45}$. 
Said differently, $B_9$  can be obtained by blowing up $\PP^2$ at three points
in general position. Now, $\CH_1(B_9) = \langle h,e_1,e_2,e_3\rangle$  and
\begin{align*}
[L_{12}] &= h-e_3, \quad [L_{13}] = h-e_2, \quad [L_{23}] = h - e_1  \\
 [L_{i4}] &= h - e_k - e_m\quad  \text{where}\quad \{i,k,m\} = \{1,2,3\}\\
 [L_{i5}] &= e_i. \qquad  \qquad\quad \text{for} \quad i=1,2,3
\end{align*}
The intersection pairing has matrix $\diag(1,-1,-1,-1)$ and the canonical class
is $K_{B_9} = -3h + e_1 + e_2 +e_3$.
\par

\subsubsection{Dimension two, case ($\Sigma$INT)} 
\label{sec:2dimSIGMAINT}
We will also need to work with the quotient of $B$ by a symmetric group
$\Sigma$. Let $S\subseteq\{1,\dots,5\}$, and let $\Sigma = \Sym(S)$ act on 
$B$ by permuting the coordinates. The Chow ring $\CH_*(B/\Sigma)_{\QQ}$ is
isomorphic to $(\CH_*(B)_{\QQ})^\Sigma$, and the ring structure is given by
\[(D_1.D_2) = 1/|\Sigma|\cdot \eta_*(\eta^*D_1.\eta^*D_2)\]
where $\eta:B\to B/\Sigma$ denotes the quotient map. Note that the pullback of a
divisor is the sum of the irreducible components of the preimage, each weighted
with the order of its stabilizer. Moreover, by \cite[17.4.10]{FultonIT}, 
the ring structure is independent of the presentation of $B/\Sigma$ as a quotient
by a finite group.

We only present the cases of $B$ and $\Sigma$ that are actually needed in the
subsequent computations. The case of a permutation group of $4$ elements appearing
in line $7$ of the above table will not be treated, since it is covered
by the case in line $6$ and the commensurability 
Theorem~\ref{thm:lyap_spec_comminvariant}.

\paragraph{\textit{The case} $B_9/\Sigma_3$} Let $S= \{1,2,3\}$, and consider the action of 
$\Sigma_3 = \Sym(S)$ on $B=B_9$. The map $\eta$ is ramified of order 2 along the 
divisors $L_{12}$, $L_{23}$ and $L_{13}$. The group $\CH_1(B_9/\Sigma_3)_{\QQ}$ is generated 
by the images $[\ol{L}_{14}]$,  $[\ol{L}_{15}]$, and $[\ol{L}_{12}]$ of the  divisors and
\begin{align*}
\eta^*[\ol{L}_{14}] &= [L_{14}] + [L_{24}] + [L_{34}],\\
  \eta^*[\ol{L}_{15}] &= [L_{15}] + [L_{25}] + [L_{35}],\\
  \eta^*[\ol{L}_{12}] &= 2([L_{12}] + [L_{23}] + [L_{13}]).
\end{align*}
Note that $[\ol{L}_{12}] = 2([\ol{L}_{14}] + [\ol{L}_{15}])$.
The intersection matrix with respect to this system of generators is
\[\begin{pmatrix}
   -1/2 & 1 & 1\\
   1    &-1/2 & 1\\
   1    & 1 & 4
  \end{pmatrix}.\]
By Riemann-Hurwitz,
\[\eta^*K_{B_9/\Sigma_3} = K_{B_9} - ([L_{12}] + [L_{23}] + [L_{13}]).\]
By intersecting with all generators in $\CH_1(B_9/\Sigma_3)_{\QQ}$, we obtain
\[K_{B_9/\Sigma_3} = - [\ol{L}_{12}].\]

\paragraph{\textit{The case }$B_7/\Sigma_3$} Again, let $S= \{1,2,3\}$, and let $B_7$ be the
space obtained from $B_{10}$ by collapsing all $L_{i5}$, $i=1,\dots,3$. This is
the projective
plane blown up at one point. Let again $h$ denote the hyperplane class and let $e$ denote the
class of the exceptional divisor. Then the classes of the boundary divisors are given by
\begin{align*}
[L_{i4}] &= h, \quad i = 1,2,3\\
[L_{12}] &= [L_{23}] = [L_{13}] = h - e\\
[L_{45}] &= e 
\end{align*}
The Chow group $\CH_1(B_7/\Sigma_3)_{\QQ}$ is generated by
$[\ol{L}_{12}],  [\ol{L}_{14}]$ and $[\ol{L}_{45}]$, and
\[\eta^*[\ol{L}_{12}] = 2([L_{12}] + [L_{23}] + [L_{13}]),\quad
  \eta^*[\ol{L}_{14}] = [L_{14}] + [L_{24}] + [L_{34}],\quad
  \eta^*[\ol{L}_{45}] = [L_{45}].  
\]
The respective matrix of the intersection product is
\[\begin{pmatrix}
   0 & 3 & 1\\
   3 & 3/2 & 0\\
   1 & 0   & -1/6
  \end{pmatrix},
\]
and the canonical class is given by
\[K_{B_7/\Sigma_3} = -\tfrac{2}{3}([\ol{L}_{12}] + [\ol{L}_{14}]).\]
\par

\subsubsection{Dimension three} 
\label{sec:3dimINT}
In order to compute the intersection ring of the moduli space
$B_{14}$ (we keep indexing by the number of boundary divisors of stable
configurations) of the three-dimensional ball quotient 
we use the formalism of weighted stable curves of Hassett (\cite{HasWeighted}).
In his language, we are interested in $\ol{\cM}_{0; (\frac{7}{12},\frac{5}{12},\frac14,\frac14,\frac14,\frac14)}$.
The space $\ol{\cM}_{0; (1,\frac{1}{4},\frac14,\frac14,\frac14,\frac14)}$ is isomorphic to $\PP^3$
and $\ol{\cM}_{0; (1,\frac{2}{4},\frac14,\frac14,\frac14,\frac14)}$ is the blowup of $\PP^3$ 
corresponding to curves where the points with labels $I_3 = \{2,4,5,6\}$, 
$I_4 = \{2,3,5,6\}$, $I_5 = \{2,3,4,6\}$ resp.\ $I_6 = \{2,3,4,5\}$ come
together. Reducing the weights from $(1,\frac{2}{4},\frac14,\frac14,\frac14,\frac14)$ to
$(1,\frac{5}{12},\frac14,\frac14,\frac14,\frac14)$ and further to 
$(\frac{7}{12},\frac{5}{12},\frac14,\frac14,\frac14,\frac14)$ is an
isomorphism, as can be checked using the criterion in \cite[Section~4.2]{HasWeighted}.
Consequently, we are interested in a projective space $\PP^3$, blown up at 
four points $P_i$ corresponding to the set of indices $I_i$, $i=3,4,5,6$.
\par
We deduce from \cite[Proposition~6.7 (e)]{FultonIT}  that $\Pic(B_{14})$ is
freely generated by 
the pullback $h$ of the hyperplane class of $\PP^3$ and the classes of the exceptional 
divisors $e_3,e_4,e_5,e_6$. Moreover, $\CH_1(B_{14})$ is generated by the pullback
of a line $\ell$ from $\PP^3$ and lines $\ell_i$ on the exceptional divisors.
The ring structure of the chow ring is given by
$ h \cdot h = \ell$,  $e_i \cdot e_i = (-1)\cdot \ell_1$
and 
$ h \cdot \ell = 1$, $e_i \cdot l_i = 1$
as well as  zero for all intersections not listed above. 
\par
If we normalize $x_1 = \infty$ and let $z_i = x_2 - x_i$, then
$(z_3:z_4:z_5:z_6)$ is indeed a coordinate system, in which $P_3 = (1:0:0:0)$
etc.\ and the semi-stable point is $P_{12} = (1:1:1:1)$. In these coordinates, 
one checks that
\begin{equation*}
\begin{aligned}
L_{1j} & = e_j, \quad L_{2j} = h-e_{3}-e_4-e_5-e_6 + e_j, \quad j=3,4,5,6. \\
L_{jk} & = h-e_{3}-e_4-e_5-e_6 + e_j + e_j \quad 3 \leq j < k \leq 6 \\
\end{aligned}
\end{equation*}
\par
Finally, the formula for the behaviour of the canonical class yields
\begin{equation*}
K_{B_{14}} = -4h + 2e_{3}+2e_4+2e_5+2e_6.
\end{equation*}

\subsection{Decomposition of the VHS for cyclic coverings}
A general reference for the following discussion 
is \cite{bouwhabil} and \cite{bouwprank}. 
Let $G= \Gal(X/\PP^1) \isom \ZZ/(d)$ be the Galois group of a cyclic cover of
type $(d;a_1,\dots,a_N)$. Fix a primitive $d$-th root of unity $\zeta$, the generator $g$
of $G$ given by 
\[g: x\mapsto x,\quad y \mapsto \zeta y\]
 and let $\chi:G \to \CC^\times$ be the
character of $G$ defined by $\chi(g) = \zeta^{-1}$. As $G$ acts on
$H^1_{\dR}(X)$, we can decompose this space into eigenspaces
\[H^1_{\dR}(X) = \bigoplus_{k=1}^d H^1_{\dR}(X)_{\chi^k}.\]
Since the action of $G$ respects the Hodge decomposition $H^1_{\dR}(X) =
H^{1,0}\oplus H^{0,1}$, we also have a decomposition of $H^{1,0}$ and $H^{0,1}$.
\par
The eigenspace decomposition carries over to the relative situation of \eqref{eqn:family-cycliccovers}. We have
\[R^1f_*\CC = \bigoplus_{k=1}^d \LL_k.\]
Let 
\[\cL_k^{1,0} \subseteq \LL_k\tensor\cO_{B_0}\] be the subbundle of holomorphic
forms, and let
\[\cL_k^{0,1} = (\LL_k\tensor\cO_{B_0})/\cL^{1,0}_k.\]
\par
The Chevalley-Weil formula gives information about the dimensions of the
eigenspaces.
Define for each $k= 1,\dots ,d$
\begin{align*}
 \mu_i(k)  &= \fracpart{\frac{k\cdot a_i}{d}},\qquad \mu(k) = (\mu_i(k))_i\\
 a_i(k)    &= \mu_i(k)\cdot d\\
 \sigma(k) &= \sum_{i=1}^N \mu_i(k)
\end{align*}
Here $\fracpart{x} = x - [x]$ denotes the fractional part of $x\in \RR$.
Further, let $s(k)\in \{0,\dots,N\}$ be the number of $a_i$ that are not
equal to $0$ modulo $d/\gcd(k,d)$.
\par
\begin{lemma} \label{le:basis-coho of cyclic covering}
With the notations above,
 \begin{enumerate}[a)]
  \item The rank of $\cL_k^{1,0}$ is $\sigma(k) - 1$, and the rank of
$\cL_k^{0,1}$ is $s(k) - 1 - \sigma(k)$. 
  \item $\LL_k$ is a $\CC$-local system of dimension $s(k) - 2$,
polarized by a hermitian form of signature $(\sigma(k)-1, s(k) - 1 -
\sigma(k))$.
  \item When $k$ is prime to $d$, an explicit basis of $\cL_k^{1,0}$ is given by
\[\omega_k^l = y^{-k}x^l f_k \dd x\]
with $f_k = \prod_{j} (x-x_j)^{[\tfrac{ka_j}{d}]}$ where $l = 0,\dots,
\sigma(k)-2$.
 \end{enumerate}
\end{lemma}
\begin{proof}
a) is shown in {\cite[Lemma 4.5]{bouwprank}}, and b) follows from a) and the
fact that $\LL_k$ is a polarized $\CC$-VHS. Part c) is proved in 
{\cite[Lemma 1.1.2]{bouwhabil}}. In both source, our index $k$ is replaced by $d-k$.
\end{proof}
\par
The Galois conjugates of the local system $\LL_1$ are parametrized by $k\in
\{1,\dots,d-1\}$ such that $\gcd(k,d)=1$. We will speak of
$(\mu_i(k))_i$ as a Galois conjugate tuple. Furthermore, the local system
$\LL_1$ corresponds to a non-arithmetic lattice in $\PU(1,n)$ if and only if
$(\mu_i)_i$ satisfies \eqref{eq:SigmaINT} and at least one Galois conjugate
$\LL_k$ has signature different from $(0,n)$ and $(n,0)$. 

\subsection{The period map}\label{sec:periodmap}
For each $k=1,\dots, d-1$, where $\cL_k^{1,0}$ is not trivial, we obtain a
period map $\hat B_0\to \Per(\LL_k)$ from the universal covering $\hat
B_0$ of $B_0$ to the period domain of the VHS $\Per(\LL_k)$. It is equivariant
with respect to the action of $\pi_1(B_0)$ on $\hat B_0$ and on $\Per(\LL_k)$
via the monodromy representation $\rho_k:\pi_1(B_0,\underline{x}')\to
\Aut((\LL_k)_{\underline{x}'})$. Let $K_k = \ker(\rho_k)$, and let $\tilde B_0^k =
\hat B_0 /K_k$. Denote $p_k: \tilde B_0^k\to \Per(\LL_k)$ the resulting period
map.
\par
The map $p_k$ can be though of as a multi-valued map on $B_0$. For
$\underline{x}\in B_0$ choose a path connecting $\underline{x}$ to
$\underline{x}'$. Via this path, the fibers of $R^1f_*\CC$ above $\underline{x}$
and $\underline{x}'$ can be canonically identified. Moreover,
\[(R^1f_*\CC)_{\underline{x}'} \isom H^1(\XFam_{\underline{x}'},\CC) \isom
H^1_{\sing}(\XFam_{\underline{x}'},\CC)\]
Hence, the inclusion
\[H^{1,0}(\XFam_{\underline{x}}) \to H^1_{\sing}(\XFam_{\underline{x}'},\CC)\]
is given by integration $\omega\mapsto (\gamma \mapsto \int_\gamma \omega)$.
Since this inclusion is equivariant with respect to the $G$-action, it induces
an inclusion 
\[(\cL^{1,0}_k)_{\underline{x}} \isom H^{1,0}(\XFam_{\underline{x}})_{\chi^k}
\to H^1_{\sing}(\XFam_{\underline{x}'},\CC)_{\chi^k} \isom
(\LL_k)_{\underline{x}'}\]
and ${p}_k$ maps $\underline{x}$ to the point in $\Per(\LL_k)$ defined by
the image of this morphism. Since $p_k(\underline{x})$ depends on the
chosen path, $p_k$ is multi-valued on $B_0$.
\par
Assume for the rest of Section \ref{sec:periodmap} that $\sum_i \mu_i(k) = 2$ 
and $\mu_i(k) \not\in \ZZ$ for
all $i$. Then $\cL^{1,0}_k$ is a line bundle with a global non-zero section
$\omega_k(\underline{x}) = \omega_k^0(\underline{x}) \in H^0(B_0,\cL^{1,0}_k)$ and $\Per(\LL_k)$ is a
ball in a projective space of dimension $n = N-3$. 
\par
\begin{lemma}[{\cite[Lemma 3.9]{DeligneMostow86}}] 
\label{lem:DM-Lemma on Etaleness of period map in the interior} Let $N\geq 4$.
The (multi-valued) period map 
\[p_k: B_0 \to \PP(H^1(\XFam_{\underline{x}'},\CC)_{\chi^k}),\qquad 
\underline{x}\mapsto
\biggl(\gamma\mapsto
\int_{\gamma} \omega_k(\underline{x})\biggr)\]
 has injective differential for every $\underline{x}\in B_{0}$.
 It is given by the map \[(\cT B_{0})_{\underline{x}}
\to \Hom(H^{1,0}(\XFam_{\underline{x}})_{\chi^k},\ \ 
H^{1}(\XFam_{\underline{x}},
\CC)_{\chi^k}/H^{1,0}(\XFam_{\underline{x}})_{\chi^k})\]
mapping $\pder{v}$ to 
\[(\gamma\mapsto \pder{v} \int_{\gamma} \omega_k = \int_{\gamma}
\pder{v}\omega_k)\]
\end{lemma}
This implies that $p_k$ is a local isomorphism on $\tilde B_0^k$, since 
the dimensions of the range and image agree. Recall also from Sec.~\ref{sec:Higgscurvature} that
the derivative of the period map is equal to the map $\tilde\tau$,
 derived from the dual $\tau^\dual$ of the Higgs field.

\subsubsection{Extending the period map} \label{sec:extending-periodmap}

As in \cite{DeligneMostow86}, we consider the Fox completions $\tilde
B^k$ of $\tilde B_0^k$ over $B^k = B^{\mu(k)}$ and $\tilde{B}^k_{\semistable}$
over
$\overline{B}^k = \ol{B}^{\mu(k)}$. The period map $p_k$ has continuous
extensions, denoted $p_k^k$,
to $\tilde B^k$ and to $\tilde{B}_{\semistable}^k$
(\cite[Sect. 8]{DeligneMostow86}), both of which are equivariant for the
monodromy action of $\pi_1(B_0)$. The following diagram gives an overview of
the spaces involved
\begin{equation} \label{eq:coveringdiagram}
 \xymatrix{      & \hat{B_0} \ar[d]^{/ K_k} \\
{\rm Per}(\LL_k) & \tilde B_0^k \ar[l]_{p_k} \ar[d]
\ar@{}[r]|-*[@]{\subset}& \tilde{B}^{k} 
\ar@/^1pc/[ll]^<<<{p_k^k} \ar[d]^u \ar@{}[r]|-*[@]{\subseteq} &
\tilde{B}_{\semistable}^k \ar[d]\\
& B_0 \ar@{}[r]|-*[@]{\subset}
& B^{k} \ar@{}[r]|-*[@]{\subseteq}& {\ol B}^k & B^{k,\rm nc} \ar[l]
\\
}
\end{equation}
If there is additional symmetry in the tuple $(\mu_i(k))_i$, then the above
discussion carries over to a quotient (see \cite{mostowhalfint}). Let
$S\subset
\{1,\dots,N\}$, and assume $\mu_i(k) = \mu_j(k)$ for all $i,j\in S$. The group
$\Sigma = \Sym(S)$ acts on $B_0$, and there is an open, dense submanifold $F$
where the action is free. The VHS on $B_0$ descends to a VHS $\LL_{k,\Sigma}$
on the quotient $F/\Sigma$, and we obtain a period map $\widehat{F/\Sigma} \to
\Per(\LL_{k,\Sigma})$ from the universal covering of $F/\Sigma$ to the period domain of
$\LL_{k,\Sigma}$, which is equivariant for the monodromy action $\rho_{k,\Sigma}$ of
$\pi_1(F/\Sigma)$. Set 
\[K_{k,\Sigma} = \ker(\rho_{k,\Sigma})
\quad\text{and}\quad \tilde F_\Sigma^k
= \widehat{F/\Sigma}/ K_{k,\Sigma},\]
 and let 
\[p_{k,\Sigma}^k : \widetilde{B}_\Sigma^k \to \Per(\LL_{k,\Sigma})\]
denote the induced map
from the Fox completion of $\tilde F_\Sigma^k$ over $B^k/\Sigma$ to the
period domain of $\LL_{k,\Sigma}$.
\par
\begin{remark} Assume that $\mu = \mu(1)$ satisfies \eqref{eq:INT},
then $p^1_1$ is an isomorphism from $\tilde B^1$ onto a complex ball, and turns
$B^1$ into a ball quotient. If $\mu$ satisfies
\eqref{eq:SigmaINT}, the same is true for the induced map $p^1_{1,\Sigma}$ and
$B^1/\Sigma$. 
\end{remark}
\par
We shall also need the extension of the period map of $\LL_k$ to the Fox
completion $\tilde B^1$. The generic case is the following: $\mu(1)$ satisfies
\eqref{eq:SigmaINT} and $k$ parametrizes a Galois conjugate $\LL_{k,\Sigma}$ of $\LL_{1,\Sigma}$.
Since the kernels $K_{1,\Sigma} = K_{k,\Sigma}$, we obtain a period map
\[p_{k,\Sigma}^1: \tilde F_\Sigma^1 \to \Per(\LL_{k,\Sigma})\]
which we wish to extend to $\tilde B^1_\Sigma$. Via $p^1_{1,\Sigma}$, the set
$\tilde F_\Sigma^1$ is the complement in $\BB^n$ of normal crossing divisors,
the preimages of the stable $\ol{L}_{ij}$'s. By Lemma \ref{lemma:branch orders
are integers} below, the monodromy about each of these divisors is of finite
order, whence by \cite[Theorem~9.5]{griffglobal} the map $p^1_{k,\Sigma}$
extends holomorphically to
\[p^1_{k,\Sigma} : \tilde B^1_\Sigma \to \Per(\LL_{k,\Sigma})\]
with image in the interior of the period domain. The following diagram captures this situation.
\begin{equation} 
\label{eq:coveringdiagram-Sigma}
 \xymatrix{      & \widehat{F/\Sigma} \ar[d]^{/ K_{1,\Sigma}} \\
{\rm Per}(\LL_{k,\Sigma}) 
& \tilde F_\Sigma^1 \ar[l]_{p_{k,\Sigma}^1} \ar@/^1pc/[ddr] \ar@{}[r]|-*[@]{\subset}
& \tilde{B}^{1}_\Sigma 
 \ar@/^2pc/[ll]_<<<<<<<<<<{p_{k,\Sigma}^1} \ar[ddr] \ar[r]^{\cong}_{p^1_{1,\Sigma}} 
 & \BB^n\\
& F \ar@{}[r]|-*[@]{\subseteq} \ar[dr]
& B^{1} \ar[dr] \\
&
& F/\Sigma \ar@{}[r]|-*[@]{\subseteq} 
&B^1/\Sigma\\
}
\end{equation}
\par
There is a third case, when we need an extension of $p_k$. As above let $k$
parameterize a Galois conjugate of $\LL_1$, and let $\LL_1$ have signature
$(1,n)$, but do not assume that $\mu(1)$ satisfies \eqref{eq:INT} (e.g. it
only satisfies \eqref{eq:SigmaINT}). Let $\tilde B_1^1$ denote the Fox
completion of $\tilde B^1_0$ over the union of $B_0$ with the codimension
1-strata of the $\mu(1)$-stable points. Then the map $p_k^1: \tilde B^1_0 \to
\Per(\LL_k)$ is again well-defined and extends to $p_k^1: \tilde B^1_1 \to
\Per(\LL_k)$. The point here is that $\tilde B^1_1$ is a complex manifold by
the
discussion in \cite[Sect. 10]{DeligneMostow86}, the monodromy about the stable
divisors being of finite order. We can thus apply the same reasoning as in the
second case.

\subsubsection{Relations between period maps}\label{sec:relations-between-periodmaps}
The period map, and in particular its domain of definition depends on $k$. If
$p_k$ and $p_\ell$ belong to Galois conjugate local systems, their domains of
definition coincide, but in general they need not be related in any way, as the
following example shows.
\par
\begin{exam}\label{ex:non-extendability}
{\rm  Consider the family of cyclic coverings $f:\XFam \to B_0$ of type $(12; 3,3,3,7,8)$. This
family of curves is fiberwise a degree 2-covering of a family of cyclic
coverings $f':\XFam'\to B_0$ of type $(6;3,3,3,1,2)$ by taking the quotient by
$\langle{g^2}\rangle\subset G$, and the resulting map $\XFam \to \XFam'$ induces an inclusion of VHS
with image
\[ \bigoplus_{k\congruent 0 \bmod 2} \LL_k \subset R^1f_*\CC.\]
We neither have $K_2 \subseteq K_1$, nor $K_1\subseteq K_2$. To see this, consider the
monodromy transformation, i.e. the image of a small loop $\gamma_{ij}$ about
one of the divisors $L_{ij}$ under the monodromy representation $\rho_k$.
In case $\sum_{i} \mu_i(k) = 2$ and $\mu_i(k)\not\in \ZZ$, its order is the denominator 
of the reduced fraction $|1-\mu_i(k)-\mu_j(k)|^{-1}$ or
$\infty$ if $\mu_i(k) + \mu_j(k) = 1$. Thus, $L_{12}$ is a 
parabolic divisor for $\LL_2$, which is elliptic for $\LL_1$, whereas e.g. for
$L_{15}$, the monodromy transformation has order $12$ for $k=1$ and order $3$
for $k=2$. 
\par
In particular, the pullback of the local system $\LL_2$ to the universal cover
 $\tilde B^1 \isom \BB^n$ of $B^1$ cannot be
extended as local system over the preimage of $L_{12}$ in $\BB^n$, thus for a
generic geodesic, parallel transport is not defined for all times.
\par
For each of the non-arithmetic ball quotients listed in the table
in Section~\ref{sec:nonarithexam} there exists a sub-VHS for which
an analogous statement holds.
}\end{exam}
\par
For later use, we record a relation between
$(\mu_i(1))_i$ and $(\mu_i(k))_i$ for $k$ prime to $d$. Its proof is a
straightforward computation.
\par
\begin{lemma} \label{lemma:branch orders are integers}
Let $d>1$, $k\in \{0,\dots,d-1\}$ be prime to $d$, and let $i,j\in \{1,\dots,
N\}$. Then
\begin{enumerate}[a)]
 \item $\mu_{i}(k) + \mu_{j}(k) = 1$ iff
$\mu_{i}(1) + \mu_{j}(1) = 1$
 \item If $\kappa_{ij} = (1-\mu_i(1)-\mu_j(1))^{-1}$ is in $\ZZ$ or if
$a_i = a_j$, $\kappa_{ij}\in \tfrac{1}{2}\ZZ$ and $k$ is odd, then
\[\kappa_{ij} \cdot (1-\mu_i(k) - \mu_j(k)) \in \ZZ.\]
\end{enumerate}
\end{lemma}

\subsubsection{Explicit description of the period map}
\label{sec:explicit basis of H1}
We describe a trivialization of $\LL_k$ at $\underline{x}\in B_0$. Instead of
working with $(\LL_k)_{\underline{x}} = H^1(\XFam_{\underline{x}},\CC)_{\chi^k}$, we use the isomorphism
\begin{equation} \label{eqn:iso-coho-and-coho-in-local-system}
\begin{aligned}
H^1(\XFam_{\underline{x}},\CC)_{\chi^k} \isom H^1(\PP^1 \setminus
\{\underline{x}\}, L) 
\end{aligned}
\end{equation}
where $\{\underline{x}\}$ is the set of points defined by $\underline{x}$ and
$L$ is the rank one local system on $\PP^1\setminus\{\underline{x}\}$ having
local monodromy $\exp(2\pi i \mu_i(k))$ about $x_i$ (see
\cite[Sect. 2]{DeligneMostow86} for a discuss of the cohomology of $L$).
Explicitly, $H^*(\PP^1 \setminus \{\underline{x}\}, L)$ can be computed using
the de Rham complex of $\mathrm{C}^\infty$-forms on $\PP^1\setminus
\{\underline{x}\}$. One can also consider the subcomplex of compactly supported
forms, whose cohomology we denote by $H_c^*(\PP^1 \setminus \{\underline{x}\},
L)$. 
\par
If none of the $\mu_i(k)$ is an integer, then $H_c^*(\PP^1 \setminus
\{\underline{x}\},L) \isom H^*(\PP^1 \setminus \{\underline{x}\},
L)$ (see \cite[2.3]{DeligneMostow86}). The Poincar\'e duality pairing
\[H^i(\PP^1 \setminus \{\underline{x}\}, L) \tensor H^{2-i}_c(\PP^1 \setminus
\{\underline{x}\}, L^\dual) \to \CC,\quad \alpha \tensor \beta \mapsto
\int_{\PP^1\setminus\{\underline{x}\}} \alpha \wedge \beta\]
allows us to identify the dual of $H^1(\PP^1 \setminus \{\underline{x}\},
L)$ with the first cohomology of the complex of currents on $\PP^1\setminus
\{\underline{x}\}$ with values in $L^\dual$. We describe a particular basis of
$H^1(\PP^1 \setminus \{\underline{x}\}, L^\dual)$ (compare
\cite[2.5]{DeligneMostow86}). Assume $\{1,\dots,N\} = S_1\cup S_2$ is a
partition such that $\sum_{i\in S_1} \mu_i(k) \not\in \ZZ$. Embed the union of
two trees $T_1\cup T_2$ into $\PP^1$ such that the vertices of $T_i$ are
$\set{x_s}{s\in S_i}$, and let $\gamma_j$ be the oriented edges of $T_1\cup
T_2$. Then a basis of $H^1(\PP^1\setminus \{\ul{x}\}, L^\dual)$ is given by the
currents of integration along the paths $\gamma_j$, $j=1,\dots, N-2$, each
tensored with a global section $e_j$ of $L^\dual$ restricted to (the interior of)
$\gamma_j$. 
\par
If $\cL^{1,0}_k$ is a line bundle, the basis element of 
Lemma~\ref{le:basis-coho of cyclic covering} is identified by the 
isomorphism \eqref{eqn:iso-coho-and-coho-in-local-system} with
the multi-valued form $\omega_k = \prod_j (x-x_j)^{-\mu_j(k)}\dd x$
tensored with an appropriate multi-valued section $e$ of $L$, so as to 
produce an honest single-valued section
of $H^0(\PP^1\setminus\{\underline{x}\}, \Omega^1(L)) \subset 
H^1(\PP^1\setminus\{\underline{x}\}, L)$.
The period map is then the
map 
\[{p}_k = (F_1^k:\dots: F_{N-2}^k)\ \text{with}\ F_j^k = \int_{\gamma_j\tensor
e_j} \omega_k\tensor e.\]

\section{Lyapunov exponents of ball quotients constructed via cyclic coverings} \label{sec:Lyapcyclic}

In this section we calculate explicitly the Lyapunov exponents
for all non-arithmetic ball quotients arising from cyclic covers. 
In the table below we also add the set of relative orbifold Euler numbers
for an easy comparison. Their defintion is given in the next section.
%
\par
\begin{theorem} \label{thm:allprimvalues}
The non-negative distinct Lyapunov exponents of the primitive part 
of the two-dimensional non-arithmetic ball quotients 
arising from cyclic coverings are given in the following table.
Here, $\UU$ denotes
the maximal unitary subsystem of $\PP$.  
\[\begin{array}{|c|c|c|c|c|c|c|c|c|c|c|}
\hline
  & d  & g &  \dim_\RR \PP & \dim_\RR \UU & \text{Prim.\ Lyapunov
spectrum} & \text{Relative}\,\,\, e^\orb  \\
\hline
1 & 12 & 12 & 12 & 0 &  \{1, \quad5/13,\quad 0\} & \{1, \quad 1/13\}\\
\hline
2 & 12 & 12 & 12 & 0 & \{1, \quad 5/17,\quad 0\}  & \{1, \quad 1/17\}\\
\hline
3 & 12 & 12 & 12 & 0 & \{1, \quad 7/22,\quad 0\}  & \{1, \quad 1/22\} \\
\hline
4 & 12 & 11 & 12 & 0 &  \{1, \quad 5/13,\quad 0\} & \{1, \quad 1/13\} \\
\hline
5 & 15 & 18 & 24 & 6 &  \{1, \quad 16/37, \quad  7/37,\quad 0\} & \{1, \quad 4/37, \quad  1/37\}\\
\hline
6 & 18 & 25 & 18 & 6 &  \{1, \quad 5/16,\quad 0\} & \{1, \quad 1/16\} \\
\hline
7 & 18 & 25 & 18 & 6 & \{1, \quad 5/16,\quad 0\} & \{1, \quad 1/16 \} \\
\hline
 8 & 20 & 22 & 24 & 6   & \{1,\quad 4/11,\quad 3/11,\quad 0\} & \{1, \quad 4/33, \quad 1/33\}\\
\hline
 9 & 20 & 23 & 24 & 6   & \{1,\quad 4/11,\quad 3/11,\quad 0\} & \{1, \quad 4/33, \quad 1/33\}\\
\hline
10 & 20 & 27 & 24 & 12  & \{1, \quad 11/46,\quad 0\} & \{1, \quad 1/46\} \\
\hline
11 & 24 & 30 & 24 & 12  & \{1, \quad 7/22,\quad 0\} & \{1, \quad 1/22\} \\
\hline
12 & 24 & 31 & 24 & 12  & \{1, \quad 7/22,\quad 0\} & \{1, \quad 1/22\} \\
\hline
13 & 30 & 37 & 24 &  6  & \{1, \quad 16/37, \quad  7/37,\quad 0\} & \{1, \quad 4/37, \quad  1/37\}\\
\hline
14 & 42 & 52 & 36 & 18  & \{1, \quad 16/61, \quad 13/61,\quad 0 \} & \{1, \quad 4/61, \quad 1/61\}\\
\hline
15 & 42 & 58 & 36 & 18  & \{1, \quad 16/61, \quad 13/61,\quad 0 \} & \{1, \quad 4/61, \quad 1/61\}\\
\hline
\end{array} \]
\par
\smallskip
For the three-dimensional ball quotient, the positive primitive Lyapunov spectrum
is given by
\[1, 1, \tfrac{25}{93}, \tfrac{25}{93}, 0,0,0,0.\]
\end{theorem}
\par
\begin{cor}\label{cor:commensurability-result}
 The non-arithmetic lattices in $\PU(1,n)$ arising via cyclic coverings fall into precisely nine commensurability classes.
\end{cor}
\begin{proof}
 Using the trace field, we only need to decide whether no. $2$ or $3$ belong to
the class of $\{1,4\}$ or are commensurable and whether no. $10$ belongs to the
class of $\{8,9\}$. But by Corollary~\ref{cor:comminvariants} the primitive
Lyapunov spectrum is a commensurability invariant among lattices admitting a
modular embedding.
\end{proof}
\par
The proof of the theorem relies on computing the right-hand side of
\eqref{eq:sumlyapformula}. We first
show how to relate this quantity to contributions coming from the boundary
divisors, and then compute these contributions for each case. 
\par

\subsection{Intersection products of Chern classes}
\label{sec:intersection products of Chern classes}
Let $B^u = \BB^n/\Gamma'$ be a ball quotient, where $\Gamma'\subset \PU(1,n)$
acts cofinitely and freely on $\BB^n$. We assume that there is a smooth
projective variety $Y$ such that $B^u$ embeds into $Y$ with $Y\setminus B^u =
\Delta$, a divisor with normal crossings. Suppose further that we are given a
polarized $\CC$-VHS $\LL^u$ on $B^u$ of weight $1$ and signature $(1,n)$, whose local
monodromies are unipotent. Recall from 
Section~\ref{sec:Higgscurvature}
that Higgs field $\tau$ of $\LL^u$ comes with two derived maps 
$\tilde\tau$ and $\tau^\vee$.
\par
\par
\begin{lemma}\label{lemma:compute c1 of 10 eq. compute c1 of kokern}
Assume that $\tilde\tau$ (or equivalently, $\tau^\dual$) is injective.
Then
\begin{align}\label{eqn:10,01 und kokern}
 \frac{(n+1)\Chern_1(\cE^{1,0}).\Chern_1(\omega_Y)^{n-1}}{
\Chern_1(\omega_Y)^n }
= 1 - \frac{\Chern_1(\koker \tau^\dual).\Chern_1(\omega_Y)^{n-1}}
	   {\Chern_1(\omega_Y)^n},
\end{align}
where $\omega_Y = \bigwedge^n\Omega_Y(\log \Delta)$.
\end{lemma}
\begin{proof}
From the assumption, we obtain a short exact sequence
\begin{align*}
0 \to \cE^{1,0}\tensor \cT_Y(-\log \Delta) \to \cE^{0,1} \to \koker \tau^\dual
\to 0.
\end{align*}
By \cite[Theorem 1.1]{EsnaultVieweg02}, we have 
\begin{align*}
 \Chern_1(\cE^{1,0} \oplus \cE^{0,1}) = 0,
\end{align*}
and from the above exact sequence we obtain
\[\Chern_1(\cE^{0,1}) = \Chern_1(\cT_Y(-\log \Delta)) + n\cdot
\Chern_1(\cE^{1,0}) + \Chern_1(\koker \tau^\dual).\]
Combining these two equations yields
\[(n+1)\Chern_1(\cE^{1,0}).\Chern_1(\omega_Y)^{n-1} =
-(\Chern_1(\cT_Y(-\log\Delta)) + \Chern_1(\koker
\tau^\dual)).\Chern_1(\omega_Y)^{n-1}.\]
\end{proof}

\begin{rem}
For the computation of $\Chern_1(\koker\tau^\dual)$, it will be convenient to
consider the following short exact sequence associated with $\tilde\tau$
\begin{align} \label{eqn:ses for tautilde}
0 \to \cT_Y(-\log\Delta) \to \SheafHom(\cE^{1,0},\cE^{0,1}) \to
\koker \tilde\tau \to 0.
\end{align}
Then $\Chern_1(\koker \tilde\tau) = \Chern_1(\koker\tau^\dual)$, using
additivity of $\Chern_1$ on short exact sequences.
\end{rem}
\par

\subsection{Computation of $\Chern_1(\koker\tau^\dual)$ for the case (INT)}
Let $N\geq 4$, and let $f:\XFam \to B_0$ be a cyclic covering of type
$(d,a_1,\dots,a_N)$, such that the collection $\mu_i = a_i/d$ satisfies
\eqref{eq:INT}. Let  $B \isom \BB^n/\Gamma$ be
the ball quotient parametrizing $\mu$-stable points. Let $1\leq k \leq d-1$ be
coprime to $d$, and such that the direct summand $\LL = \LL_k$ in the VHS of $f$
is polarized by a hermitian form of signature $(1,n)$. 
\par
We consider a finite index normal subgroup $\Gamma'$ of $\Gamma$ 
such that the local monodromies about the boundary divisors are unipotent under $\rho_k$, 
and such that moreover $\Gamma'$ acts freely on $\BB^n$. 
We denote $\pi:B^u\to B$
and $u:\BB^n\to \BB^n/\Gamma'$ the two projections, and let
$\LL^u$ be the pullback of $\LL$ to $\pi^{-1}(B_0)\subset B^u$. By the discussion 
in Section~\ref{sec:extending-periodmap}, we can extend $\LL^u$ to a VHS on $B^u$,
and we denote it by the same letter.
\par
The results of the preceding sections apply to $\LL^u$. Let $\tau^\dual =
\tau^\dual_k$ be the dual of the Higgs field, and let $p_k = p_k^1$ be the
period map
for $\LL$, which we can also think of as a multi-valued map on $B^u$. We have
\par
%
\begin{lemma}\label{lemma:computing Chern_1 of kokertaudual-INT}
 \[\Chern_1(\koker \tau^\dual) = 
\sum_{\substack{\{i,j\}:\\ \mu_i+\mu_j < 1}}
n^k_{ij}\cdot \sum_{\substack{L\subset \pi^{-1}L_{ij}\\ \text{irred.}}} [L]
\qquad + \quad \sum_{\substack{\Delta'\subset \Delta\\ \text{irred.}}}
n_{\Delta'}\Delta'\]
Moreover, $n^k_{ij}$ is the vanishing order of $\det(J(p_k))$ at a point of
$\pi^*L_{ij}^\circ$.
\end{lemma}
\begin{proof}
Define the line bundles
\[\cL_1 = \bigwedge^n \cT_B(-\log\Delta)\quad \text{and}\quad \cL_2 =
\bigwedge^n \SheafHom(\cE^{1,0},\cE^{0,1}).\]
Then
 \[\Chern_1(\koker \tau^\dual) = \Chern_1(\cL_2) -
\Chern_1(\cL_1).\]
 Let $\phi: \cL_1\to \cL_2$ be the map induced by $\tilde\tau_k$. It is locally
at
$p\in Y$ given by multiplication with $a_p\in \cO_{Y,p}$. We tensor both line
bundles with $\cL_2^{-1}$. Then the image of $\cL_1\tensor\cL_2^{-1}$ under
$\phi\tensor \id$ is the ideal sheaf
$\cI \subseteq \cO_Y$ locally generated by $a_p$. Moreover,
\[\Chern_1(\cI) = -\Chern_1(\koker \tilde\tau),\]
and $\Chern_1(\cI) = -D$ for some divisor $D$ on $Y$.
\par
It follows from Lemma~\ref{lem:DM-Lemma on Etaleness of period map in the
interior}, that $a_p = 1$ if $p\in
\pi^{-1}(B_0)$. Hence, 
\[\supp(-D) \subseteq \bigcup_{(i,j):\mu_i+\mu_j<1}\pi^{-1}L_{ij}
\cup \Delta.\]
\par
Let $p \in \pi^{-1}L_{ij}^\circ$, and let $L$ be the irreducible component
of
$\pi^*L_{ij}$ containing $p$. Then $\ord_L(-D) = \ord_L(a_p)$. Since $\phi =
\wedge^n\tilde \tau$ and since in the interior, $\tilde \tau$ is the derivative of the
period map $p_k$, it follows that $\ord_L(a_p)$ is independent of the
irreducible component of $\pi^{-1}L_{ij}$, since $p_k$ is 
equivariant under the full group $\Gamma$, and given by the
vanishing order of $\det(J(p_k))$
along a component of the preimage of $L_{ij}$.
\end{proof}
\par

\subsubsection{Near elliptic divisors}
Next we study $\tilde\tau_k$ at one of the elliptic divisors $L_{ij}$. By Lemma
\ref{lemma:branch orders are integers}, $\mu_i(k) + \mu_j(k) \neq 1$, and
\[\ell_{ij}^k = \kappa_{ij}\cdot |1-\mu_i(k)-\mu_j(k)| \in \NN.\]
\par
\begin{lemma}\label{le:compute coefficients of kokertaudual-case INT}
 We have 
\[n^k_{ij} = \begin{cases}
               \ell_{ij}^k - 1 & \text{if}\ \mu_i(k)+\mu_j(k) < 1\\
               n\cdot\ell_{ij}^k - 1& \text{if}\ \mu_i(k)+\mu_j(k) > 1\\
             \end{cases}\]
\end{lemma}
\par
To prove this lemma, we will need an explicit description of the period
map locally about $L_{ij}$. We recall the discussion in
\cite{DeligneMostow86}. Consider a neighborhood of a point
\[\underline{x}'  = (x_1',\dots,x_{N}') \in \PGL_2(\CC)\backmod (\PP^1)^N,\]
where precisely two entries coincide. Up to renumbering, we can assume them to
be the first two. Lift the point to $\CC^N$ by normalizing
$x_1' = x_2' = 0$, $x'_{N-1} = a$, $x'_{N} = b$ with $a,b\in \CC$ some
suitable points. Let $\varepsilon$ be so small that the disks of radius
$\varepsilon$
about the $x_i'$ do not intersect, and let $U$ be the set of 
\[\underline{x} = (x_1=0,x_2,\dots,x_{N-2},x_{N-1}=a,x_{N}=b)\]
such that $|x_i - x_i'| < \varepsilon$. Then $U$ describes a neighborhood of
$\underline{x}'\in B$.
\par
About $\underline{x}\in U$, we can choose a
basis
of $H^1(\PP^1\setminus\{\underline{x}\}, L)$ as follows (compare Section~\ref{sec:explicit basis of H1}). 
Let $\gamma_1$ be the
straight line segment connecting $x_1 = 0$ and $x_2$. Embed a tree $T$ into
$\CC\setminus B_{\varepsilon}(0)$ such
that its
vertices are $x_i$, $i= 3,\dots, N$, and let $\gamma_j$ ($j=2,\dots,N-2$) be
the edges of $T$ (with a chosen orientation). For each $j=1,\dots,N-2$, choose a
section $e_j$ of $L^\dual$ on interior of $\gamma_j$. 
\par
%
%
With this choice of basis, we obtain in particular for the first of the components 
$F_j^k$ of $p_k$
\begin{align*}
F_1^k(\underline{x}) &= \int_0^{x_2} x^{-\mu_1(k)} (x-x_2)^{-\mu_2(k)}
\prod_{i\neq 1,2} (x-x_i)^{-\mu_i(k)} \dd x\\
		   &= x_2^{1-\mu_1(k)-\mu_2(k)} \int_0^1 x^{-\mu_1(k)}
      (x-1)^{-\mu_2(k)} \prod_{i\neq 1,2} (x_2x-x_i)^{-\mu_i(k)} \dd x 
\end{align*}
\par
The factor
\[I_k(\underline{x}) = \int_0^1 x^{-\mu_1(k)}
      (x-1)^{-\mu_2(k)} \prod_{i\neq 1,2} (x_2x-x_i)^{-\mu_i(k)} \dd x,\]
and $I_k$ is well-defined at $\{x_2 = 0\}$. Moreover  (compare
\cite[9.5]{DeligneMostow86}), $I_k$ is holomorphic and
$I_k\neq 0$ at $\{x_2 = 0\}$, so we may suppose that $I_k$ does not vanish on $U$.
\par
\paragraph{\textit{The case} $\mu_i(k) + \mu_j(k) < 1$.} The above
discussion, together with Lemma \ref{lem:DM-Lemma on Etaleness of period map in
the interior}, yields the following lemma, see {\cite[9.5]{DeligneMostow86}}.
\par
\begin{lemma} \label{lemma:vanishing order of periodmap-case1}
\begin{enumerate}[a)]
 \item If $\mu_i(k) + \mu_j(k) <
1$, then in a
neighborhood $U$ of the divisor $L_{ij}$, the period map is a multi-valued map
to $\PP^n$, locally given as
\[(z_1,\dots, z_n)\mapsto (z_1^{1-\mu_i(k)-\mu_j(k)}, z_2,\dots, z_n),\]
where $z_1=0$ is a local equation for $L_{ij}$.
 \item Assume that $(\mu_i(1))_i$ satisfies \eqref{eq:INT}. The ramification
order of $\pi:B^u \to B$ at $L_{ij}$ is 
\[\kappa_{ij} = (1-\mu_i-\mu_j)^{-1}.\]
\end{enumerate}
\end{lemma}
\par
\begin{proof}[Proof of Lemma \ref{le:compute coefficients of kokertaudual-case
INT}, first case]
The period map $p_k$ thought of as a multi-valued map on $B^u$ is locally about
$L_{ij}$ given as the composition
\[ \left((z_1,\dots,z_n)\mapsto (z_1^{1-\mu_i(k)-\mu_j(k)}, z_2,\dots, z_n) \right) \circ
\left((z_1,\dots,z_n) \mapsto (z_1^{\kappa_{ij}},z_2,\dots,z_n)\right).\]
Hence, $\det(J(p_k)) = \ell_{ij}^k z_1^{\ell_{ij}^k-1}$.
\end{proof}
\par

\paragraph{\textit{The case} $\mu_i(k) + \mu_j(k) > 1$.} Then the multivalued function
\[F_1^k = x_2^{1-\mu_1(k) - \mu_2(k)} I_k\]
has a pole at $x_2 = 0$ of order $|1-\mu_1(k) - \mu_2(k)|\in \QQ$ and the projective
tuple
\[(x_2^{\mu_1(k) +
\mu_2(k) - 1} F_1^k: \dots : x_2^{\mu_1(k) +
\mu_2(k) - 1} F_{N-2}^k)\]
is well-defined on $U$. We set $m=\mu_1(k)+\mu_2(k)-1$ and
dehomogenize $p_k$ with respect to the first coordinate
\begin{align*}
{p}_k(\underline{x}) &= (F_2^k/F_1^k, \dots, F_{N-2}^k/F_1^k)\\
		      &= (x_2^m F_2^k/I_k, \dots , x_2^m F_{N-2}^k/I_k).
\end{align*}
\par


For the computations below the following lemma is useful. Its proof
is a straightforward computation.
\par
\begin{lemma}\label{lem:relation between Jacobians of map to projective space}
 Let $U\subset \CC^n$, 
\[f:U\to \PP^n,\quad w=(w_1,\dots,w_n) \mapsto (f_0(w):f_1(w):\dots:f_n(w))\]
be a holomorphic map, and let $w\in U$ be such that $f_0(w)\neq 0$. Then in an
open neighborhood of $w$ the determinants of
\[J_1 = \begin{pmatrix}
         \pder[(f_i/f_0)]{w_j}
        \end{pmatrix}_{i,j}\qquad \text{and}\qquad
J_2 = \begin{pmatrix}
       f_0 & \pder[f_0]{w_1} & \dots & \pder[f_0]{w_n}\\
       f_1 & \pder[f_1]{w_1} & \dots & \pder[f_1]{w_n}\\
       \vdots & \vdots & \ddots & \vdots\\
       f_n & \pder[f_n]{w_1} & \dots & \pder[f_n]{w_n}\\
      \end{pmatrix}\]
are related by
\[\det(J_2) = f_0^{n+1} \det(J_1).\]
\end{lemma}

We would like to determine the vanishing order of the determinant of period map
$\ol{p}_k$ 
at a point $\underline{x} = (x_1 = 0,x_2= 0, x_3,\dots,x_{N})$. Note that
\[\begin{pmatrix}
          \pder[(f_i/f_0)]{x_j}
         \end{pmatrix}_{i,j} = f_0^{-2}
\begin{pmatrix}
 \begin{vmatrix}
  f_0 & f_i\\[2mm]
  \pder[f_0]{x_j} & \pder[f_i]{x_j}
 \end{vmatrix}
\end{pmatrix}_{i,j}\]
and recall that only $x_2,\dots,x_{n+1}$ are varying coordinates. We have to
evaluate the vanishing order of the determinant of the following matrix. We
suppress the super- and subscript $k$ in the following.
\[J = I^{-2} \begin{pmatrix}
             \begin{vmatrix}
              I & x_2^mF_2\\
              \pder[I]{x_2} & mx_2^{m-1}F_2 + x_2^m\pder[F_2]{x_2} 
             \end{vmatrix} & 
             \begin{vmatrix}
              I & x_2^mF_2\\
              \pder[I]{x_3} & x_2^m\pder[F_2]{x_3} 
             \end{vmatrix} & \dots &
             \begin{vmatrix}
              I & x_2^mF_2\\
              \pder[I]{x_{n+1}} & x_2^m\pder[F_2]{x_{n+1}} 
             \end{vmatrix}\\
             \vdots & \vdots & \ddots & \vdots \\
             \begin{vmatrix}
              I & x_2^mF_{n+1}\\
              \pder[I]{x_2} & mx_2^{m-1}F_{n+1} + x_2^m\pder[F_{n+1}]{x_2} 
             \end{vmatrix} &
	     \begin{vmatrix}
              I & x_2^mF_{n+1}\\
              \pder[I]{x_3} & x_2^m\pder[F_{n+1}]{x_3} 
             \end{vmatrix} &\dots &
	     \begin{vmatrix}
              I & x_2^mF_{n+1}\\
              \pder[I]{x_{n+1}} & x_2^m\pder[F_{n+1}]{x_{n+1}} 
             \end{vmatrix}       
      \end{pmatrix}
\]

\begin{lemma} \label{le:vanishing of detperiodmap-case two}
 $J$ vanishes of order $x_2^{nm-1}$ at $\underline{x}$.
\end{lemma}
\begin{proof}
Write $\det(J) = I^{-2}x_2^{(n-1)m + m-1}\det(\tilde J)$ with
\[\tilde J = \begin{pmatrix}
             \begin{vmatrix}
              I & x_2F_2\\
              \pder[I]{x_2} & mF_2 + x_2\pder[F_2]{x_2} 
             \end{vmatrix} & 
             \begin{vmatrix}
              I & F_2\\
              \pder[I]{x_3} & \pder[F_2]{x_3} 
             \end{vmatrix} & \dots &
             \begin{vmatrix}
              I & F_2\\
              \pder[I]{x_{n+1}} & \pder[F_2]{x_{n+1}} 
             \end{vmatrix}\\
             \vdots & \vdots & \ddots & \vdots \\
             \begin{vmatrix}
              I & x_2F_{n+1}\\
              \pder[I]{x_2} & mF_{n+1} + x_2\pder[F_{n+1}]{x_2} 
             \end{vmatrix} &
	     \begin{vmatrix}
              I & F_{n+1}\\
              \pder[I]{x_3} & \pder[F_{n+1}]{x_3} 
             \end{vmatrix} &\dots &
	     \begin{vmatrix}
              I & F_{n+1}\\
              \pder[I]{x_{n+1}} & \pder[F_{n+1}]{x_{n+1}} 
             \end{vmatrix}                     
  \end{pmatrix}
\]
We show that $\tilde J$ does not vanish at $\underline{x} =
(0,x_3,\dots,x_{n+1})$.
 Evaluation at $\underline{x}$ yields
\begin{align*}\det(\tilde J) &= mI 
 \begin{vmatrix}
  F_2 & I\pder[F_2]{x_3} - F_2\pder[I]{x_3} &\dots & I\pder[F_2]{x_{n+1}} -
F_2\pder[I]{x_{n+1}}\\
  \vdots & \vdots & \ddots & \vdots\\
  F_{n+1} & I\pder[F_{n+1}]{x_3} - F_{n+1}\pder[I]{x_3} &\dots &
I\pder[F_{n+1}]{x_{n+1}} - F_{n+1}\pder[I]{x_{n+1}}\\           
 \end{vmatrix}\\
 &= mI
 \begin{vmatrix}
  F_2 & I\pder[F_2]{x_3}  &\dots & I\pder[F_2]{x_{n+1}} \\
  \vdots & \vdots & \ddots & \vdots\\
  F_{n+1} & I\pder[F_{n+1}]{x_3}  &\dots & I\pder[F_{n+1}]{x_{n+1}}\\           
 \end{vmatrix} \\
 &= mI^{1+(n-1)}
 \begin{vmatrix}
  F_2 & \pder[F_2]{x_3}  &\dots & \pder[F_2]{x_{n+1}} \\
  \vdots & \vdots & \ddots & \vdots\\
  F_{n+1} & \pder[F_{n+1}]{x_3}  &\dots & \pder[F_{n+1}]{x_{n+1}}\\           
 \end{vmatrix} 
\end{align*}
By Lemma \ref{lem:DM-Lemma on Etaleness of period map in the interior} 
and Lemma \ref{lem:relation between Jacobians of map to projective space}, we
know that 
the determinant of
\[ \begin{pmatrix}
  F_2 & \pder[F_2]{x_3}  &\dots & \pder[F_2]{x_{n+1}} \\
  \vdots & \vdots & \ddots & \vdots\\
  F_{n+1} & \pder[F_{n+1}]{x_3}  &\dots & \pder[F_{n+1}]{x_{n+1}}\\           
 \end{pmatrix} 
\]
does not vanish at a point $(x_3,\dots,x_{n+1})$ where all coordinates are
 distinct and different from $0$, $x_{N-1}$, $x_{N}$,
since this matrix describes the period map of the tuple
\[\mu' = (\mu_1(k)+\mu_2(k), \mu_2(k), \dots, \mu_N(k)).\]
When $N=4$, the above determinant degenerates to $F_2$, whose non-vanishing is
the content of \cite[Proposition 2.13]{DeligneMostow86}.
%
\end{proof}
\par
\begin{proof}[Proof of Lemma \ref{le:compute coefficients of kokertaudual-case
INT}, second case]
If we precompose the period map starting on $U$ with $\pi:B^u\to B$, then 
using the chain rule, we see that $\det(J(p_k))$ vanishes of order
$\kappa_{ij}nm-1$ at a generic point of the preimage of $L_{ij}$.
\end{proof}
\par

\subsection{Near parabolic divisors}

\begin{lemma}\label{le:no contribution from boundary}
We have $n_{\Delta'} = 0$ for every irreducible component $\Delta'$ of $\Delta$.
\end{lemma}
\par
\begin{proof}
By Lemma~\ref{lem:DM-Lemma on Etaleness of period map in the interior} the
period map is a local isomorphism away
from the elliptic divisors, hence in particular near a generic
point of a parabolic boundary divisor. It thus suffices to
prove the lemma for a uniformizing VHS.
\par
To emphasize similarity with the case of the universal family
of elliptic curves, we use the Siegel domain realization
$$\BB_{\rm Siegel} = \{(z_1:z_2:...:z_n:1): \Im(z_1) - \sum_{i=2}^n |z_i|^2 > 
0  \} 
\subset \PP^n$$ 
and work near the boundary point $(1:0:\cdots:0) \in \PP^n$. A loop
around the boundary point is represented by the parabolic matrix
$T = I_{n+1} + E_{1,n+1}$. We write $q = \exp(2\pi i z_1)$.
\par
We also work with the dual uniformizing VHS, where now $\WW^{1,0}$
has rank $n$.  Over $\BB_{\rm Siegel}$ we choose a basis
$s_1,\ldots,s_{n+1}$ of the constant local system and define the
Hodge filtration $\cW^{1,0}$ as the kernel of the tautological quotient
map 
\[\CC^{n+1}\tensor\cO_{\BB_{\rm Siegel}} \to \cO_{\BB_{\rm Siegel}},\quad
\sum_{i=1}^{n+1} f_is_i \mapsto (\sum_{i=1}^{n} z_i f_i) + f_{n+1}.\]
Explicitly,
$\cW^{1,0}$ is generated by $\omega_i = s_i - z_is_{n+1}, i=1,\ldots,n$. The
sections
$s_i$ for $i=2,\ldots,n+1$ are $T$-invariant and extend to $\tilde{s_i}$ over
the boundary of the quotient by $T$ of $\BB_{\rm Siegel}$. Together with
$\tilde{s_1} = s_1 - z_1 s_{n+1}$ they form
a basis of the Deligne extension of the local system.
\par
A basis of  $\cT_Y(-\log\Delta)$ near a point in $\Delta$ is given
by $2\pi i q\frac\partial{\partial q},\frac\partial{\partial z_2},
\ldots,\frac\partial{\partial z_n}$.  One calculates that
the matrix of derivatives of the $\omega_i$ in these directions
is minus the identity. This is equivalent to $\tau$ being an isomorphism
near the boundary and to the vanishing of the cokernel, as claimed.
\end{proof}
\par

\subsection{Computation of $\Chern_1(\koker\tau^\dual)$ for the case
($\Sigma$INT)}
\label{sec:SigmaINT-computations of chern1-koker}
The calculations proceed analogously to the \eqref{eq:INT} case with some minor
twists. 
\par
Let $N\geq 5$, and let $f:\XFam \to B_0$ be a cyclic covering of type
$(d,a_1,\dots,a_N)$, and assume now that the collection $\mu_i = a_i/d$
satisfies \eqref{eq:SigmaINT} for some $S\subset \{1,\dots,N\}$. Then $B/\Sigma
= \BB^n/\Gamma$ is a ball quotient. The primitive part of the VHS defined by
$f$ then furnishes a VHS on $F/\Sigma$. Let $\LL =
\LL_{k,\Sigma}$ be one of its direct summands, which is polarized by a hermitian
form of signature $(1,n)$.
\par
We choose again an appropriate subgroup $\Gamma'\subset \Gamma$, 
and define $B^u$, $\pi$, $Y$, $\Delta$
and $\LL^u$ analogously to the \eqref{eq:INT}-case. Again the results of
Section \ref{sec:intersection products of Chern classes} apply; Let $\tau^\dual
= \tau^\dual_k$ be the dual of the Higgs field, and let
$p_{k,\Sigma} = p_{k,\Sigma}^1$ be the period map
for $\LL$.
\par
Let $F\subset B_0$ be the open dense submanifold, where $\Sigma$ acts freely,
and let $M\subset B_0\setminus F$ be the codimension one fix locus of $\Sigma$
in $B_0$ (which is only present in a few cases). We denote by $M_l$, $l=1,\dots,r$
the irreducible
components of the closure of $M$ in $B$. Let $\varpi: B \to B/\Sigma$ be the
canonical projection, and let 
\[\{\ol{L}_{ij}\}\quad \text{and} \quad\{\ol{M}_l\}\]
be a system of representatives of $\Sigma$-orbits of the divisors $L_{ij}$
and $M_l$.
\par
\begin{lemma}[{\cite[Lemma 8.3.2]{delcommen}}]
\label{le:ramification of BtoBSigma}
The map $\varpi:B\to B/\Sigma$ is ramified of order two at
each smooth point of
$\bigcup_{i,j \in S} L_{ij} \cup \bigcup_{l} M_l$. 
\end{lemma}
\par
\begin{lemma} \label{le:Chern 1-kokertaudual SigmaINT}
 The ramification order of $\pi:B^u \to B/\Sigma$ at a
point $x\in B^u$ 
\[
  \begin{cases}
   \kappa_{ij} &,\ \text{for generic}\ x \ \text{in}\ \pi^{-1}\ol{L}_{ij}
\ \text{with}\ i,j\not\in S\\
  2\kappa_{ij} &,\ \text{for generic}\ x \ \text{in}\ \pi^{-1}\ol{L}_{ij}
\ \text{with}\ i,j\in S\\
   2	       &,\ \text{for generic}\ x \in \pi^{-1}\ol{M}_l\\
   1	       &,\ \text{if}\ x\in \pi^{-1}(\varpi(F))
  \end{cases}
\]
Moreover, we have
\[\Chern_1(\koker \tau^\dual) = 
\sum_{i,j}
n^k_{ij}\cdot \sum_{\substack{L\subset \pi^{-1}\ol{L}_{ij}\\ \text{irred.}}}
[L]\]
where the $n_{ij}^k$ are given as in Lemma \ref{le:compute coefficients of
kokertaudual-case INT}.
\end{lemma}
\begin{proof}
a) Let $x$ be a point in $\varpi(F)$ or on one of the divisors $\ol{L}_{ij}$ or
$\ol{M}_l$, but outside the intersection with any of the other divisors. The
map $\varpi$ is locally at $x$ a ramified covering map. By Lemma
\ref{le:ramification of BtoBSigma}, its
multi-valued inverse is of the form
\[(z_1,\dots,z_n) \mapsto
(z_1^{1/2},z_2,\dots,z_n)  \quad \text{or}\quad \id:\CC^n\to \CC^n \]
depending on whether $x\in \bigcup_{i,j \in S} L_{ij} \cup
\bigcup_{l} M_l$ or not.
If $p_k$ denotes the multi-valued period map associated with $\LL_k$ on $B$,
then $p_{k,\Sigma}$ is locally at $x$ given as the composition of the
multi-valued inverse of $\varpi$ and $p_k$. For $k=1$, the
single-valued period map is an isomorphism of
$\widetilde{B}_{\Sigma}^1$ with $\BB^n$, and the map $\BB^n\to B^u$ is
unramified, so $\pi:B^u\to B$ must precisely make up for the ramification of the
multi-valued map $p_{1,\Sigma}$. This shows a).
\par
b) First, notice that we can copy the proof of Lemma \ref{lemma:computing
Chern_1 of kokertaudual-INT} to see that the support of $\Chern_1(\koker
\tau^\dual)$ is contained in the complement of $\pi^{-1}(\varpi(F))$. Next by
essentially the same arguments as in the proof of Lemma \ref{le:no contribution
from boundary}, there are no contributions from the boundary $\Delta$.
Furthermore, $n_{ij}^k$ is again given by the vanishing order of
$\det(J(p_{k,\Sigma}))$. By the proof of Part a) and the chain rule,
this vanishing order is still given as in Lemma \ref{le:compute coefficients of
kokertaudual-case INT}.
\end{proof}

\subsection{Collecting the contributions}
In this section, we gather all the steps needed to compute the Lyapunov
exponents, and thus to prove Theorem \ref{thm:allprimvalues}.

Since \eqref{eq:SigmaINT} comprises \eqref{eq:INT}, we consider the setup of
Section \ref{sec:SigmaINT-computations of chern1-koker}, and work with the
notations introduced there. In particular, $N\geq 5$, $\LL = \LL_{k,\Sigma}$ is
a
direct summand of the primitive part of the VHS on $B/\Sigma$, induced by a
cyclic
covering, and $\tau^\dual$ is the dual Higgs field of the pullback VHS $\LL^u$
on $B^u$.

\begin{prop}\label{prop:Lyapexp of summand of cyclic covering}
The positive Lyapunov spectrum of $\LL^u$ is
\[\lambda_1 = \lambda_1 \geq \underbrace{0 = \dots = 0}_{n-1}\]
where
\[\lambda_1 = 
 1 - \frac{\Chern_1(\koker \tau^\dual).\Chern_1(\omega_Y)^{n-1}}
	   {\Chern_1(\omega_Y)^n}.\]
\end{prop}
\begin{proof}
 By Theorem \ref{thm:main} iii), $2n-2$ of the $2n+2$ Lyapunov exponents 
 of $\LL^u$ are zero. Moreover, since we are working with a $\CC$-variation,
 the first Lyapunov exponent occurs with multiplicity $2$, i.e.
 \[\lambda_1 = \lambda_2.\]
 By Proposition \ref{prop: duplication} and its proof, the Lyapunov spectrum of
$\WW_{\RR}$, the $\RR$-variation associated with $\LL_u \oplus \ol{\LL_u}$, is
the
same as the one of $\LL^u$. Therefore, by Theorem \ref{thm:main} iv),
\[2\lambda_1 = \frac{(n+1)\Chern_1(\cW^{1,0}).\Chern_1(\omega_Y)^{n-1}}{
\Chern_1(\omega_Y)^n }\]
with $\cW^{1,0} = \cE^{1,0} \oplus \ol{\cE}^{1,0}$, where $\cE^{1,0}$ and
$\ol{\cE}^{1,0}$ are
the 
$1,0$-parts of the Deligne extensions of the holomorphic subbundles associated
with $\LL^u$ and $\ol{\LL^u}$. From the antiholomorphic isomorphism 
$\cE^{0,1}\isom \ol{\cE}^{1,0}$ and the fact that $\Chern_1(\cE^{1,0}) =
-\Chern_1(\cE^{0,1})$, 
we infer that
\[\Chern_1(\cW^{1,0}) = 2\Chern_1(\cE^{1,0}).\]
Now, Lemma \ref{lemma:compute c1 of 10 eq. compute c1 of kokern} implies
the claim.
\end{proof}
\par
Next, we argue that we can compute the above intersection products on
$B/\Sigma$ instead of $B^u$. We introduce the divisor $D_k\in
\CH_1(B^{\nc}/\Sigma)_\QQ$ given by
\[D_{k,\Sigma} = \sum_{i,j\in S} \frac{n_{ij}^k}{2\kappa_{ij}}
\bigl[\ol{L}_{ij}\bigr]
+ \sum_{i,j\not\in S} \frac{n_{ij}^k}{\kappa_{ij}} \bigl[\ol{L}_{ij}\bigr],\]
and the orbifold canonical divisor
\begin{align}\label{eq:SigmaKorb}
K^{\orb}_\Sigma = K_{B^{\nc}/\Sigma} + \sum_{i, j \in S}
\Bigl(1-\frac{1}{2\kappa_{ij}}\Bigr)\bigl[\ol{L}_{ij}\bigr]
+ \sum_{i, j \not\in S}
\Bigl(1-\frac{1}{\kappa_{ij}}\Bigr)\bigl[\ol{L}_{ij}\bigr]
+ \sum_{l}
\frac{1}{2} \bigl[\ol{M}_{l}\bigr] 
\end{align}
with $n_{ij}^k$ from Lemma \ref{le:compute
coefficients of kokertaudual-case INT}, $\kappa_{ij} = (1-\mu_i-\mu_j)^{-1}$ and
the convention $1/\kappa_{ij} = 0$, if $\mu_i+\mu_j=1$.
We also define the divisors
\begin{align}\label{eq:Korb}
D_k = \sum_{\substack{i,j\\ \mu_i+\mu_j < 1}}
\frac{n_{ij}^k}{\kappa_{ij}}
[L_{ij}]\qquad \text{and}\qquad 
K^{\orb} = K_B + \sum_{\substack{i,j\\ \mu_i+\mu_j \leq 1}}
(1-\frac{1}{\kappa_{ij}}) [L_{ij}]
\end{align}
upstairs in $\CH_1(B^\nc)_\QQ$.
\par
\begin{lemma} We have
\[\lambda_1 = 1 - \frac{\Chern_1(\koker \tau^\dual).\Chern_1(\omega_Y)^{n-1}}
	   {\Chern_1(\omega_Y)^n} 
= 1 -  \frac{D_{k,\Sigma}.(K^{\orb}_\Sigma)^{n-1}}{(K^{\orb}_\Sigma)^n}
= 1 -  \frac{D_k.(K^{\orb})^{n-1}}{(K^{\orb})^n}.\]
\end{lemma}
\begin{proof}
 By Lemma \ref{le:Chern 1-kokertaudual SigmaINT}, the divisors
$D_k$ and $K^{\orb}$ are adapted to satisfy
\[\pi^*D_{k,\Sigma} = \Chern_1(\koker \tau^\dual)\quad \text{and}\quad
\pi^*K^{\orb}_\Sigma =
\Chern_1(\omega_Y),\]
and by Lemma \ref{le:ramification of BtoBSigma},
\[\varpi^*D_{k,\Sigma} = D_k\qquad \text{and}\qquad \varpi^*K^{\orb}_\Sigma =
K^{\orb}.\]
The ring structure on $\CH_*(B^\nc/\Sigma)_\QQ$ does not depend on the
presentation of $B^\nc/\Sigma$ as a finite quotient, whence
\[(1/\deg(\pi)) \pi^*D_1.\pi^*D_2 =  D_1.D_2 = (1/|\Sigma|)
\varpi^*D_1.\varpi^*D_2\]
for any $D_1$, $D_2$ in $\CH_*(B^\nc/\Sigma)_\QQ$.
\end{proof}

We can now plug in the concrete realizations of
$\CH_*(B^{\nc})_\QQ$ from Section~\ref{sec:realization} to finish
the computation in each individual case.

\section{Orbifold Euler numbers and log-ball quotients}
\label{sec:orbeuler_logbq}

Thurston considers in his paper \cite{thurstonshapes} the moduli space 
$C(\mu_1,\dots,\mu_N)$ of Euclidean metrics on the sphere of total area one 
with conical  singularities at $N$ points of fixed angles
$(2\pi\mu_1,\dots,2\pi\mu_N)$ 
(where $\mu_i\in \QQ\cap (0,1)$). If the sum of angles $\sum \mu_i =2$, Thurston  provides this space with 
a hyperbolic metric so that, up to taking the 
quotient by the finite group permuting the points of equal angle, the metric 
completion $\ol{C}(\mu_1,\dots,\mu_N)$ of this moduli space is 
the space of $(\mu_i)_i$-stable points $B^\mu$. This metric completion is a
hyperbolic cone manifold, and an orbifold ball quotient precisely when
\eqref{eq:SigmaINT} is satisfied.
\par
For each of these hyperbolic cone manifolds, one can study its volume. In
particular, for $\mu = (\mu_i(1))_i$ parametrizing an orbifold ball quotient and
$\mu(k)$ a Galois conjugate tuple thereof with $\sum_i\mu_i(k) =2$, the ratio
of the volumes of $C(\mu_1(k),\dots,\mu_N(k))$ and $C(\mu_1(1),\dots,\mu_N(1))$
is an interesting invariant, which remains unchanged under passage to a
covering defined by a finite index subgroup of the associated lattice.
\par
\par
Here we study the corresponding algebro-geometric version of this
invariant, the orbifold Euler number. By the
Gau\ss-Bonnet theorem 
it is proportional to the volume of
the hyperbolic cone manifold.
\par
Since all but one of the known non-arithmetic lattices are in dimension at most two
and for simplicity we restrict to the surface case. The orbifold Euler number,
as introduced by Langer \cite{langerorbifold}, is defined for any pair $(X,D)$
of a normal projective surface and a $\QQ$-Divisor $D =\sum a_i D_i$,
where the $D_i$ are prime divisors and $0\leq a_i\leq 1$, and is given by
 \begin{align} \label{eqn:orbEuler number}
    \orbEuler(X,D) = \topEuler(X) - \sum a_i \topEuler(D_i\setminus \Sing(X,D))
- \!\!\!\!\!\sum_{x\in \Sing(X,D)} 1- \orbEuler(x,X,D).
 \end{align}
It extends the well-studied case when
$a_i = (1-1/m)$ with $m \in \NN$, \ie when $(X,D)$ is an actual orbifold. Here,
$\Sing(X,D)$ is the locus where either $X$ or $\supp D$ is singular, and 
$\orbEuler(x,X,D)$ is the local orbifold Euler number that we will give
below in the cases of interest.
\par
Fix a quintuple $(\mu_i)_{i=1}^5$. First, we only suppose that $\sum_i
\mu_i = 2$ and consider the space $B^{\mu,\nc}$ of $\mu$-semistable points,
blown up at the cusps, together with the divisor
\begin{align} \label{eq:R-divisor}
R^\mu = \sum_{i<j:\mu_i+\mu_j\leq 1}
(\mu_i+\mu_j)[L_{ij}].
\end{align}
Further let $K_{B^{\mu,\nc}}$ denote the canonical divisor class of
$B^{\mu,\nc}$.
\par
\begin{theorem}\label{thm:log-ballquotients}
For any $(\mu_i)_{i=1}^5\in (\QQ\cap(0,1))^5$ with $\sum_i\mu_i=2$, the pair
$(B^{\mu,\nc}, R^\mu)$ is a is a log-ball quotient in the sense that
with the definition $\Chern_1^\orb(B^{\mu,\nc}, R^\mu) = K_{B^{\mu,\nc}} +
R^\mu$ equality is attained in the generalized Bogomolov-Miyaoka-Yau
inequality, i.e.\ 
\begin{align} \label{eq:BMY}
3 \orbEuler(B^{\mu,\nc}, R^\mu) = \bigr(\Chern_1^\orb(B^{\mu,\nc},
R^\mu)\bigl)^2.
\end{align}
\end{theorem}
\par
It is classically known that 2-dimensional ball quotients are the surfaces
of general type realizing equality in the Bogomolov-Miyaoka-Yau
inequality $3\Chern_2 \geq \Chern_1^2$. This inequality can be generalized to 
orbifolds and further to log-canonical pairs $(X,D)$ (see \cite{langerorbifold} and 
references therein). The BMY-equality of orbifold Chern classes has been used
to find new examples of ball quotients (for an account see \cite[\S 10]{YoshidaFuchs87}). Our
theorem yields new examples of log-varieties satisfying a BMY-equality without being 
a ball quotient.
\par
In the case that $\mu$ satisfies \eqref{eq:INT}, the orbifold
canonical divisor $\Chern_1^\orb(\cB^\ell, \cR^\ell)$ coincides with $K^\orb$ in
\eqref{eq:Korb} and we recover the orbifold version of Chern class
proportionality on a ball quotient.
\par
From this theorem we will deduce that the relative
orbifold Euler characteristics are ratios of intersection numbers on the ball
quotient and thus commensurability invariants by
Corollary~\ref{cor:comminvariants}. For the precise statement, let us fix
a quintuple $(\mu_i)_{i=1}^5$ satisfying \eqref{eq:SigmaINT}, and let
$\mu(k) = (\mu_i(k))_i$ be a Galois conjugate quintuple parametrizing a local
system of signature $(1,2)$. For $\ell \in \{1,k\}$, we consider the space
$\cB^{\ell} = B^{\nc,\mu(\ell)}/\Sigma$ and define the divisor $\cR^\ell$ to be
the image of $R^{\mu(\ell)}$ in the Chow ring of $\cB^\ell$.
\par
\begin{cor} \label{cor:compute-otherinvariant}
Suppose $(\mu_i)_{i=1}^5$ satisfies \eqref{eq:SigmaINT}, and let
$\mu(k) = (\mu_i(k))_i$ be a Galois conjugate quintuple parametrizing a local
system of signature $(1,2)$. Then
 \[\frac{\orbEuler(\cB^{k},\cR^k)}{\orbEuler(\cB^{1},\cR^1)} =
\frac{(K^\orb - D_k)^2}{(K^{\orb})^2} =
\frac{\Chern_1(\SheafHom(\cE^{1,0}_k,\cE^{0,1}_k))^2}{\Chern_1(\omega_{Y})^2}
 =
\frac{9\Chern_1(\cE^{1,0}_k)^2}{\Chern_1(\omega_{Y})^2}
 \]
where $Y$ is as in Section \ref{sec:SigmaINT-computations of chern1-koker} 
and $\cE^{1,0}_k$
(resp. $\cE^{0,1}_k$) is the holomorphic subbundle (quotient bundle) of the
pullback VHS of $\LL_k$.
\end{cor}
\par
Using the structure of the Chow ring explicated in Section \ref{sec:2dimINT},
we can evaluate the right hand side of the \eqref{eq:BMY}. We can also evaluate
the left hand side explicitly.
\par
\begin{prop}\label{prop:formula for orbEuler for mu}
The orbifold Euler number of $(B^{\mu,\nc},R^\mu)$ is
\begin{gather}\label{eq:orbEuler-explicit-Formula}
  \begin{split}
\orbEuler(B^{\mu,\nc}, R^\mu) &= 7 
+ \sum_{i<j:\mu_i+\mu_j\leq 1} (\mu_i+\mu_j)\ 
+ \sum_{i<j:\mu_i+\mu_j>1} (\mu_i +\mu_j -1)^2 - 2\\
&\quad + \sum_{\substack{i<j,l<m:\\ i<l, j\neq m\\
\mu_i+\mu_j\leq
1\\ \mu_l+\mu_m\leq 1}} (1- \mu_i- \mu_j)(1- \mu_l-\mu_m) - 1.
  \end{split}
\end{gather}
\end{prop}
\par
We will combine these two results for the proof of
Theorem~\ref{thm:log-ballquotients}. It would however be interesting to find a
more conceptual explanation for the proportionality relation, using the fact
that on the open dense set $\tilde B_0$ in the Fox completion $\tilde B^\mu$,
the period map is \'etale and induces a sheaf isomorphism of
$\Omega^1_{B_0}$ with $\SheafHom(\cL^{1,0},\cL^{0,1})$, the sheaf induced by
pulling back the cotangent sheaf of $\BB^n$. The major difficulty seems to be
the lack of a complex structure on the Fox completion for arbitrary $\mu$ with
$\sum_i \mu_i=2$.

\subsection{Orbifold Euler numbers}
 For our computations it will be sufficient to know
$\orbEuler(x,X,D)$ when $X$ is smooth and $D$ has normal crossings or
ordinary triple points, i.e.\ locally at $x\in \Sing(D)$ the divisor $D$ is
analytically isomorphic to
$\sum_{i=1}^n a_iL_i$ where $L_i$ are $n \leq 3$ distinct lines in $\CC^2$
passing through the origin. Then $\orbEuler(x,X,D) = \orbEuler(0,\CC^2,\sum a_i
L_i)$. In this situation, we have in particular \cite[Theorem
8.3]{langerorbifold}
 \[\orbEuler(0,\CC^2, \sum a_i L) = \begin{cases}
		       (1-a_1)\cdot(1- a_2),   &\text{if}\ n=2\\
		       \tfrac{1}{4}\cdot(a_1+a_2+a_3 -2)^2, & \text{if}\ n=3,
				      a_1\leq a_2 \leq a_3, \\
				      &\quad a_3 < a_1+a_2\ \text{and}\\
				      &\quad a_1+a_2+a_3\leq 2.\\ 
				    \end{cases}\]
Another property of $\orbEuler$ that we will use is the fact that it behaves
multiplicatively under pullback, i.e if $f:Y\to X$ is a finite, proper morphism
of normal, proper surfaces and $f^*(K_X + D)
= K_Y + D'$ for a $\QQ$-divisor $D'$, then
\[\orbEuler(Y,D') = \deg(f) \cdot \orbEuler(X,D).\]
\par
We can now prove the explicit formula for the orbifold Euler number of the pair
$(B^{\mu,\nc},R^\mu)$. 
\par
\begin{proof}[Proof of Proposition~\ref{prop:formula for orbEuler for mu}]
The topological Euler characteristic of $B^{\mu,\nc}$ is $7$ minus the number of
contracted $L_{ij}$'s. $L_{ij}$ is contracted whenever $\mu_i+\mu_j>1$, whence
the third summand has a $(-2)$ instead of the $(-1)$ appearing in 
\eqref{eqn:orbEuler number}. The second summand is the
codimension $1$ contribution. The coefficient of $L_{ij}$, whenever present in
$R^\mu$, is $a_{ij} = \mu_i+\mu_j$, and $L_{ij}\setminus \Sing(R^\mu)$ is a
thrice-punctured $\PP^1$. The remaining summands are the codimension $2$
contributions. For each contracted $L_{ij}$ we have an ordinary triple point
with contribution
\[\tfrac{1}{4}(a_{lm}+a_{lq}+a_{mq} - 2)^2 = (\mu_l+\mu_m+\mu_q - 1)^2 =
(\mu_i+\mu_j-1)^2,\]
where $\{l,m,q\} = \{1,\dots,5\}\setminus\{i,j\}$. Note that
$a_{\sigma(l)\sigma(m)} \leq a_{\sigma(l)\sigma(q)}+a_{\sigma(m)\sigma(q)}$ for
every permutation $\sigma$ of $\{l,m,q\}$ and $a_{lm}+a_{lq}+a_{mq}< 2$. The
last summand is the contribution from the ordinary double points.
\end{proof}

\subsection{Proof of Theorem \ref{thm:log-ballquotients} and Corollary
\ref{cor:compute-otherinvariant}}
The space $B^{\mu,\nc}$ results from $B_{10}$ by blowing down all the divisors
$L_{ij}$ with $\mu_i+\mu_j>1$. Before engaging in the proof of
Theorem~\ref{thm:log-ballquotients}, we compute the pullback of the
``boundary`` divisor $R^\mu$ under this blowup map.

\begin{lemma} \label{le:B10Bkcherncomparison}
Let $\mu\in (\QQ\cap (0,1))^5$, $\sum_i\mu_i= 2$ and let $b:B_{10} \to
B^{\mu,\nc}$ denote the map blowing down all divisors $L_{ij}$ with $\mu_i+\mu_j
> 1$.
Then
\begin{equation} \label{eq:bpullback}
b^*(K_{B^{\mu,\nc}} + R^\mu) = K_{B_{10}} + \sum_{\substack{i<j:\\
\mu_i+\mu_j\leq 1}} (\mu_i+\mu_j)[L_{ij}] + \sum_{\substack{i<j:\\ \mu_i+\mu_j
> 1}} (3-2(\mu_i+\mu_j))[L_{ij}].
\end{equation}
\end{lemma}
\begin{proof} 
Since $b$ is a composition of several blowdowns that commute with each other,
we can treat each blowdown individually. So assume first that there is only one
pair  $\{a,b\} \subset \{1,\dots,5\}$ such that $\mu_a + \mu_b >1$. The image
of the exceptional divisor $L_{ab}$ under $b$ is the
intersection of the lines $L_{cd}^\mu$, $L_{ce}^\mu$ and $L_{de}^\mu$ in
$B^\mu$, where
$\{a,b,c,d,e\} = \{1,\dots,5\}$. The total transform of a boundary divisor
$[L_{ij}^\mu]$ is thus
\[b^*[L_{ij}^\mu] = \begin{cases} 
               [L_{ij}] + [L_{ab}],&\quad \text{if}\
\{i,j\}\subset\{c,d,e\}\\
	       [L_{ij}],&\quad \text{otherwise.}
             \end{cases}
\]
Therefore, we obtain altogether
\begin{align*}
b^*(K_{B^{\mu,\nc}} + R^\mu) &= K_{B_{10}}  
+ \sum_{\substack{i<j:\\ \mu_i+\mu_j \leq 1}} (\mu_i+\mu_j) [L_{ij}]\\
      &\qquad + \sum_{\substack{a<b:\\ \mu_a+\mu_b> 1\\ \{a,b,c,d,e\} =
\{1,\dots,5\}}} 
(-1 +
2(\mu_c+\mu_d+\mu_e))[L_{ab}].
\end{align*}
Using $\sum_i\mu_i=2$, we obtain the claim.
\end{proof}
\par
\begin{proof}[Proof of Theorem~\ref{thm:log-ballquotients}]
We first show the identity $(K_{B^\mu}+R^\mu)^2 = \orbEuler(B^\mu,R^\mu)$
for a configuration of type $B_{10}$. We have $K_{B_{10}}^2 = 5$, and
$[L_{ij}].K_{B_{10}} = -1$. Moreover \[ [L_{ij}].[L_{lm}] = \begin{cases}
                    0, &\quad \text{if}\ |\{i,j\} \cap \{l,m\}| = 1\\
		    1, &\quad \text{if}\ \{i,j\} \cap \{l,m\} = \emptyset\\
		   -1, &\quad \text{if}\ \{i,j\} = \{l,m\}.\\
                   \end{cases}
\]
Hence,
\begin{align*}
 \bigl(\sum_{i<j} (\mu_i+\mu_j)[L_{ij}]\bigr)^2 &= - \sum_{i<j} (\mu_i+\mu_j)^2
+
    2\cdot \sum_{i<j,l<m, i<l} (\mu_i+\mu_j)(\mu_l+\mu_m)\\
    &= -4 \sum_i \mu_i^2 - 2\sum_{i<j}\mu_i\mu_j + 12\sum_{i<j} \mu_i\mu_j
    = 18\sum_{i<j} \mu_i\mu_j -16,
\end{align*}
where we use $4 = (\sum_i \mu_i)^2 = \sum_i \mu_i^2 + 2\sum_{i<j} \mu_i\mu_j$.
Altogether we obtain
\begin{align*}
 (K_{B_{10}} + \sum_{i<j} (\mu_i+\mu_j)[L_{ij}])^2 &= K_{B_{10}}^2 +
2K_{B_{10}}.(\sum_{i<j} (\mu_i+\mu_j)[L_{ij}]) + (\sum_{i<j}
(\mu_i+\mu_j)[L_{ij}])^2 \\
    &= 5 + (-2)\cdot 4\sum_{i} \mu_i + 18\sum_{i<j}\mu_i\mu_j - 16\\
    &= 18\sum_{i<j}\mu_i\mu_j - 27.
\end{align*}
On the other hand,
\begin{align*}
 \orbEuler(B_{10},R^\mu) &= 7 + 4\cdot \sum_i \mu_i - 15 + \sum_{i<j,l<m,i<l}
(1-\mu_i-\mu_j)(1-\mu_l-\mu_m) \\
      &= 15 - \tfrac{4!}{2} \sum_i \mu_i + \sum_{i<j,l<m,i<l}
(\mu_i+\mu_j)(\mu_l+\mu_m)\\
      &= -9 + 6 \sum_{i<j} \mu_i\mu_j
\end{align*}
Now let $\mu$ be arbitrary. By Lemma \ref{le:B10Bkcherncomparison}, the blowdown
map $b: B_{10} \to B^\mu$ satisfies
\[b^*(K_{B^\mu} + R^\mu) = \overbrace{K_{B_{10}} + \sum_{\substack{i<j}}
(\mu_i+\mu_j)[L_{ij}]}^A + \overbrace{\sum_{\substack{i<j:\\ \mu_i+\mu_j>
1}}3(1-\mu_i-\mu_j)[L_{ij}]}^B.\]
We proceed by computing the individual summands of 
\[(K_{B^\mu} + R^\mu)^2 =
(b^*(K_{B^\mu} + R^\mu))^2 = A^2 + 2AB + B^2.\]
For $i<j$ with $\mu_i+\mu_j>1$, let
$\{l,m,q\}$ be the complement in $\{1,\ldots,5\}$. Then
\begin{align*}
(K_{B_{10}} + \sum_{\substack{i'<j'}}
(\mu_{i'}+\mu_{j'})[L_{i'j'}]).[L_{ij}]) 
&= -1 +
2(\mu_l+\mu_m+\mu_q) - (\mu_i+\mu_j)\\
&= 3(1-\mu_i-\mu_j)
\end{align*}
As different $L_{ij}$'s with $\mu_i+\mu_j>1$ do not intersect,
\[\biggl(\sum_{\substack{i<j:\\ \mu_i+\mu_j>
1}}3(1-\mu_i-\mu_j)[L_{ij}] \biggr)^2 = -9\sum_{\substack{i<j:\\ \mu_i+\mu_j>
1}}(1-\mu_i-\mu_j)^2\]
Putting everything together, using that the
identity $A^2 = 18\sum_{i<j}\mu_i\mu_j - 27$ also holds for arbitrary $\mu$, 
we obtain
\begin{align*}
(K_{B^\mu} + R^\mu)^2 &= (K_{B_{10}} + \sum_{\substack{i<j}}
(\mu_i+\mu_j)L_{ij})^2 + 9\sum_{\substack{i<j:\\\mu_i+\mu_j> 1}}
(1-\mu_i-\mu_j)^2\\
    &= 18\sum_{i<j}\mu_i\mu_j - 27 + 9\sum_{\substack{i<j:\\\mu_i+\mu_j>
1}}(1-\mu_i-\mu_j)^2.
\end{align*}
Now we investigate the amount by which the Euler characteristic changes for
each $i<j$ with $\mu_i+\mu_j>1$. First observe that we still can evaluate the
right hand side of \eqref{eq:orbEuler-explicit-Formula} for the boundary divisor
$\sum_{i<j}(\mu_i+\mu_j)[L_{ij}]$ on $B_{10}$
and that $\orbEuler(B^\mu,R^\mu)$ is off this
quantity by some ``error term'' $M$
\[\orbEuler(B^\mu,R^\mu) = 6\sum_{i<j} \mu_i\mu_j  - 9 + M\]
with $M$ being the sum over all $i<j$ with $\mu_i+\mu_j>1$ of the expression
\begin{align*}
 - (\mu_i+\mu_j) + (\mu_i+\mu_j -1)^2 - 2 -
(1-\mu_i-\mu_j)(3-2(\mu_l+\mu_m+\mu_q)) + 3,
\end{align*}
where $\{l,m,q\}$ is again a complement of $\{i,j\}$. Evaluating this yields
\[M = 3\sum_{\substack{i<j:\\ \mu_i+\mu_j>1 }} (1-\mu_i-\mu_j)^2,\]
which altogether implies $3\orbEuler(B^\mu,R^\mu) = (K_{B^\mu} + R^\mu)^2$.
\end{proof}
\par
The starting point for the proof of
Corollary~\ref{cor:compute-otherinvariant} is that 
for $\mu = \mu(1)$ satisfying \eqref{eq:SigmaINT}, the space $B^{\mu,\nc}$ is
larger than any of its Galois conjugates $B^{\mu(k),\nc}$ (with $\sum_i\mu_i(k)
= 2$) in the sense that $B^{\mu,\nc}$ is an intermediate space in the chain of
blowdowns $B_{10} \to \dots\to B^{\mu(k),\nc}$. More precisely,
\begin{lemma}\label{le:B1Bk-blowdown}
 Suppose $\mu \in (\QQ\cap(0,1))^5$ satisfies \eqref{eq:SigmaINT}, and let
$\mu(k)$ be a Galois conjugate of signature $(1,2)$.
\begin{enumerate}[a)]
 \item For any
pair $\{i,j\}\subset\{1,\dots,5\}$, $\mu_i+\mu_j>1$ implies
$\mu_i(k)+\mu_j(k)>1$.
\par
In particular, there is a blowdown map 
$b: B^{\mu,\nc} \to B^{\mu(k),\nc}$
contracting some
of the boundary divisors $L_{ij}$ with self-intersection $(-1)$.
\item $b^*(K_{B^{\mu(k),\nc}} + R^{\mu(k)}) = K^\orb - D_k$.
\end{enumerate}
\end{lemma}
\begin{proof}
a) One checks all the finitely many cases. b) By the proof of Lemma
\ref{le:B10Bkcherncomparison}, we find that the coefficient of
$b^*(K_{B^{\mu(k),\nc}} + R^{\mu(k)})$ in front of $[L_{ij}]$ is 
\[3 - 2(\mu_i(k) + \mu_j(k)) = 1 - \tfrac{n_{ij}^k}{\kappa_{ij}}.\]
\end{proof}
\par
This property is not true for arbitrary $\mu$ with $\sum_i\mu_i=2$ and we
suspect it to reflect a metric contraction property of the non-uniformizing
period maps relative to the uniformizing one, analogously to the case of
Teichm\"uller curves.
\par
\begin{proof}[Proof of Corollary \ref{cor:compute-otherinvariant}]
First, note that
\[\frac{\orbEuler(\cB^k,\cR^k)}{\orbEuler(\cB^1,\cR^1)} =
\frac{\orbEuler(B^{\mu(k),\nc},\cR^{\mu(k)})}{\orbEuler(B^{\mu,\nc},R^\mu)}\]
using the multiplicative behavior of $\orbEuler$ under pullback. The first
equality reduces thus to combining Theorem~\ref{thm:log-ballquotients} with
Lemma \ref{le:B1Bk-blowdown} b).
\par
For the second and third equality, let again
$\pi:Y \to \cB^1$ denote the
extension of the canonical projection $B^u = \BB^2/\Gamma' \to B^1/\Sigma$ as
in Section \ref{sec:SigmaINT-computations of chern1-koker}. Then
by \eqref{eqn:ses for tautilde} and using $\Chern_1(\cE^{1,0}) =
-\Chern_1(\cE^{0,1})$, we obtain
\begin{align*}
\pi^*(K^\orb - D_k) &= \Chern_1(\omega_Y) - \Chern_1(\koker \tilde \tau)\\
		    &= -\Chern_1(\SheafHom(\cE^{1,0},\cE^{0,1}))\\
		    &= - \Chern_1((\cE^{1,0})^*\tensor \cE^{0,1}) =
3\Chern_1(\cE^{1,0})
\end{align*}
\end{proof}

\bibliographystyle{halpha}
\bibliography{bq_biblio}

\begin{thebibliography}{CMSP03}

\bibitem[BM10]{bouwmoel}
I.~Bouw and M.~M{\"o}ller.
\newblock Teichm\"uller curves, triangle groups, and {L}yapunov exponents.
\newblock {\em Ann. of Math. (2)}, 172(1):139--185, 2010.

\bibitem[Bou01]{bouwprank}
I.~Bouw.
\newblock The {$p$}-rank of ramified covers of curves.
\newblock {\em Compositio Math.}, 126(3):295--322, 2001.

\bibitem[Bou05]{bouwhabil}
I.~Bouw.
\newblock {P}seudo-elliptic bundles, deformation data and the reduction of
  {G}alois covers, 2005.

\bibitem[CKS86]{catkapsch}
Eduardo Cattani, Aroldo Kaplan, and Wilfried Schmid.
\newblock Degeneration of {H}odge structures.
\newblock {\em Ann. of Math. (2)}, 123(3):457--535, 1986.

\bibitem[CMSP03]{CMSP}
J.~Carlson, S.~M{\"u}ller-Stach, and C.~Peters.
\newblock {\em Period mappings and period domains}, volume~85 of {\em Cambridge
  Studies in Advanced Mathematics}.
\newblock Cambridge University Press, Cambridge, 2003.

\bibitem[CW93]{CoWoMod}
P.~B. Cohen and J.~Wolfart.
\newblock Fonctions hyperg\'eom\'etriques en plusieurs variables et espaces des
  modules de vari\'et\'es ab\'eliennes.
\newblock {\em Ann. Sci. \'Ecole Norm. Sup. (4)}, 26(6):665--690, 1993.

\bibitem[Del70]{delequadiff}
P.~Deligne.
\newblock {\em \'{E}quations diff\'erentielles \`a points singuliers
  r\'eguliers}.
\newblock Lecture Notes in Mathematics, Vol. 163. Springer-Verlag, Berlin,
  1970.

\bibitem[Del87]{delfinitude}
P.~Deligne.
\newblock Un th\'eor\`eme de finitude pour la monodromie.
\newblock In {\em Discrete groups in geometry and analysis ({N}ew {H}aven,
  {C}onn., 1984)}, volume~67 of {\em Progr. Math.}, pages 1--19. Birkh\"auser
  Boston, Boston, MA, 1987.

\bibitem[DM86]{DeligneMostow86}
P.~Deligne and G.~D. Mostow.
\newblock Monodromy of hypergeometric functions and nonlattice integral
  monodromy.
\newblock {\em Inst. Hautes \'Etudes Sci. Publ. Math.}, (63):5--89, 1986.

\bibitem[DM93]{delcommen}
P.~Deligne and G.~D. Mostow.
\newblock {\em Commensurabilities among lattices in {${\rm PU}(1,n)$}}, volume
  132 of {\em Annals of Mathematics Studies}.
\newblock Princeton University Press, Princeton, NJ, 1993.

\bibitem[DPP11]{derauxppcensus}
M.~Deraux, J.~Parker, and J.~Paupert.
\newblock Census of the complex hyperbolic sporadic triangle groups.
\newblock {\em Exp. Math.}, 20(4):467--486, 2011.

\bibitem[DPP14]{derauxppnew}
M.~Deraux, J.~Parker, and J.~Paupert.
\newblock {New non-arithmetic complex hyperbolic lattices}, 2014,
  arXiv:math.GT/1401.0308.

\bibitem[EKZ11]{cyclicekz}
Alex Eskin, Maxim Kontsevich, and Anton Zorich.
\newblock Lyapunov spectrum of square-tiled cyclic covers.
\newblock {\em J. Mod. Dyn.}, 5(2):319--353, 2011.

\bibitem[EKZ14]{ekz}
Alex Eskin, Maxim Kontsevich, and Anton Zorich.
\newblock Sum of {L}yapunov exponents of the {H}odge bundle with respect to the
  {T}eichm\"uller geodesic flow.
\newblock {\em Publ. Math. Inst. Hautes \'Etudes Sci.}, 120:207--333, 2014.

\bibitem[EV02]{EsnaultVieweg02}
H{\'e}l{\`e}ne Esnault and Eckart Viehweg.
\newblock Chern classes of {G}auss-{M}anin bundles of weight 1 vanish.
\newblock {\em $K$-Theory}, 26(3):287--305, 2002.

\bibitem[FMZ14a]{FMZinvariant}
Giovanni Forni, Carlos Matheus, and Anton Zorich.
\newblock Lyapunov spectrum of invariant subbundles of the {H}odge bundle.
\newblock {\em Ergodic Theory Dynam. Systems}, 34(2):353--408, 2014.

\bibitem[FMZ14b]{FMZzero}
Giovanni Forni, Carlos Matheus, and Anton Zorich.
\newblock Zero {L}yapunov exponents of the {H}odge bundle.
\newblock {\em Comment. Math. Helv.}, 89(2):489--535, 2014.

\bibitem[For02]{forni02}
G.~Forni.
\newblock Deviation of ergodic averages for area-preserving flows on surfaces
  of higher genus.
\newblock {\em Ann. of Math. (2)}, 155(1):1--103, 2002.

\bibitem[Ful84]{FultonIT}
W.~Fulton.
\newblock {\em Intersection theory}, volume~2 of {\em Ergebnisse der Mathematik
  und ihrer Grenzgebiete (3) [Results in Mathematics and Related Areas (3)]}.
\newblock Springer-Verlag, Berlin, 1984.

\bibitem[GH14]{GrHub}
J.~Grivaux and P.~Hubert.
\newblock Les exposants de {L}iapounoff du flot de {T}eichm{\"u}ller [d'après
  {E}skin-{K}ontsevich-{Z}orich]).
\newblock {\em Ast\'erisque}, (353):Exp. No. 1060, 2014.
\newblock S{\'e}minaire Bourbaki. Vol. 2011/2012. Expos{\'e}s 1059--1065.

\bibitem[Gol99]{Goldman99}
W.~M. Goldman.
\newblock {\em Complex hyperbolic geometry}.
\newblock Oxford Mathematical Monographs. The Clarendon Press Oxford University
  Press, New York, 1999.
\newblock Oxford Science Publications.

\bibitem[Gri70]{griffglobal}
Phillip~A. Griffiths.
\newblock Periods of integrals on algebraic manifolds. {III}. {S}ome global
  differential-geometric properties of the period mapping.
\newblock {\em Inst. Hautes \'Etudes Sci. Publ. Math.}, (38):125--180, 1970.

\bibitem[Has03]{HasWeighted}
B.~Hassett.
\newblock Moduli spaces of weighted pointed stable curves.
\newblock {\em Adv. Math.}, 173(2):316--352, 2003.

\bibitem[Hel94]{HelgGeomAna}
S.~Helgason.
\newblock {\em Geometric analysis on symmetric spaces}, volume~39 of {\em
  Mathematical Surveys and Monographs}.
\newblock American Mathematical Society, Providence, RI, 1994.

\bibitem[Hun00]{Hunt00}
B.~Hunt.
\newblock Higher-dimensional ball quotients and the invariant quintic.
\newblock {\em Transform. Groups}, 5(2):121--156, 2000.

\bibitem[Huy05]{Huybrechts05}
D.~Huybrechts.
\newblock {\em Complex geometry: An introduction}.
\newblock Universitext. Springer-Verlag, Berlin, 2005.

\bibitem[KN96]{KobayashiNomizu2}
S.~Kobayashi and K.~Nomizu.
\newblock {\em Foundations of differential geometry. {V}ol. {II}}.
\newblock Wiley Classics Library. John Wiley \& Sons Inc., New York, 1996.
\newblock Reprint of the 1969 original, A Wiley-Interscience Publication.

\bibitem[Kon97]{kontsevich}
M.~Kontsevich.
\newblock Lyapunov exponents and {H}odge theory.
\newblock In {\em The mathematical beauty of physics (Saclay, 1996)}, volume~24
  of {\em Adv. Ser. Math. Phys.}, pages 318--332. World Sci. Publishing, River
  Edge, NJ, 1997.

\bibitem[KZ97]{kontsevichzorich}
M.~Kontsevich. and A.~Zorich.
\newblock {Lyapunov exponents and Hodge theory}, 1997, arXiv:hep-th/9701164.

\bibitem[Lan02]{LangAlgebra}
Serge Lang.
\newblock {\em Algebra}, volume 211 of {\em Graduate Texts in Mathematics}.
\newblock Springer-Verlag, New York, third edition, 2002.

\bibitem[Lan03]{langerorbifold}
Adrian Langer.
\newblock Logarithmic orbifold {E}uler numbers of surfaces with applications.
\newblock {\em Proc. London Math. Soc. (3)}, 86(2):358--396, 2003.

\bibitem[Mar91]{margulisbook}
G.~A. Margulis.
\newblock {\em Discrete subgroups of semisimple {L}ie groups}, volume~17 of
  {\em Ergebnisse der Mathematik und ihrer Grenzgebiete (3) [Results in
  Mathematics and Related Areas (3)]}.
\newblock Springer-Verlag, Berlin, 1991.

\bibitem[McM13]{mcmullenhodge}
Curtis~T. McMullen.
\newblock Braid groups and {H}odge theory.
\newblock {\em Math. Ann.}, 355(3):893--946, 2013.

\bibitem[McR11]{mcreynolds}
D.~B. McReynolds.
\newblock {Arithmetic lattices in $SU(n,1)$}, 2011, preprint, available on the
  author's web page.

\bibitem[M{\"o}l11]{moellerST}
M.~M{\"o}ller.
\newblock Shimura and {T}eichm\"uller curves.
\newblock {\em J. Mod. Dyn.}, 5(1):1--32, 2011.

\bibitem[Mos80]{mostowremarkable}
G.~D. Mostow.
\newblock On a remarkable class of polyhedra in complex hyperbolic space.
\newblock {\em Pacific J. Math.}, 86(1):171--276, 1980.

\bibitem[Mos86]{mostowhalfint}
G.~D. Mostow.
\newblock Generalized {P}icard lattices arising from half-integral conditions.
\newblock {\em Inst. Hautes \'Etudes Sci. Publ. Math.}, (63):91--106, 1986.

\bibitem[Mos88]{mostowdisc}
G.~D. Mostow.
\newblock On discontinuous action of monodromy groups on the complex
  {$n$}-ball.
\newblock {\em J. Amer. Math. Soc.}, 1(3):555--586, 1988.

\bibitem[MVZ12]{MVZ}
Martin M{\"o}ller, Eckart Viehweg, and Kang Zuo.
\newblock Stability of {H}odge bundles and a numerical characterization of
  {S}himura varieties.
\newblock {\em J. Differential Geom.}, 92(1):71--151, 2012.

\bibitem[Par09]{parkersurvey}
J.R. Parker.
\newblock Complex hyperbolic lattices.
\newblock In {\em Discrete groups and geometric structures}, volume 501 of {\em
  Contemp. Math.}, pages 1--42. Amer. Math. Soc., Providence, RI, 2009.

\bibitem[Pau10]{paupertIII}
J.~Paupert.
\newblock Unfaithful complex hyperbolic triangle groups. {III}. {A}rithmeticity
  and commensurability.
\newblock {\em Pacific J. Math.}, 245(2):359--372, 2010.

\bibitem[PS03]{peterssteenmonodromy}
C.~A.~M. Peters and J.~H.~M. Steenbrink.
\newblock Monodromy of variations of {H}odge structure.
\newblock {\em Acta Appl. Math.}, 75(1-3):183--194, 2003.
\newblock Monodromy and differential equations (Moscow, 2001).

\bibitem[Roy80]{Roy}
H.~L. Royden.
\newblock The {A}hlfors-{S}chwarz lemma in several complex variables.
\newblock {\em Comment. Math. Helv.}, 55(4):547--558, 1980.

\bibitem[Rud87]{RudinComplex}
Walter Rudin.
\newblock {\em Real and complex analysis}.
\newblock McGraw-Hill Book Co., New York, third edition, 1987.

\bibitem[Rue79]{Ruelle79}
David Ruelle.
\newblock Ergodic theory of differentiable dynamical systems.
\newblock {\em Inst. Hautes {\'E}tudes Sci. Publ. Math.}, (50):27--58, 1979.

\bibitem[Sau90]{sauter}
J.K. Sauter, Jr.
\newblock Isomorphisms among monodromy groups and applications to lattices in
  {${\rm PU}(1,2)$}.
\newblock {\em Pacific J. Math.}, 146(2):331--384, 1990.

\bibitem[Sch73]{schmid73}
W.~Schmid.
\newblock Variation of {H}odge structure: the singularities of the period
  mapping.
\newblock {\em Invent. Math.}, 22:211--319, 1973.

\bibitem[Shi04]{Shiga}
H.~Shiga.
\newblock On holomorphic mappings of complex manifolds with ball model.
\newblock {\em J. Math. Soc. Japan}, 56(4):1087--1107, 2004.

\bibitem[Thu98]{thurstonshapes}
W.~P. Thurston.
\newblock Shapes of polyhedra and triangulations of the sphere.
\newblock In {\em The {E}pstein birthday schrift}, volume~1 of {\em Geom.
  Topol. Monogr.}, pages 511--549 (electronic). Geom. Topol. Publ., Coventry,
  1998.

\bibitem[VZ07]{VZShimhigher}
E.~Viehweg and K.~Zuo.
\newblock Arakelov inequalities and the uniformization of certain rigid
  {S}himura varieties.
\newblock {\em J. Differential Geom.}, 77(2):291--352, 2007.

\bibitem[Wri12]{Wright11}
Alex Wright.
\newblock Schwarz triangle mappings and {T}eichm\"uller curves: abelian
  square-tiled surfaces.
\newblock {\em J. Mod. Dyn.}, 6(3):405--426, 2012.

\bibitem[Yos87]{YoshidaFuchs87}
Masaaki Yoshida.
\newblock {\em Fuchsian differential equations}.
\newblock Aspects of Mathematics, E11. Friedr. Vieweg \& Sohn, Braunschweig,
  1987.
\newblock With special emphasis on the Gauss-Schwarz theory.

\bibitem[Zor06]{zorich06}
A.~Zorich.
\newblock Flat surfaces.
\newblock In {\em Frontiers in Number Theory, Physics and Geometry. Volume 1:
  On random matrices, zeta functions and dynamical systems}, pages 439--586.
  Springer-Verlag, Berlin, 2006.

\end{thebibliography}
\end{document}